\documentclass[a4paper]{amsart}
\usepackage{mathrsfs}
\usepackage{amsmath,amssymb,bbm}
\usepackage[dvips]{graphicx}
\usepackage{psfrag}

\swapnumbers
\newtheorem{theorem}{Theorem}
\newtheorem{lemma}{Lemma}[section]
\newtheorem{propos}{Proposition}[section]

\newtheorem{definition}{Definition}[section]
\newtheorem{remark}{Remark}[section]

\newcommand{\N}{\mathbb{N}}        % per i naturali
\newcommand{\R}{\mathbb{R}}        % per i reali
\newcommand{\RN}{\mathbb{R}^N}
\newcommand{\RdN}{\mathbb{R}^{2N}}
\newcommand{\RDN}{\mathbb{R}^{2N}}
\newcommand{\Omegac}{\RN\setminus \Omega}
\newcommand{\vf}{\varphi}
\newcommand{\eps}{\varepsilon}

\newcommand{\Xzs}{\mathcal{X}^{s}_{0}}
\newcommand{\spazio}{\mathcal X^s_0}

\newcommand{\Leb}{{\mathscr L}}

\newcommand{\dd}{\mathfrak d}

\DeclareMathOperator{\deriv}{d}

\newcommand{\dx}{\deriv\!x}
\newcommand{\dy}{\deriv\!y}

\newcommand{\drho}{\deriv\!\rho}

\newcommand{\calX}{{\mathcal X}}

\DeclareMathOperator{\pv}{p.\!v.}
\newcommand{\Lap}{(-\Delta)} %Laplacian

\newcommand{\Ds}{\Lap^{s}}
\newcommand{\Dsm}{\Lap^{s/2}}

\newcommand{\Lsm}{\Lap^{s/2}}

\renewcommand{\d}{\mathrm{d}}

\newcommand{\dual}[2]{\left\langle #1 ,  #2 \right\rangle}

\newcommand{\lO}{\lambda^a}
\newcommand{\lOe}{\lambda^b}
\newcommand{\lj}{\lambda^j}
\newcommand{\Nkm}{N_{k,m}^a}
\newcommand{\Nkme}{N_{k,m}^b}
\newcommand{\eHs}{e_s}
\newcommand{\diHs}{\mathrm{dist}_s}
\newcommand{\dHs}{d_{H,s}}
\newcommand{\Hs}{H^s(\RN)}
\newcommand{\T}{T_a}
\newcommand{\Te}{T_b}
\newcommand{\Oa}{\Omega_a}
\newcommand{\Ob}{\Omega_b}
\newcommand{\Oj}{\Omega_j}

\newcommand{\eO}{e^a}
\newcommand{\eOe}{e^b}
\newcommand{\uj}{u^j}
\newcommand{\Pe}{P_{D \to \Omega_b}}
\newcommand{\BDL}{\mathcal L(\Nkm,\Xzs(D))}

\newcommand{\eX}{e_X}
\newcommand{\diX}{\mathrm{dist}_X}
\newcommand{\vep}{\varepsilon}
\DeclareMathOperator{\supp}{supp}

\usepackage{color}
\usepackage[normalem]{ulem} % "\sout" for striking out in text mode
\newenvironment{giuliorev}{\color{cyan}}{\color{black}}
\newcommand{\III}{\begin{giuliorev}}
\newcommand{\EEE}{\end{giuliorev}}

\newenvironment{gororev}{\color{red}}{\color{black}}
\newcommand{\goro}{\begin{gororev}}
\newcommand{\egoro}{\end{gororev}}

\usepackage{version}	% required for `\comment' (yatex added)
\begin{document}

\title[]{Quantitative estimates on localized finite differences for the fractional Poisson problem, and applications to regularity and spectral stability}

\author{Goro Akagi}
\address{GA: Mathematical Institute, Tohoku University, Japan ; 
 Helmholtz Zentrum M\"unchen, Institut f\"ur Computational Biology, Ingolst\"adter Landstra\ss e 1, 85764 Neunerberg, Germany ; 
 Technische Universit\"at M\"unchen, Zentrum Mathematik, Boltzmannstra\ss e 3, D-85748 Garching bei M\"unchen, Germany}
\email{akagi@m.tohoku.ac.jp}
\urladdr{}
\author{Giulio Schimperna}
\address{GS: Dipartimento di Matematica ``F. Casorati'', via Ferrata 1, I-27100 Pavia, Italy}
\email{giusch04@unipv.it}
\urladdr{http://www-dimat.unipv.it/giulio}
\author{Antonio Segatti}
\address{AS: Dipartimento di Matematica ``F. Casorati'', via Ferrata 1, I-27100 Pavia, Italy}
\email{antonio.segatti@unipv.it}
\urladdr{http://www-dimat.unipv.it/segatti}
\author{Laura V. Spinolo}
\address{LVS: IMATI-CNR ``E. Magenes", via Ferrata 1, I-27100 Pavia, Italy}
\email{spinolo@imati.cnr.it}
\urladdr{http://arturo.imati.cnr.it/spinolo/}

%\author{Goro Akagi, Giulio Schimperna, Antonio Segatti, Laura Spinolo}
\date{}

\begin{abstract}
We establish new quantitative estimates for localized finite differences of 
solutions to the Poisson problem for the fractional Laplace operator with homogeneous
Dirichlet conditions of solid type settled in bounded domains satisfying the 
Lipschitz cone regularity condition. We then apply these estimates to obtain 
(i)~regularity results for solutions of fractional Poisson problems in Besov spaces; 
(ii)~quantitative stability estimates for solutions of fractional Poisson problems
with respect to domain perturbations; (iii)~quantitative stability estimates for eigenvalues 
and eigenfunctions of fractional Laplace operators  with respect to domain perturbations. 
\end{abstract}
\maketitle

%%%%%%%%%%%%%%%%%%%%%%%%%%%%%%%%%%%%%%%%%%%%%%%%%%%%%%%%%%%%%%%%%%%%%%%%%%%%%%%%%%%%%%%%%%%%%
%
% INTRODUCTION
%
%%%%%%%%%%%%%%%%%%%%%%%%%%%%%%%%%%%%%%%%%%%%%%%%%%%%%%%%%%%%%%%%%%%%%%%%%%%%%%%%%%%%%%%%%%%%%

\section{Introduction}

We focus on the Poisson problem for the fractional Laplacian operator, namely 
on the system
\begin{equation}
\label{eq:DP}
\begin{cases}
\Ds u = f \ \hbox{ in } \Omega\\
u = 0 \ \hbox{ in } \Omegac.
\end{cases}
\end{equation}
In the above expression, $\Omega \subset \R^N$ is an open and bounded set, $s$ is an index belonging to the interval $(0, 1)$,
and the regularity of the function $f$ is discussed below. The symbol $\Ds$ denotes the $s$-fractional Laplacian 
operator: in~\S~\ref{subsec:frac} we provide both the definition of $\Ds$ and the rigorous (distributional) formulation 
of problem~\eqref{eq:DP} by following the approach provided, e.g., in \cite{ASS,serva-valdi11}. Here, we just mention that we are 
concerned with solutions $u$ belonging to the space 
\begin{equation}
\label{e:accaesse}
    \spazio (\Omega) : = \big\{ u \in H^s (\R^N): \; u \equiv 0 \; 
    \text{in $\R^N \setminus \Omega$}   \big\}   
\end{equation}
and that in the following we denote by $\| u \|_s$ the so-called \emph{Gagliardo semi-norm} 
of $u$, which is again defined in~\S~\ref{subsec:frac}. Note, furthermore, that the so-called \emph{solid} boundary conditions
at the second line of~\eqref{eq:DP} are consistent with the fact that the fractional Laplacian is a \emph{nonlocal} operator. 
Also, the fractional Laplacian operator coupled with the solid boundary conditions is usually termed \emph{restricted fractional Laplacian}.   

Problem~\eqref{eq:DP} can be addressed by relying on variational techniques. For instance, a straightforward application of  the Lax-Milgram lemma gives existence and uniqueness of a solution
$u\in \spazio(\Omega)$, provided that $f$ belongs to the dual space $\spazio(\Omega)'$. 
It is therefore natural to investigate 
whether or not the condition $f\in L^2$ implies additional regularity of $u$. This is the main goal of the present paper.

In order to state our results in a precise way, we start by introducing
some further  notation. Let $h \in \R^N$ be a vector, $|h|<1$. We fix a function $u: \R^N \to \R$ 
and  a smooth cut-off function $\phi: \R^N \to \R$, and we define the functions $u_h$ and $T_h u$ by setting 
\begin{equation} 
\label{eq:local_transl}
   u_h(x):=u (x + h ), \qquad 
  (T_h u)(x) :=\phi(x) u_h(x) + \big[1-\phi(x)\big] u(x), \quad \text{for every $x\in \RN$}. 
\end{equation} 
Note that the quantity $T_h u -u = \phi [u_h -u]$ can be viewed as a localized version of a finite 
difference. In the following we will mostly focus on the case when the domain $\Omega$ satisfies a so-called $(\rho, \theta)$-Lipschitz cone condition. 
The precise definition is provided in~\S~\ref{s:Lip} below, here we just mention that, very loosely speaking, 
this condition is a sort of quantified Lipschitz condition imposed on the boundary $\partial \Omega$. Also, 
$\rho \in ]0, + \infty[$ and $\theta \in ]0, \pi/2[$ are regularity parameters: 
the bigger the $\rho$, the more regular the domain, the larger the $\theta$, the more regular the domain. 

Our main result establishes a precise quantitative control on $T_hu -u$ in the case 
when $u$ is a weak solution of~\eqref{eq:DP}. 
\begin{theorem}\label{t:main}
 Let $\Omega\subset \RN$ be a bounded, open set and $f \in L^2 (\R^N)$.
 Assume that $\phi$ is a smooth cut-off function, namely
 \begin{equation}
 \label{e:hypphi}
    \phi \in W^{1, \infty} (\R^N), \quad 
    0 \leq \phi (x) \leq 1  
 \end{equation}
 and 
 \begin{equation}
 \label{e:supporto}
          \mathrm{supp} \, \phi \subseteq B_1(0).  
 \end{equation} 
 Assume also that $u \in \spazio (\Omega)$ is a weak solution of~\eqref{eq:DP} and that the product 
 \begin{equation}
 \label{e:hyprodotto} 
  \phi u_h  \in \spazio(\Omega).
 \end{equation}
 Then there is a constant $C$, which only depends on 
 $N, \; s, \; \mathrm{Lip} \, \phi$ and $\mathrm{diam} \, \Omega$, such that
 \begin{equation}
    \| T_h u - u \|^2_{L^2(\RN)}
     \le \| u_h - u \|^2_{L^2(\RN)}
     \le C \vert h \vert^{2s} \| f \|^2_{\spazio(\Omega)'}.
  \label{eq:key1}
 \end{equation}
 Moreover, if $\Omega$ satisfies a $(\rho, \theta)$-Lipschitz cone condition 
 for some $\rho \in ]0, + \infty[$, $\theta \in ]0, \pi/2[$ and 
 $$
     |h| \leq \frac{\rho \sin \theta}{4}, 
 $$
 then there is a constant $\tilde C$, which only depends on 
 $N, \; s, \; \mathrm{Lip} \, \phi$, $\mathrm{diam} \, \Omega$, $\rho$ and $\theta$, such that 
 \begin{equation}
   \| T_h u - u\|^2_{s}\le \tilde C \vert h\vert^{ s} \| f \|_{L^2(\R^N)}
   \| f\|_{H^{-s}(\R^N)}.
    \label{eq:key2}
 \end{equation}
\end{theorem}
The following remarks are in order:
\begin{itemize}
\item in the statement of the theorem, $B_1(0)$ is the open unit
ball centered at $0$, 
$\mathrm{Lip} \, \phi$
denotes the Lipschitz constant of $\phi$,
and $\mathrm{diam} \, \Omega$ is the diameter of $\Omega$, namely
\begin{equation}
\label{e:diametro}
    \mathrm{diam} \, \Omega : = \sup_{x, y \in \Omega} |x-y|. 
\end{equation}
\item A relevant feature of Theorem~\ref{t:main} is that by following the proof one can reconstruct the value of the constants $C$ and $\tilde C$.
\item The most interesting estimate is~\eqref{eq:key2}, whereas establishing~\eqref{eq:key1} is quite easy. Also, note 
that~\eqref{eq:key1} holds for any open and bounded set $\Omega$, whereas to obtain~\eqref{eq:key2} we have to 
assume the Lipschitz cone condition. Indeed, in the general case we can only 
establish a weaker version of~\eqref{eq:key2}, see Lemma~\ref{l:bootstrap1} in~\S~\ref{s:localized} below. 
\item Let $D$ be a sufficiently large ball containing both $\Omega$ and $\Omega + h$, for every  $|h|<1$. By 
recalling~\eqref{eq:DP}, we infer that the solution $u$ is only affected by the values attained by $f$ 
on $D$. Hence, $u$ does not change if we replace $f$ with its truncation to $0$ outside $D$. This implies
that in the right hand side of~\eqref{eq:key2} one could for instance use $\| f \|_{L^2 (D)}$, instead of 
$\| f \|_{L^2 (\R^N)}$. However, to simplify the notation here and in the following we always compute the 
norms on the whole $\R^N$. 
%
%The reason why to obtain~\eqref{eq:key2} we have to require some regularity hypotheses  is, very loosely speaking, the following.
%To establish~\eqref{eq:key2} we rely on a sort of bootstrap argument which uses suitable regularity properties of the solution $u$ as a
%function defined on the whole $\R^N$, and to establish these properties we need some regularity assumptions on $\Omega$.  
%
\end{itemize}
In the following, we discuss some possible applications of Theorem~\ref{t:main}. 
First, we formulate Theorem~\ref{th:domain_perturb}, which provides precise quantitative
estimates on how the solution of the Poisson problem~\eqref{eq:DP} depends on the domain $\Omega$. 
Note that in the statement of Theorem~\ref{th:domain_perturb} 
the quantity $ \dd(\Omega_b,\Omega_a)$ is a way of measuring  the ``distance" between the sets $\Omega_a$ and $\Omega_b$: 
the precise definition is provided in~\S~\ref{subsec:haus}. 
\begin{theorem}[Domain perturbations]\label{th:domain_perturb}
 Let $\Omega_a, \Omega_b\subset\RN$ be bounded open sets
 contained in a sufficiently large open ball
 $D$ of $\RN$. Let us assume that $\Omega_a$ satisfies the 
 $(\rho,\theta)$-Lipschitz cone condition.
 Let $f\in L^2(\R^N)$ and let $u_a\in \Xzs(\Omega_a)$ 
 and $u_b\in \Xzs(\Omega_b)$
 denote the weak solutions to \eqref{eq:DP} in $\Omega=\Omega_a$ 
 and $\Omega=\Omega_b$, respectively. There is a positive constant $C$, which only  depends
 on $N$, $s$, $\rho$, $\theta$ and $\mathrm{diam} \, D$, such that, if 
 \begin{equation}\label{eq:hypo_dist}
   \dd(\Omega_b,\Omega_a) < \frac{\rho\sin\theta}{2},
 \end{equation}
 then 
 \begin{equation} \label{e:goal2}
   \| u_a - u_b\|_{s}
   \le C \| f \|_{L^2(\R^N)}^{1/2} \| f \|_{H^{-s}(\R^N)}^{1/2} 
    \dd(\Omega_b,\Omega_a)^{s/2}.
 \end{equation}
\end{theorem}
We point out that, as in the case of Theorem~\ref{t:main}, by following the proof of Theorem~\ref{th:domain_perturb}
one can reconstruct the precise  value of the constant $C$ in~\eqref{e:goal2}. Moreover, we observe that
the proof of Theorem~\ref{th:domain_perturb} combines Theorem~\ref{t:main} with a localization argument due to Savar\'e and Schimperna~\cite{SS02}. 
Finally, in the statement of Theorem~\ref{th:domain_perturb} we impose a regularity assumption on $\Omega_a$ only, while $\Omega_b$ may 
be any open and bounded domain satisfying~\eqref{eq:hypo_dist}. This lack of symmetry is consistent with the fact that the quantity 
$  \dd(\Omega_b,\Omega_a)$ is not symmetric in $\Omega_a$ and $\Omega_b$, namely in general $  \dd(\Omega_b,\Omega_a) \neq   \dd(\Omega_a,\Omega_b)$.

In the case of the standard Laplacian, both the optimal regularity 
properties of the solution $u$ to the Poisson
problem corresponding to \eqref{eq:DP} and the relation between the regularity of $u$ and
the regularity
properties of $\Omega$ and of $f$ are well known. In particular, in the case 
when $\Omega$ is a Lipschitz domain, the
results in~\cite{SS02} state that, if $f\in L^2$, then $u$ belongs
to the {\it Besov}\/ space $u \in B^{3/2}_{2,\infty}$
(see \S~\ref{subsec:besov} below for the definition of $B^{3/2}_{2,\infty}$ ). In particular, in general one cannot achieve the higher regularity $u \in H^{3/2}$, as one can see by considering the one-dimensional example 
$u(x)=(1-x^2)^+$, which solves \eqref{eq:DP} with $s=1$ and $f=2 \chi_{(-1,1)}$  (cf.~also \cite[Rem.~2.4]{SS02}).
In this case the regularity $u \in H^{3/2}(\R)$ is not attained
because $u$ has jump discontinuities at $\pm1$. 

On the other hand, the regularity theory for the fractional Poisson problem
for~\eqref{eq:DP} is far less established. 
In this paper we are interested in possible extensions of the following result  
\cite[Prop.~1.4 (ii)-(iii)]{ROS}:
\begin{propos} \label{prop:ROS}
 Let $\Omega \subset \RN$ be a bounded $C^{1,1}$-domain.
 If $s \in(0, \frac N 4 ) \cap (0,1)$, then the solution $u$ to
 \eqref{eq:DP} satisfies 
 \begin{equation}\label{ROS11}
   \| u \|_{L^q(\Omega)} \le C \| f \|_{L^2 (\Omega) }
    \ \text{ for }\, q = \frac{2N}{N-4s}, 
 \end{equation}
 while for $s\in(\frac N 4,1) \cap (0,1)$ we have
 \begin{equation}\label{ROS12}
   \| u \|_{C^\alpha(\overline\Omega)} \le C \| f \|_{L^2 (\Omega) }
    \ \text{ for } \alpha = \min \left\{ s , 2s-\frac N 2 \right\}. 
 \end{equation}
 In both cases the constant $C>0$ only depends  on 
 $s$, $|\Omega|$, and $q$ (or $\alpha$).
\end{propos}
We now state our regularity result. Note that the main novelties of  Theorem~\ref{th:regularity} compared to Proposition~\ref{prop:ROS} are the following. First, we impose 
weaker regularity assumptions on the domain (we only require Lipschitz regularity). Second, 
we establish Sobolev and Besov-type regularity, more precisely we show that, if $f \in L^2$, then the solution $u$ belongs to the Besov space $B^{3s/2}_{2,\infty}(\RN)$ 
(we refer again to~\S~\ref{subsec:besov} for the precise definition).
\begin{theorem}\label{th:regularity}
 Let $\Omega$ be a bounded domain satisfying a $(\rho, \theta)$-Lipschitz cone condition for some 
 values ${\rho \in ]0, + \infty[}$ and $\theta \in ]0, \pi/2[$.  
 Assume $f\in L^2(\RN)$ and let $u$ be the weak solution of~\eqref{eq:DP}. Then
 \begin{equation}\label{eq:regularity}
   u \in B^{3s/2}_{2,\infty}(\RN).
 \end{equation}  
 Moreover, we have the explicit regularity estimate
 \begin{equation}\label{eq:2regularity}
      \| u \|_{B^{3s/2}_{2, \infty}(\R^N)}
        \leq  C(N, s, \mathrm{diam} \, \Omega, \rho, \theta)  
      \|f \|_{H^{-s} (\R^N) }^{1/2}
    \| f \|_{L^2(\RN)}^{1/2}. 
 \end{equation}  
\end{theorem}
We make the following remarks:
\begin{itemize}
\item note that~\eqref{eq:2regularity} implies, in particular, that $u\in H^{\frac{3s}2-\epsilon}(\R^N)$
for any $\epsilon>0$. The optimality of~\eqref{eq:2regularity} is, however, unclear. Indeed, in~\S~\ref{ss:exantonio} we discuss an explicit example
where the solution has stronger regularity. However, we do not know whether or not this is a general fact.
\item By proceeding as in \cite[Corollary~3]{SS02}, one can see that the above 
result extends to the case when $f$ belongs to the interpolation space 
$B_{2,1}^{-s/2}=(L^2,H^{-s})_{1/2,1}$. In particular, in 
that case, estimate \eqref{eq:2regularity} is replaced by
\begin{equation}\label{eq:3regularity}
     \| u \|_{B^{3s/2}_{2, \infty}(\R^N)}
       \leq  C(N, s, \mathrm{diam} \, \Omega, \rho, \theta)  
     \|f \|_{B^{-s/2}_{2,1} (\R^N) }.
\end{equation}  
\end{itemize} 
We conclude by discussing some new spectral stability estimate for the Poisson problem~\eqref{eq:DP}. 
To this aim, we first introduce some notation. 
We say that $(u, \lambda)$ is an eigencouple for the operator $(-\Delta)^s$ in $\Omega$ if 
the eigenfunction $u \in \spazio (\Omega)$, $u \neq 0$, the eigenvalue $\lambda \in \R$  and 
the following holds
\begin{equation}\label{eigen:intro} 
 \begin{cases}
  \Ds u = \lambda u \ \hbox{ in }\, \Omega,\\ 
  u = 0 \ \hbox{ in }\, \Omegac.
 \end{cases}
\end{equation}
Owing to classical functional analytic results, the operator $(-\Delta)^s$
admits a diverging sequence of positive eigenvalues 
\begin{equation}
\label{e:sequence}
    0 < \lambda_1< \lambda_2 \leq \lambda_3 \leq \dots \leq \lambda_n \nearrow + \infty,
\end{equation}
provided $\Omega$ is an open and bounded set. We refer to~\S~\ref{ss:stpreli} for a more extended discussion and we point out that here and in 
the following we count each eigenvalue according to its multiplicity, namely according to the dimension of the associated eigenspace.

By combining Theorem~\ref{t:main} with an argument in~\cite{LMS2013} we establish the stability of the eigenvalues 
of the operator $(-\Delta)^s$ with respect to domain perturbations.   
\begin{theorem}[Spectral stability]\label{th:spectral_stability}
 Let $\Omega_a, \Omega_b \subset \R^N$ be two open, bounded sets satisfying the following conditions:
 \begin{itemize}
 \item[i)] $\Omega_a$ and $\Omega_b$ both satisfy a $(\rho, \theta)$-Lipschitz cone condition. 
 \item[ii)] $\Omega_a$ and $\Omega_b$ are both contained in some open ball $D \subset \R^N$.
 \item[iii)] There is a ball $B_r$ with radius $r$ such that $B_r \subseteq \Omega_a \cap \Omega_b$. 
 \end{itemize}
 Then for every $n \in \N$ there are constants $\nu>0$ and $C>0$, which only depend on  $N, \; s, \; \rho, \; \theta, \; \mathrm{diam} \; D, r$ and $n,$ such that, if 
 \begin{equation}
 \label{e:vicini}
           d_{H}^c(\Omega_a,\Omega_b) <  \nu,
 \end{equation}
 then
 \begin{equation}\label{eq:stima_eigen}
    \vert \lambda^a_{n}-\lambda^b_{n}\vert \le C 
     d_{H}^c(\Omega_a,\Omega_b)^s.
 \end{equation}
In the previous expression, $\lambda^a_{n}$ and $\lambda^b_{n}$ denote the $n$-th eigenvalue of $(- \Delta)^s$ in $\Omega_a$ and  $\Omega_b$, respectively.
\end{theorem}
Some remarks are here in order. First, in the statement of the above theorem $d_{H}^c(\Omega_a,\Omega_b)$ denotes the 
so-called \emph{complementary Hausdorff distance}, i.e.~the Hausdorff distance between the sets $\R^N \setminus \Omega_a$ 
and $\R^N \setminus \Omega_b$, see~\S~\ref{subsec:haus} for the precise definition. Second, the only reason why we assume hypothesis iii)
is because we need an upper bound on $\max \{ \lambda^a_{n}, \lambda^b_{n} \}$. Indeed, by combining condition iii) with the monotonicity
of eigenvalues with respect to set inclusion we obtain an upper bound which only depends on $N$, $s$, $r$ and $n$. Also, note that as in the
case of Theorems~\ref{t:main} and~\ref{th:domain_perturb} by following the proof of Theorem~\ref{th:spectral_stability} 
one can reconstruct the values of the constants $\nu$ and $C$.
\begin{remark}\label{rem:dicho}
{\rm
As one can infer from the statements of Theorems~\ref{th:domain_perturb} 
and~\ref{th:spectral_stability}, we have the following dichotomy:
\begin{itemize}
\item to control the difference between the \emph{eigenvalues}, we need to control
the complementary Hausdorff distance $d_H^{\, c} (\Omega_a, \Omega_b)$ and
we only use property i) in Definition~\ref{d:cone}. However, we have to require
that both $\Omega_a$ and $\Omega_b$ satisfy the Lipschitz 
regularity condition. 
\item On the other hand, to control the difference between 
the \emph{solutions} of the Poisson problem,
we need a control on a different type of set distance,
namely $\dd(\Omega_b, \Omega_a)$ (cf.~\eqref{d:front}). 
This forces us to use both properties i) and ii) in Definition~\ref{d:cone}. 
On the other hand, we only require Lipschitz regularity of $\Omega_a$
(whereas $\Omega_b$ can be any domain satisfying 
\eqref{e:dentro}). 
%
%This depends in principle on the fact
%that, rather than estimating the norm $\| u_a - u_b \|_s$ directly,
%it is possible, and more convenient, provide an estimate of
%$\| u_a - u_a^\eps \|_s$ and of $\| u_a - u_a^{-\eps} \|_s$,
%which is based on the sole regularity of $\Omega_a$.
%
\end{itemize}
This dichotomy is basically due to the fact that we a-priori have 
some additional information on the behavior of the eigenvalues: namely, 
they behave monotonically with respect to domain inclusion, cf.~\S~\ref{ss:stpreli},
whereas this property is not available for the solutions of the Poisson problem.
}
\end{remark}
We eventually discuss the stability of eigenfunctions for $\Ds$ under domain perturbations. 
We first state the following simple property, which holds for a very large class of 
domains:
\begin{propos}
\label{P:fund} Fix $s \in (0, 1)$ and let $\Omega \subset \R^N$ be an open bounded set with negligible boundary, namely $|\partial \Omega| = 0$.
Assume that $\{ \Omega_j \}_{j \in \N}$ is a sequence of open and bounded sets in $\R^N$ such that 
\begin{equation}
\label{e:siavvicinano}
       \mathfrak{d}(\Omega_j,   \Omega) < \frac{1}{j}. 
\end{equation}
 Let $(u^j, \lambda^j)$ be a sequence of eigencouples for the operator $(-\Delta)^s$ on $\Omega_j$ such that
 \begin{equation}
\label{e:ones}
    \| u^j \|_{L^2(\RN)} =1.
\end{equation}
 Assume furthermore that 
\begin{equation}
\label{e:autosiavvicinano}
       \lambda^j \to \lambda \quad \text{as $j \to + \infty$}
\end{equation}
for some $\lambda > 0$. Then there are a subsequence $\{j_k\}$  and $u \in \Xzs(\Omega)$ such that 
\begin{equation}
\label{e:strongconvergence}
      u^{j_k} \to u \quad \text{strongly in $H^s (\R^N)$}
\end{equation}
 and $(u,\lambda)$ is an eigencouple for $(-\Delta)^s$ on $\Omega$.
\end{propos}
We make the following remarks:
\begin{itemize}
\item if the multiplicity of $\lambda$ is bigger than one, then it might happen that different subsequences converge to different, linearly independent, eigenfunctions associated to $\lambda$. 
\item If the multiplicity of $\lambda$ is bigger than one, then there is no hope of establishing a convergence rate (see Remark \ref{R:cex} in~\S~\ref{s:ef} for a counterexample).
\item If the multiplicity of $\lambda$ is 1 (i.e., if $\lambda$ is a \emph{simple} eigenvalue), then
we can establish quantitative estimates, as the next result shows. 
\end{itemize}
The following result provides an estimate for the convergence rate of (normalized) principal eigenfunctions, which are always simple. In the statement of Theorem~\ref{t:eigenfunction}, $\lambda_1(D)$ and $\lambda_1(B_r)$ denote the first eigenvalue of $\Ds$ on $D$ and $B_r$, respectively. 
\begin{theorem}\label{t:eigenfunction}
 Fix $s \in (0, 1)$ and let $\Omega_a$, $\Omega_b$ be two bounded, open sets satisfying the conditions i), ii) and iii) in the statement of Theorem~\ref{th:spectral_stability} and 
 \begin{equation}\label{dd-hypo}
 \max \left\{ \dd(\Omega_b,\Omega_a), \dd(\Omega_a,\Omega_b) \right\} < \frac{\rho\sin\theta}{2}.  
 \end{equation}
Let $\lambda^a_1$, $\lambda^b_1$ denote the first eigenvalue of the operator $(-\Delta)^s$ in $\Omega_a$ and $\Omega_b$, respectively, 
and let $e^a$ and $e^b$ be the corresponding eigenfunctions satisfying
\begin{equation}
\label{e:autofunctionie}
   \| e^a \|_{L^2(\RN)} = \| e^b\|_{L^2(\RN)} =1, \quad \int_{\R^N} (-\Delta)^{s/2} e^a(x) 
    (-\Delta)^{s/2} e^b(x) \, \d x \ge 0. 
\end{equation}
Define $\delta$ by setting 
$$
\delta : = \frac{1}{2} \left( \frac{1}{\lambda_1^a} - \frac{1}{\lambda_2^a}\right)>0.  
$$
Then there is a positive constant $\nu>0$, which only depends on $N$, $s$, $\rho$, $\theta$, $\mathrm{diam}\,D$, $r$, $\lambda_1(D)$ 
and $\delta$, such that the following holds\/{\rm :} 
\begin{equation}\label{ef-result}
 \left\| \dfrac{e^a}{\sqrt{\lO_1}} - \dfrac{e^b}{\sqrt{\lOe_1}} \right\|_s \leq \dfrac{C}{\sqrt{\lambda_1(D)}} \max \left\{ \delta^{-1}, \lambda_1(B_r) \right\} \min \left\{ \dd(\Oa,\Ob), \dd(\Ob,\Oa) \right\}^{s/2}, 
\end{equation}
for some constant $C = C(N,s,\rho,\theta,\mathrm{diam}\,D) \geq 0$, provided that 
$
     d_H^c (\Omega_a,   \Omega_b) \leq \nu.
$  
\end{theorem}
The proof of Theorem~\ref{t:eigenfunction} mainly relies on the abstract theory developed by Feleqi~\cite{Feleqi}.
Also, note that one can also obtain similar results for eigenfunctions associated to other simple eigenvalues. Moreover, 
in the case of non-simple eigenvalues, we can control a suitable notion of ``distance between eigenspaces''. We refer the reader to
Theorem~\ref{T:1} in \S~\ref{s:ef} for more details.

\medskip

We conclude the introduction by outlining the plan of the paper: 
in the next section we introduce some functional-analytic
background. In \S~\ref{s:localized}, we establish the main estimates on localized finite differences
of solutions. In \S~\ref{sec:geo} we discuss the regularity conditions on domains and state some
related geometrical properties. In \S~\ref{s:projection} we apply these results to control
the difference of two 
solutions to \eqref{eq:DP} supported in different domains. 
In \S~\ref{ss:proof_d_perturbation} we establish some domain perturbation estimates 
that constitute a weaker version of Theorem~\ref{th:domain_perturb}. These
estimates are improved in the subsequent \S~\ref{ss:proof_regularity} 
and \S~\ref{ss:boot2} by means of a bootstrap argument. In this way
we complete the proofs of Theorem~\ref{t:main} and of 
Theorem~\ref{th:domain_perturb}. Next, in \S~\ref{ss:proof_s_stability}
we discuss the behavior of eigenvalues under domain perturbations
and establish Theorem~\ref{th:spectral_stability}; finally, in \S~\ref{s:ef} 
we establish Theorem~\ref{t:eigenfunction} on the 
behavior of eigenfunctions.

%%%%%%%%%%%%%%%%%%%%%%%%%%%%%%%%%%%%%%%%%%%%%%%%%%%%%%%%%%%%%%%%%%%%%%%%%%%%%%%%%%%%%%%%%%%%%%%%%%%%%%%%%%%%%%%%%%%%%%%%%%

\subsection{Notation}
\label{s:notation}

For the reader's convenience, we collect the main notation used in the sequel.
In the rest of the paper, we denote by $C(a_1, \dots, a_k)$ a (generic) constant
that may only depend on the quantities $a_1, \dots, a_k$. 
Its precise value may vary on occurrence. Also, we use the following notation:
\begin{itemize}
\item $\Leb^N(E)$: the Lebesgue measure of the (measurable) set $E\subseteq \R^N$. 
\item a.e. $x$: $\Leb^N-$almost every $x$.  
\item $x \cdot y$: the Euclidean scalar product between the vectors $x$, $y \in \R^N$. 
\item $|x|$: the Euclidean norm of $x \in \R^N$.  
\item $B_r (x)$: the open ball of radius $r$ centered at $x$ in $\R^N$.
\item $d(x, E)$: the distance from the point $x \in \R^N$ to the set $E \subseteq \R^N$, namely
$$
    d(x, E) : = \inf_{y \in E} |x -y|. 
$$ 
\item $e(E,F)$: the {\it excess}\/ of the set $E \subseteq \R^N$ with respect to
the set $F \subseteq \R^N$, i.e.,
$$
    e(E, F) : = \sup_{y \in E} d(y,F).
$$ 
\item $d_H(E,F)$: the {\it Hausdorff distance}\/ between the sets $E,F \subseteq \R^N$,
given by 
$$d_H(E,F)=e(E,F)+e(F,E).
$$
\item $d_H^{c}(E, F)$: the {\it complementary Hausdorff distance}\/ defined by formula~\eqref{e:chaus} below. 
\item $\dd(E, F)$: the distance defined by formula~\eqref{d:front} below. Roughly speaking, it 
provides a measure of the excess of the boundary $\partial E$ with respect to the boundary $\partial F$. 
\item $H^s(\R^N)$: the fractional Sobolev space $W^{s,2}$. We usually focus on the case $s \in (0, 1)$.
\item $H^{-s} (\R^N)$: the dual space of $H^s (\R^N)$.
\item $(\cdot ,\cdot)$: the standard scalar product in $L^2(\RN)$, namely
$$
   (u, v) : = \int_{\R^N} u(x) v(x) \,\dx.  
$$ 
%
% \item $H_0(\Omega)$: the space of the $L^2(\RN)$ functions that vanish a.e.~in $\RN\setminus \Omega$,
%$$
%H_0(\Omega) : = \big\{ u\in L^2(\RN): \ u(x) = 0 \ \hbox{for a.e.~} x\in \RN\setminus\Omega\big\}.
%$$
%
\item $\spazio(\Omega)$: the functional space 
$$
    \spazio(\Omega) : = \big\{  u \in H^s (\R^N): \ 
       \text{$u (x) =0 $ for a.e.~$x \in \R^N\setminus \Omega$} \big\}.
$$
Here $\Omega \subseteq \R^N$ is a given open set.  
\item $B^r_{2,\infty}(\R^N):$ the Besov space defined as in~\S~\ref{subsec:besov}.  
\item $[\cdot, \cdot]_s$: the bilinear form 
$$
    [u, v]_s : =   \big( {\Dsm u}, {\Dsm v}\big)  
      = \int_{\R^N}  [ \Dsm u] (x)  [\Dsm  v] (x)\,\dx,
$$
which is defined on $H^s (\R^N)$ and is a scalar product 
on $\spazio (\Omega)$ if $\Omega$ is bounded.
\item $\| \cdot \|_s$: the Gagliardo semi-norm defined as in~\eqref{gagl}.  
\item $\spazio(\Omega)'$: the dual space of $\spazio(\Omega)$.
\item $\dual \cdot \cdot $: the duality product between 
$\spazio(\Omega)'$ and $\spazio(\Omega)$.
\item $u_h$, $T_h u$: the functions defined as in~\eqref{eq:local_transl}. 
\item $\hat v$ or $\mathfrak F v$: the Fourier transform of the function $v$ (whenever it makes sense).
\item $C^\infty_c (\Omega)$: the set of smooth, compactly supported functions, defined on the set $\Omega$. 
\end{itemize}

%%%%%%%%%%%%%%%%%%%%%%%%%%%%%%%%%%%%%%%%%%%%%%%%%%%%%%%%%%%%%%%%%%%%%%%%%%%%%%%%%%%%%%%%%%%%%
%
% PRELIMINARIES
%
%%%%%%%%%%%%%%%%%%%%%%%%%%%%%%%%%%%%%%%%%%%%%%%%%%%%%%%%%%%%%%%%%%%%%%%%%%%%%%%%%%%%%%%%%%%%%

\section{Functional analytic background material: fractional Laplacian and Besov spaces}
\label{sec:prem}

For the reader's convenience, in this section we discuss some functional analytic results 
that are pivotal to our analysis. In particular, in~\S~\ref{subsec:frac} we  provide the rigorous 
formulation of the Poisson problem~\eqref{eq:DP}. In~\S~\ref{subsec:besov} we introduce the definition
and some important property of a particular class of Besov spaces.

%%%%%%%%%%%%%%%%%%%%%%%%%%%%%%%%%%%%%%%%%%%%%%%%%%%%%%%%%%%%%%%%%%%%%%%%%%%%%%%%%%%%%%%%%%%%%%%%%%%%%%%%%%%%%%%%%%%%%%%%%%

\subsection{The Poisson problem for the fractional Laplace operator}
\label{subsec:frac}

%%%%%%%%%%%%%%%%%%%%%%%%%%%%%%%%%%%%%%%%%%%%%%%%%%%%%%%%%%%%%%%%%%%%%%%%%%%%%%%%%%%%%%%%%%%%%%%%%%%%%%%%%%%%%%%%%%%%%%%%%%

\subsubsection{The fractional Laplace operator} 

Given $s\in (0,1)$ and $u$ in the Schwartz
class $\mathcal S$ of the rapidly decaying functions at infinity,
$\Ds u$ is defined as 
\begin{equation}\label{def:fract_lapl}
\Ds u(x):= C(N,s)\pv\int_{\RN} \frac{u(x)-u(y)}{\vert x-y\vert^{N+2s}} \,\d y,
\end{equation}
where the notation $\pv$ means that the integral is taken in the 
{\em Cauchy principal value}\/ sense, namely
\begin{equation}\label{p.v.}
  \pv \int_{\RN} \frac{u(x)- u(y)}{\vert x-y\vert^{N + 2 s}} \,\d y
   = \lim_{\varepsilon\searrow 0} \int_{\RN\setminus B_\varepsilon(x)} 
          \frac{u(x)- u(y)}{\vert x-y\vert^{N + 2 s}} \,\d y
\end{equation}
and in~\eqref{def:fract_lapl}
$C(N, s)$ is the normalization constant (cf., e.g., \cite{Dine_Pala_Vald})
$$
     \left( \int_{\R^N} \frac{1 - \cos (x_1)  }{|x|^{N+2s}} \,\dx \right)^{-1} .
$$
For any $s\in (0,1)$ and any $x,y\in \RN$ we 
will also use the shorthand notation $K_s(x-y)=\vert x-y\vert^{-N-2s}$
to denote the singular kernel in~\eqref{def:fract_lapl}.
The operator $\Ds$ can be equivalently introduced by means of
the Fourier transform, which we define for general
function $v\in \mathcal S$ as follows:
$$
  \mathfrak{F}v(\xi) = \hat{v}(\xi) 
   = \frac{1}{(2\pi)^N}\int_{\RN}e^{-ix\cdot \xi}u(x)\,\d x.
$$
Moreover, $\mathfrak{F}^{-1}$ stands for the inverse transform of $\mathfrak F$, 
As usual, the above definition can be extended to tempered distributions.

We can then introduce $\Ds$ as the pseudo-differential 
operator with symbol $\vert \xi\vert ^{2s}$, namely
\begin{equation}
\label{def:fract_laplbis}
  \Ds v = \mathfrak{F}^{-1}(\vert \xi\vert ^{2s}\mathfrak{F}(v)),
   \quad \forall\, v\in \mathcal{S}.
\end{equation}
In particular, when $ v\in H^s(\RN)$, then $\Dsm v \in L^2(\RN)$
owing to the above characterization.

%%%%%%%%%%%%%%%%%%%%%%%%%%%%%%%%%%%%%%%%%%%%%%%%%%%%%%%%%%%%%%%%%%%%%%%%%%%%%%%%%%%%%%%%%%%%%%%%%%%%%%%%%%%%%%%%%%%%%%%%%%

\subsubsection{Functional framework}

Even if the equations we are going to study 
are settled only in $\Omega$, the behavior of $\Ds u$ 
depends on the interplay between the values of $u$
inside and outside $\Omega$. This is related
to the non locality of $\Ds$, which implies that, even when
$u$ has compact support, $\Ds u$ does not necessarily
have the same property.
For this reason, when we consider a solution $u$ to 
\eqref{eq:DP}, $u$ will be always thought as a function
defined on the whole space~$\RN$ that identically
vanishes outside $\Omega$. In particular, the global regularity of 
$u$ will be influenced by this fact (cf.~Theorem \ref{th:regularity}
and the examples discussed in the Introduction).
 
We now proceeding along the lines of \cite{serva-valdi11} (see also \cite{ASS}),
%we
%introduce some functional spaces. 
%Firstly, given an open and bounded set $\Omega$,
%we set
%
%\begin{equation}\label{defiH}
%   H_0(\Omega):=\big\{ v \in L^2(\RN): v=0~\text{a.e.~in $\R^N\setminus \Omega$} \big\}.  
%\end{equation}
%
%
%
%The space $H_0(\Omega)$ will be seen as a (closed) subspace of 
%$L^2(\RN)$. The scalar product in $L^2(\RN)$ (and, hence, in $H_0$)
%will be noted as $(\cdot,\cdot)$ in the sequel. The
%induced norm will be generally noted as $\| \cdot \|_{L^2 (\R^N)}$.
%Next, 
and, given $s\in (0,1)$, we introduce the Sobolev type spaces
\begin{equation}\label{eq:defX0_bis}
  \spazio(\Omega) := \big\{ v\in H^s(\RN) \hbox{ such that } 
     v=0 \hbox{ a.e.~in } \RN\setminus \Omega \big\}.
\end{equation}
In particular, for $s\in (0,1)$, the extension
operator of $u\in \spazio(\Omega)$ to $0$ outside $\Omega$
is a continuous mapping of $H^s(\Omega)\to H^s(\RN)$.
Then (see \cite[Theorem~11.4, Chapter 1]{lions_mag}),
if $s \in (1/2,1)$, the functions in $\spazio(\Omega)$ are
equal to zero in the sense of traces on $\partial \Omega$.
Hence, $\spazio(\Omega)$ can be identified with $H^s_0(\Omega)$ in
that case, whereas $\Xzs \sim H^s_0(\Omega) = H^s(\Omega)$
for $s \in (0,1/2)$. Finally, in the limit case 
$s=1/2$, it turns out that
$\mathcal{X}^{1/2}_{0}(\Omega) \sim H_{00}^{1/2}(\Omega)$
(again, see \cite{lions_mag} for more details).

We denote by $(\cdot,\cdot)$ the scalar product in $L^2(\RN)$ and by
$\| \cdot \|_{L^2 (\R^N)}$ the induced norm, and we 
endow the Hilbert space $\spazio(\Omega)$ with 
the norm
\begin{equation}\label{eq:normHsig}
  \|v\|^2_{\spazio(\Omega)}:= \| v\|^2_{L^2(\R^N)} + \|v\|^2_{s}.
\end{equation}
Here, $\|\cdot\|_{s}$ denotes the so-called {\em Gagliardo-seminorm}
\begin{equation}\label{gagl}
  \| v\|^2_{s}:= \iint_{\RDN} K_s(x-y )\vert v(x)-v(y) \vert^2 \,\d x \, \d y,
\end{equation}
which is well defined for $v \in H^s(\RN)$. 
We also recall the fractional 
Poincar\'e inequality
\begin{equation} \label{eq:poincare2}
 \| v\|_{L^2(\RN)}\le C_P (\Omega, s)  \| v\|_{s},\,\,\,\,\hbox{for every  } v\in \spazio(\Omega),
\end{equation}
where the constant $C_P$ depends in principle on $\Omega$ and on $s$. Note that,
since the set $\Omega$ is bounded, then its diameter (which is defined 
by~\eqref{e:diametro}) is finite and, also, $\Omega$ is contained in a 
suitable ball $D$ with radius equal to $\mathrm{diam} \, \Omega$.
In view of the fact that \eqref{eq:defX0_bis} implies 
$\spazio (\Omega) \subseteq \spazio (D)$, it follows 
$C_P(\Omega,s) \leq C_P(D,s)$ and consequently
we can choose the constant $C_P$ in~\eqref{eq:poincare2} 
depending only on $N$, $s$ and $\mathrm{diam} \, \Omega$. Namely, we have
\begin{equation} \label{eq:poincare}
 \| v\|_{L^2(\RN)}\le C (N, s, \mathrm{diam}\, \Omega)  \| v\|_{s},\,\,\,\,\hbox{for every  } v\in \spazio(\Omega). 
\end{equation}
As a consequence of~\eqref{eq:poincare}, the Gagliardo seminorm is actually an equivalent norm on 
$\spazio(\Omega)$. Hence, we will generally use $\|\cdot\|_{s}$ in place
of $\|\cdot\|_{\spazio(\Omega)}$.

It is also important to express the Gagliardo-seminorm by using the Fourier transform. 
We have (cf.~\cite[Proposition 3.4 \& Proposition 3.6]{Dine_Pala_Vald}):
\begin{equation}\label{eq:semi_norm}
  C(N, s)\| v \|^2_{s} 
    = \| \Dsm v\|_{L^2(\RN)}^2 
    =  \big\| \vert\xi\vert^s\hat{v} \big\|^2_{L^2(\RN)} 
      \ \hbox{ for } v\in H^s(\RN) 
      \ \hbox{ and } s\in (0,1).
\end{equation}
In the following, we will also use the notation
\begin{equation}\label{e:scalarpoincare}
    [u, v]_s : = \int_{\R^N}  [ \Dsm u] (x)  [\Dsm v] (x) \,\dx  
     = \big(  \Dsm u,  \Dsm  v  \big). 
\end{equation}
Note that, owing to~\eqref{eq:poincare} and~\eqref{eq:semi_norm}, the above
bilinear form is actually a scalar product on $\spazio (\Omega)$. 

In view of the fact that we will deal with domain variations,
it will be generally convenient to view the elements of $\spazio(\Omega)$ 
as functions defined on the whole space $\RN$ that vanish a.e.~outside~$\Omega$.
In particular, we can continuously embed $\spazio(\Omega)$ into $L^2(\R^N)$. 
This embedding is not dense, since the closure of $\spazio(\Omega)$
in $L^2(\R^N)$ coincides with the subspace $H_0$
of $L^2(\R^N)$ containing those functions that vanish 
a.e.~outside $\Omega$.
%
%
%Moreover, we may identify the dual space $(H_0)'(\Omega)$ of
%$H_0(\Omega)$ with $H_0$ itself via the Riesz isomorphism
%induced by the scalar product $(\cdot,\cdot)$. 
%Hence, since $\spazio(\Omega)$ can be seen as a {\it dense}\/
%subspace of $H_0(\Omega)$, one may consider the Hilbert triple
%(also depending on $\Omega$).
%
%\begin{equation}\label{H-tr}
%  \spazio(\Omega) \hookrightarrow H_0(\Omega) 
%    \simeq (H_0)'(\Omega) \hookrightarrow \spazio(\Omega)',
%\end{equation}
%
%with compact and densely defined canonical injections. 
%
%It is also clear that $L^2(\RN)$ can be seen as a subspace of 
%$(H_0)'(\Omega)$ (and, consequently, of $\spazio(\Omega)'$)
%and we will make this identification whenever necessary.
%It is however important to point out that two different functions
%$f$ and $g$ in $L^2(\RN)$ may give rise to the same 
%element of $(H_0)'(\Omega)$. This, of course, happens whenever
%$f$ and $g$ coincide a.e.~in $\Omega$, but do not coincide
%a.e.~in $\RN \backslash \Omega$. 
%
In particular, if we denote by $H^{-s}(\R^N)$ the dual space of 
$H^s (\R^N)$, we have the chain of embeddings 
$$
    L^2 (\R^N) 
    \hookrightarrow H^{-s} (\R^N) 
    \hookrightarrow 
    \spazio(\Omega)'
$$ 
and both the above embeddings are continuous, namely
\begin{equation}
\label{e:embeddingchian}
   \| f \|_{ \spazio(\Omega)'  }
   \leq \| f\|_{H^{-s} (\R^N)}
   \leq \| f \|_{L^2 (\R^N)} \quad \text{for every $f \in L^2 (\R^N)$}. 
\end{equation}

%%%%%%%%%%%%%%%%%%%%%%%%%%%%%%%%%%%%%%%%%%%%%%%%%%%%%%%%%%%%%%%%%%%%%%%%%%%%%%%%%%%%%%%%%%%%%

\subsubsection{The Poisson problem for the fractional Laplacian}
\label{subsec:poiss}

With the above functional framework at our disposal, we can make precise the
notion of weak solution we are interested in. 

Given $f\in \spazio(\Omega)'$, we say that $u:\RN\to \R$ is 
a weak solution to \eqref{eq:DP} if
\begin{equation} \label{eq:weak_solution}
 \displaystyle \begin{cases}
    u\in \spazio(\Omega), \\
    \displaystyle C(N, s) \iint_{\RdN} K_s(x-y) (u(x)-u(y))(\vf(x)-\vf(y))\,\d x\,\d y = \dual f \vf,
       \,\,\forall\, \vf \in \spazio(\Omega).
 \end{cases}
\end{equation}
It is worth noting that \eqref{eq:weak_solution} may be equivalently 
reformulated as
\begin{equation} \label{eq:weak_solution_bis}
  \displaystyle \begin{cases}
     u\in \spazio(\Omega), \\
     \big( {\Dsm u}, {\Dsm \vf }\big) = \dual f \vf,
        \,\,\forall\, \vf \in \spazio(\Omega).
 \end{cases}
\end{equation}
%
%or also as
%
%\begin{equation} \label{eq:abstract_form}
% \As u = f \,\,\,\,\hbox{ in } \,\spazio(\Omega)',
%\end{equation}
%
%where we have denoted with $\As$ the realization of the 
%fractional Laplacian in the duality between
%$\spazio(\Omega)'$ and $\spazio(\Omega)$, that is
%the operator $\As:\spazio(\Omega)\to \spazio(\Omega)'$
%defined by
%
%\begin{align} \label{eq:weak_fractional}
%   \displaystyle \dual{\As u}{\varphi} 
%   & := [u, \varphi]_s =  \big( {\Dsm u}, {\Dsm \vf}\big)\\
% \nonumber   
%   & = C(N, s) \iint_{\RdN} K_s(x-y)(u(x)-u(y))(\vf(x)-\vf(y)) \,\dx \, \dy
%   \displaystyle \ \ \text{for all }\, u, \vf\in \spazio(\Omega).
%\end{align}
%%
%As noted in the introduction, $\As$ is generally indicated
%as the (variational form of the) {\it restricted}\/ fractional Laplacian.
%Indeed, the equation $\Ds u = f$ is required to hold only in $\Omega$
%even though $\Ds u$ depends also on the values (coinciding in fact
%with $0$) assumed by $u$ outside $\Omega$.
%This terminology refers to the fact, already noted in the Introduction,
%that, when $f\in L^2(\RN)$ is given (not necessarily vanishing outside 
%$\Omega$), the equation $\Ds u = f$ is required to hold only in $\Omega$
%even though $\Ds u$ depends also on the values (coinciding in fact
%with $0$) assumed by $u$ outside $\Omega$. Relation \eqref{eq:weak_fractional}
%expresses this fact in a concise and rigorous way. 
%
In what follows, when we speak of a weak (or variational)
solution $u$ to \eqref{eq:DP}, we will mean a function $u\in \spazio(\Omega)$
satisfying \eqref{eq:weak_solution} or, equivalently, \eqref{eq:weak_solution_bis}.
%\eqref{eq:abstract_form} (or, equivalently,
%\eqref{eq:weak_solution}, or \eqref{eq:weak_solution_bis}). 

By using the Lax-Milgram Lemma one can show that, if $f\in \spazio(\Omega)'$
(and hence, in particular, if $f\in L^2(\RN)$), then there is a unique solution $u 
\in \spazio(\Omega)$. In particular, 
we have the stability estimate
\begin{equation} \label{eq:lax_milgram}
  \| u \|_s \leq   \| u\|_{\spazio(\Omega)} 
    \le C (N, s, \mathrm{diam} \, \Omega)  \| f\|_{\spazio(\Omega)'}.
\end{equation}

%%%%%%%%%%%%%%%%%%%%%%%%%%%%%%%%%%%%%%%%%%%%%%%%%%%%%%%%%%%%%%%%%%%%%%%%%%%%%%%%%%%%%%%%%%%%%%%%%%%%%%%%%%%%%%%%%%%%%%%%%%

\subsection{Besov spaces and interpolation}
\label{subsec:besov}

In this paragraph we recall the definition and some properties of the Besov space $B^{r}_{2, \infty}(\R^N)$
and we refer to the books by Triebel~\cite{Triebel} and by Stein~\cite[\S V.5]{Stein} for extended discussions.  

If $r \in (0, 1]$, then 
\begin{equation}
\label{e:b:definition}
    B^r_{2, \infty} (\R^N) : = \left\{ u \in L^2 (\R^N): \, \sup_{h \in \R^N \setminus \{ 0 \}} \frac{\| u_{2h} -
    2 u_h+ u \|_{L^2(\R^N)}}{|h|^r} < + \infty   \right\}, 
\end{equation}
where the function $u_h$ is defined as in~\eqref{eq:local_transl} and $u_{2h}(x) :=u(x+ 2h)$. 
The space $B^r_{2, \infty} (\R^N)$ is a Banach space with norm 
\begin{equation}
\label{e:b:norm}
    \| u \|_{ B^r_{2, \infty} (\R^N)}: = \| u \|_{L^2 (\R^N)} + \sup_{h \in \R^N \setminus \{ 0 \}} \frac{\| u_{2h} -
    2 u_h+ u \|_{L^2(\R^N)}}{|h|^r}. 
\end{equation}
Note furthermore that, if $r \in (0, 1)$, then $u \in  B^r_{2, \infty} (\R^N)$ if and only if
$$
   \| u \|_{L^2(\R^N)} + \sup_{h \in \R^N \setminus \{ 0 \}} \frac{\| u_{h} -
    u \|_{L^2(\R^N)}}{|h|^r} < + \infty
$$
and, also, the expression on the left hand side is equivalent to the norm $\| \cdot \|_{ B^r_{2 \infty} (\R^N)}$, namely
\begin{align}
\label{e:b:equivalent}
         \sup_{h \in \R^N \setminus \{ 0 \}} \frac{\| u_{h} -
    u \|_{L^2(\R^N)}}{|h|^r}  \leq C(N, r)
     \sup_{h \in \R^N \setminus \{ 0 \}} \frac{\| u_{2h} -
    2 u_h+ u \|_{L^2(\R^N)}}{|h|^r}
    \\
    \sup_{h \in \R^N \setminus \{ 0 \}} \frac{\| u_{2h} -
    2 u_h+ u \|_{L^2(\R^N)}}{|h|^r}
    \leq C(N, r) \sup_{h \in \R^N \setminus \{ 0 \}} \frac{\| u_{h} -
    u \|_{L^2(\R^N)}}{|h|^r}.  
    \end{align} 
If $r>1$, then 
$$
B^r_{2, \infty} (\R^N) : = 
\left\{ u \in L^2 (\R^N): 
\frac{\partial u }{\partial x_i} \in B^{r-1}_{2, \infty} (\R^N) \; \text{for every} \; i=1, \dots, N
\right\}.
$$
and 
\begin{equation}
\label{e:b:higher}
 \| u \|_{ B^r_{2, \infty} (\R^N)}: = \| u \|_{L^2 (\R^N)} + 
 \sum_{i=1}^N \left\| \frac{\partial u }{\partial x_i} 
 \right\|_{B^{r-1}_{2, \infty} (\R^N)}. 
\end{equation}
In the above expressions, $\partial u / \partial x_i$ denote
distributional partial derivatives of $u$. 
We now recall some properties of $B^{r}_{2, \infty}(\R^N)$ 
that will be used in the following. First, 
\begin{equation}
\label{e:b:inclusion}
        B^r_{2, \infty} (\R^N)  \subset H^{r - \varepsilon} (\R^N), 
  \quad \text{for every $r \in ]0, + \infty[$, $\varepsilon \in (0, r)$}  
\end{equation}
and 
\begin{equation}
\label{e:b:inclusion2}
        H^r (\R^N) \subset B^r_{2, \infty} (\R^N), \quad \text{for every $r \in ]0, + \infty[$}.  
\end{equation}
Also, all the above inclusions are continuous. Second, we 
follow~\cite[p.~131]{Stein} and we consider the equation
\begin{equation}
\label{e:b:stein}
     (I - \Delta)^s u = f \quad \text{in $\R^N$},
\end{equation}
where $I$ denotes the identity operator and as usual $s \in (0, 1)$.
A rigorous formulation of~\eqref{e:b:stein} can be provided by 
using the so-called \emph{Bessel potential}. For our purposes here it 
is enough to say that, if $f \in L^2 (\R^N)$, then 
$u \in L^2 (\R^N)$ satisfies~\eqref{e:b:stein} if and only if 
\begin{equation}
\label{e:bessel}
    (1 + |\xi|^2)^{s} \hat u (\xi) = \hat f (\xi) \quad \text{for a.e. $\xi \in \R^N$}. 
\end{equation}
As a particular case of~\cite[Theorem $4'$, p.~153]{Stein} we obtain the following
\begin{lemma}
\label{l:stein}
Assume that $f \in B^{r}_{2, \infty} (\R^N)$ and that $u$ satisfies~\eqref{e:b:stein}. 
Then $u \in B^{r+2s}_{2, \infty} (\R^N)$ and 
\begin{equation}
\label{e:stein}
        \| u \|_{B^{r+2s}_{2, \infty} (\R^N)} \leq C(N, s, r) 
        \| f \|_{B^{ r }_{2, \infty} (\R^N)} . 
\end{equation}
\end{lemma}

%%%%%%%%%%%%%%%%%%%%%%%%%%%%%%%%%%%%%%%%%%%%%%%%%%%%%%%%%%%%%%%%%%%%%%%%%%%%%%%%%%%%%%%%%%%%%%%%%%%%%%%%%%%%%%%%%%%%%%%%%%
%
%     NEW SECTION
% 
%%%%%%%%%%%%%%%%%%%%%%%%%%%%%%%%%%%%%%%%%%%%%%%%%%%%%%%%%%%%%%%%%%%%%%%%%%%%%%%%%%%%%%%%%%%%%%%%%%%%%%%%%%%%%%%%%%%%%%%%%%

\section{Estimates on localized finite differences} 
\label{s:localized}

This section aims at establishing estimate~\eqref{eq:key1} and Lemma~\ref{l:bootstrap1} below. 
\begin{lemma}
\label{l:bootstrap1}
Let $\Omega \subset \R^N$ be a bounded, open set, and let $f \in L^2 (\R^N)$. 
Assume that $u$ and $\phi$ satisfy the same assumptions as in the statement of Theorem~\ref{t:main}. Then for every $\sigma \in (0, 1)$
we have 
\begin{equation}
\label{e:bootstrap1}
    \| T_h u - u \|^2_s \leq C(N, s, \mathrm{Lip} \, \phi, \mathrm{diam} \, \Omega, \sigma) |h|^{\sigma s} 
    \| f \|_{L^2 (\R^N)} \| f \|_{\spazio(\Omega) '}  
\end{equation}
\end{lemma}
Note that in~\eqref{e:bootstrap1} the constant $C$
depends on $\sigma$ and can in principle deteriorate when $\sigma \to 1^-$. 
For this reason~\eqref{e:bootstrap1} can be regarded as a weaker version of~\eqref{eq:key2}.

%%%%%%%%%%%%%%%%%%%%%%%%%%%%%%%%%%%%%%%%%%%%%%%%%%%%%%%%%%%%%%%%%%%%%%%%%%%%%%%%%%%%%%%%%%%%%%%%%%%%%%%%%%%%%%%%%%%%%%%%%%

\subsection{Proof of the estimate~\eqref{eq:key1}} 
\label{s:dimkey1}

We first recall that $0 \leq \phi \leq 1$ owing to~\eqref{e:hypphi} and we point out that 
\begin{equation}
\label{eq:est_norma21}
 \| T_h u - u \|_{L^2(\RN)} =
  \| \phi(u_h - u)\|_{L^2(\RN)} \le \|\phi\|_{L^\infty(\RN)}\| u_h-u\|_{L^2(\RN)}
  \le \| u_h-u\|_{L^2(\RN)}.
\end{equation}
Hence, to establish~\eqref{eq:key1}, it is sufficient 
to control $\| u_h-u\|_{L^2(\RN)}$. We recall
that $\hat{u}_h = e^{i\xi\cdot h}\hat u$. Then, the Plancherel theorem and 
\eqref{eq:lax_milgram} give
\begin{align}\label{e:plancherel}
  \| u_h - u\|^2_{L^2(\RN)}
   & = \int_{\RN}\vert e^{i \xi \cdot h}-1\vert^{2s}
     \vert e^{i \xi \cdot h}-1\vert^{2-2s} \vert \hat u(\xi)\vert^2\,\d \xi\\
  \nonumber 
   & \stackrel{\eqref{elem:11},\eqref{elem:12}}{\le} 8 \vert h\vert^{2s}
     \int_{\RN}\vert \xi\vert^{2s}\vert \hat{u}\vert^{2}\,\d \xi
   \stackrel{\eqref{eq:semi_norm}}{=} C (N, s)  \vert h\vert^{2s} \| u\|_{s}^2\\
  \nonumber
   & \stackrel{\eqref{eq:lax_milgram}}{\le}
     C (N, s, \mathrm{diam} \, \Omega)  \vert h\vert^{2s}\| f\|_{\spazio(\Omega)'}^2.
\end{align}
In the previous formula we have used the following elementary inequalities: 
first, since $s \in (0, 1)$, we have  
\begin{equation}\label{elem:11}
    \vert e^{i \xi \cdot h}-1\vert^{2-2s} \leq \Big( |e^{i \xi \cdot h} | + |1| \Big)^{2-2s}
    \leq 2^{2-2s} \leq 4. 
\end{equation}
Second, a direct check shows that
\begin{equation}\label{elem:12}
  \vert e^{i \xi \cdot h}-1\vert 
   = \big| \cos (\xi \cdot h) - 1 + i \sin(\xi \cdot h) \big|
   \leq 2 |\xi \cdot h| \le 2|\xi| |h|. 
\end{equation}

%%%%%%%%%%%%%%%%%%%%%%%%%%%%%%%%%%%%%%%%%%%%%%%%%%%%%%%%%%%%%%%%%%%%%%%%%%%%%%%%%%%%%%%%%%%%%%%%%%%%%%%%%%%%%%%%%%%%%%%%%%

\subsection{Preliminary results}
\label{ss:preliminary}

\begin{lemma}
\label{l:laplacianophi}
If $\phi$ is a cut-off function satisfying~\eqref{e:hypphi} and $s \in (0, 1)$, then 
\begin{equation}
\label{e:laplacianophi}
   \|(-\Delta)^{s/2}\phi\|_{L^\infty(\RN)} \le C(N, s, \mathrm{Lip} \, \phi).
\end{equation}
\end{lemma}
\begin{proof}
We first observe that, since $s/2 \in (0, 1/2)$, then 
\begin{equation}
\label{e:nopv}
    (-\Delta)^{s/2}\phi (x) = C(N, s) \int_{\R^N} \frac{\phi(x) - \phi(y)}{|x -y|^{N+s}} \,\dy
\end{equation}
 because the above integral converges. We deduce that 
 \begin{equation}
 \begin{split}
       | (-\Delta)^{s/2}\phi (x) | & \leq 
        C(N,s )  \int_{\R^N}
         \frac{|\phi (x) - \phi(y)|}{|x -y|^{N+s}} \,\dy\\ & = 
        C(N,s) \int_{|x-y| \leq 1 }
         \frac{|\phi (x) - \phi(y)|}{|x -y|^{N+s}} \,\dy + 
         C(N,s) \int_{|x-y| > 1 }
         \frac{|\phi (x) - \phi(y)|}{|x -y|^{N+s}} \,\dy  \\
         &  \leq 
        C(N,s,  \mathrm{Lip} \, \phi )
        \int_{|x-y| \leq 1 }
         \frac{1}{|x -y|^{N-1+s}} \,\dy
        + C(N, s) 
          \int_{|x-y| > 1 }
         \frac{1}{|x -y|^{N+s}} \,\dy
        \\ & = 
        C(N,s,  \mathrm{Lip} \, \phi )
        \int_0^1
         \frac{1}{\rho^s} \,\drho
         + C(N,s ) \int_1^{+\infty}      \frac{1}{\rho^{1+s}} \,\drho
         = C(N,s,  \mathrm{Lip} \, \phi ),\\  
\end{split}
\end{equation} 
which establishes~\eqref{e:laplacianophi}.
\end{proof}
\begin{lemma}
\label{l:phiu}
If $u \in \spazio (\Omega)$ and $\phi$ is a cut-off function 
satisfying~\eqref{e:hypphi}, then the product $u \phi \in \spazio (\Omega)$.
\end{lemma}
\begin{proof}[Proof of Lemma~\ref{l:phiu}]
First, we recall the definition~\eqref{eq:defX0_bis} of $\spazio (\Omega)$ and we infer that
$\phi (x) u (x) =0$ for a.e. $x \in \R^N \setminus \Omega$.  Hence, to prove that $\phi u
\in \spazio (\Omega)$ we are left to show that $\phi u \in H^s (\R^N)$.  
To this aim, we first observe that, since $0 \leq \phi \leq 1$ and $u \in L^2 (\R^N)$, 
then $\phi u \in L^2 (\R^N)$. Hence, it remains to prove that
$\| \phi u \|_s < + \infty$. Owing to~\eqref{eq:semi_norm}, this is
equivalent to showing
\begin{equation}
\label{e:goal}
(- \Delta)^{s/2} ( \phi u) \in L^2 (\R^N).
\end{equation} 
The above relation follows from general results related to the so-called 
{\it Kato-Ponce}\/ inequality. For instance, we can apply~\cite[formula (2)]{GOh}
with the choices $f=\phi$, $g=u$, $p_1=p_2=\infty$, $q_1=q_2=r=2$,
and with $s/2$ in place of $s$. By recalling~\eqref{e:laplacianophi}, we then obtain \eqref{e:goal}. Note furthermore that, strictly speaking,~\cite[formula (2)]{GOh} is only stated for smooth functions in the Schwartz class, but by relying on a standard density 
argument one can extend it to the (fractional) Sobolev setting. 
%
%To this end, we write 
%\begin{equation}
%\label{e:decompose}
%   (- \Delta)^{s/2} ( \phi u)  
%     = u \, (- \Delta)^{s/2}  \phi + \phi \, (- \Delta)^{s/2} u + 
%    \mathcal C (\phi, u), 
%\end{equation}
%where $\mathcal C (\phi, u)$ is the commutator 
%\begin{equation}
%\label{e:commutator}
%    \mathcal C (\phi, u): =  (- \Delta)^{s/2} ( \phi u)  - u \, (- \Delta)^{s/2}  \phi -
%     \phi \, (- \Delta)^{s/2} u. 
%\end{equation}
%We recall a key estimate due to Kenig, Ponce and Vega~\cite[Theorem A.12]{KPV93}:  
%%
%\begin{equation}
%\label{e:KPV2}
%  \|  \mathcal C (\phi, u) \|_{L^2 (\R^N)} 
%   \leq  C(N,s) \| \phi \|_{L^\infty (\R^N) } \|  (- \Delta)^{s/2} u \|_{L^2 (\R^N)}.
%\end{equation}
%%
%By plugging~\eqref{e:KPV2} into~\eqref{e:decompose} we obtain
%\begin{equation*}
%\begin{split}
%   \| (- \Delta)^{s/2} ( \phi u) \|_{L^2 (\R^N)} & 
%   \leq \| u \|_{L^2 (\R^N)} \| (- \Delta)^{s/2}  \phi \|_{L^\infty (\R^N) } + 
%   \| \phi \|_{L^\infty (\R^N) } \| (- \Delta)^{s/2}  u \|_{L^2 (\R^N)}\\
%   & \quad 
%   + C(N,s) \| \phi \|_{L^\infty (\R^N) } \|  (- \Delta)^{s/2} u \|_{L^2 (\R^N)}.\\
%\end{split}
%\end{equation*}
%
%Since by assumption $u \in H^s (\R^N)$,
%recalling~\eqref{e:laplacianophi} we conclude that the right hand side
%of the above formula is finite. This shows that $\phi u \in H^s (\R^N)$
%and concludes the proof of the lemma.   
%
\end{proof}
\begin{lemma}\label{l:interplp}
 Assume that $w \in H^s (\R^N)$ and that $\phi$ is a cut-off function 
 satisfying~\eqref{e:hypphi} and~\eqref{e:supporto}. 
 Let $\mathcal C(\phi, w)$ be the commutator defined by setting
 \begin{equation}
 \label{e:commutator}
     \mathcal C( \phi, w) : = 
     (- \Delta)^{s/2} ( \phi w)  - w \, (- \Delta)^{s/2}  \phi -
      \phi \, (- \Delta)^{s/2} w. 
 \end{equation}
 For every $\sigma \in (0, 1)$, there is a constant 
 $C(N, s, \mathrm{Lip} \, \phi, \sigma)$ such that 
 \begin{equation}
 \label{e:interplp}
      \| \mathcal C( \phi, w) \|_{L^2 (\R^N)} \leq 
      C(N, s, \mathrm{Lip} \, \phi, \sigma)  \| w \|^\sigma_{L^2 (\R^N)} 
      \| w \|_s^{1-\sigma}.
 \end{equation} 
 \end{lemma}
\begin{proof}[Proof of Lemma~\ref{l:interplp}]
First, we point out that, if $N \leq 2s$, then we can choose $s'<s$ 
such that $N > 2 s'$. Owing to~\cite[Proposition~2.1]{Dine_Pala_Vald}, 
we recall that, if $w \in \calX^s_0 (\Omega)$, 
then $ w \in \calX^{s'}_0 (\Omega)$ and furthermore 
$$
    \| w \|_{s'} \leq C \| w \|_{s}.
$$
Hence, it is enough to establish~\eqref{e:interplp} for $w \in \calX^{s'}_0 (\Omega)$. 
For this reason in the following we will always assume, with no loss of generality, that $N>2s$. 

We recall the fractional Sobolev embedding inequality, and we refer 
to~\cite[Theorem~6.5]{Dine_Pala_Vald} for an extended discussion. If 
$p^\ast = 2N / (N -2s)$ and $w \in H^s (\R^N)$, then $w \in L^{p^\ast} (\R^N)$ and 
\begin{equation}
\label{e:sobolevgn}
   \| w \|_{L^{p^\ast} (\R^N)}
   \leq C (N, s) \| w \|_s.  
\end{equation}
Next, we fix $\sigma \in (0, 1)$ and choose $r \in (2,p^*)$ in such a way that 
\begin{equation}
\label{e:erre}
   \frac{1}{r} = \frac{\sigma}{2} + \frac{1-\sigma}{ p^\ast}= C(N, s, \sigma).
\end{equation}
We also recall the elementary interpolation inequality 
\begin{equation}
 \label{e:interp}
    \| w \|_{L^r (\R^N)} \leq \Big(  \| w \|_{L^2 (\R^N)} \Big)^{\sigma} 
    \Big(  \| w \|_{L^{p^\ast} (\R^N)} 
    \Big)^{1-\sigma}     .
\end{equation} 
Finally, we use Theorem~A.8 in the paper by Kenig, Ponce and Vega~\cite{KPV93}, 
which states that
\begin{equation}
\label{e:kpv}
     \| \mathcal C( \phi, w) \|_{L^2 (\R^N)} \leq 
     C(N,s, p_1, p_2)  \| (- \Delta)^{s_1/2} \phi \|_{L^{p_1} (\R^N)}
     \| (- \Delta)^{s_2/2} w \|_{L^{p_2} (\R^N)}  
\end{equation}
provided that $s_1, s_2 \in [0, s]$, $s= s_1 + s_2$ and $p_1, p_2 \in (1, + \infty)$ satisfy
$$
   \frac{1}{2} = \frac{1}{p_1} + \frac{1}{p_2}.
$$
We apply~\eqref{e:kpv} with $s_1=s$, $s_2=0$ and $p_2=r$ and combine 
the result with~\eqref{e:sobolevgn} and~\eqref{e:interp}. We infer 
\begin{equation}
\label{e:quasi}
    \| \mathcal C( \phi, w) \|_{L^2 (\R^N)} \leq 
    C(N, s, \sigma) 
    \| (- \Delta)^{s/2} \phi \|_{L^{p_1} (\R^N)}
      \Big(  \| w \|_{L^2 (\R^N)} \Big)^{\sigma} 
   \Big(  \| w \|_{L^{p^\ast} (\R^N)} 
    \Big)^{1-\sigma}, 
\end{equation}
provided that
\begin{equation}
\label{e:piuno}
    \frac{1}{p_1} = \frac{1}{2} - \frac{1}{r}
    = (1-\sigma) \left( \frac{1}{2} - \frac{1}{p_\ast} \right)
    = (1-\sigma) \frac{s}{N}.
\end{equation}
To obtain the previous expression, we have used the explicit expression~\eqref{e:erre} 
of $r$. Also, note that $p_1 \in ]2, + \infty[$ since $N>2s$. 
To control the term $  \| (- \Delta)^{s/2} \phi \|_{L^{p_1} (\R^N)}$ we
use an argument similar to (but easier than) the one that gives the proof 
of~\cite[Lemma~1]{BonforteVazquez}. Namely, we write 
\begin{equation}
\label{e:elleq}
     \| (- \Delta)^{s/2} \phi \|^{p_1}_{L^{p_1} (\R^N)}
     = \int_{|x| \leq 2} |(- \Delta)^{s/2} \phi (x)|^{p_1} \,\dx +
      \int_{|x| > 2} |(- \Delta)^{s/2} \phi (x)|^{p_1} \,dx = : J_1 + J_2. 
\end{equation}
To control $J_1$, we simply recall Lemma~\ref{l:laplacianophi} and we obtain 
\begin{equation}
\label{e:geiuno}
    J_1 \leq \| (- \Delta)^{s/2} \phi  \|^{p_1}_{L^\infty (\R^N)} 2^N \omega_N \leq 
    C(N, s, \mathrm{Lip} \, \phi, p_1) = C(N, s, \mathrm{Lip} \, \phi, \sigma).   
\end{equation}
In the previous expression, $\omega_N$ denotes the measure of the unit ball in $\R^N$ 
and to establish the last equality we have used the explicit expression of $p_1$, namely~\eqref{e:piuno}. 
To control $J_2$, we first recall the equality~\eqref{e:nopv} and the fact that 
$\mathrm{supp} \, \phi \subseteq B_1(0)$ owing to~\eqref{e:supporto}. This implies that, if $|x| >2$, then 
\begin{equation}
\begin{split}
\label{e:geidue}
    | (- \Delta)^{s/2} \phi (x)| & = 
     C(N, s) \left| \int_{\R^N} \frac{\phi(x) - \phi(y)}{|x -y|^{N+s}} \,\dy \right| 
    =  C(N, s) \int_{\R^N} \frac{ \phi(y)}{|x -y|^{N+s}} \,\dy \\
    & =
     C(N, s) \int_{B_1 (0) } \frac{\phi(y)}{|x -y|^{N+s}} \,\dy
     \leq C(N, s)  \int_{B_1 (0) } \frac{1}{|x -y|^{N+s}} \,\dy
     \leq C(N, s) \frac{1}{|x|^{N+s}}. 
\end{split}
\end{equation}
To establish the last inequality, we have used the fact that, if $|x|>2$ and $|y| \leq 1$, then
$$
   |x-y| \ge |x| - |y| \ge \frac{|x|}{2}. 
$$
Using~\eqref{e:geidue} and recalling that $p_1$ is given by~\eqref{e:piuno}, we obtain 
\begin{equation}
\label{e:geidue2}
      J_2 =  \int_{|x| > 2} |(- \Delta)^{s/2} \phi (x)|^{p_1} \,\dx
      \leq C(N, s) \int_2^{+ \infty} 
      \frac{1}{\rho^{N(p_1-1) + 1+ sp_1 }} \,\drho =
      C(N, s, \sigma). 
\end{equation}
Note that the above integral converges since $p_1 >2$ owing to~\eqref{e:piuno}.
By combining~\eqref{e:quasi},~\eqref{e:elleq},~\eqref{e:geiuno} and \eqref{e:geidue2} 
we eventually establish~\eqref{e:interplp}.   
\end{proof}

%%%%%%%%%%%%%%%%%%%%%%%%%%%%%%%%%%%%%%%%%%%%%%%%%%%%%%%%%%%%%%%%%%%%%%%%%%%%%%%%%%%%%%%%%%%%%%%%%%%%%%%%%%%%%%%%%%%%%%%%%%

\subsection{Proof of Lemma~\ref{l:bootstrap1}}
\label{s:partialproof}

First of all, we decompose $\| T_h u- u\|_{s}$ as
\begin{equation}\label{eq:decompose_normas}
  \| T_h u- u\|_{s}^2 = \underbrace{\| T_h u\|^2_{s} - \| u\|^2_{s} }_{A} 
   - \underbrace{C(N, s) \big(\Lsm u,\Lsm[ (T_h u)- u]\big)}_{B}.
\end{equation}
We now separately control the terms $A$ and $B$ by proceeding according to the following steps. \\
{\sc Step 1:} we control the term $B$. By combining \eqref{e:hyprodotto}
with Lemma~\ref{l:phiu} we infer that 
$$ 
  v = T_h u-u = \phi (u_h -u)\in \spazio(\Omega)
$$ 
and hence we can use it 
as a test function in~\eqref{eq:weak_solution_bis}.
Consequently, by using~\eqref{eq:key1} we obtain 
\begin{align}
\label{e:scalar}
     |B | = \left| \Big( {\Dsm u},  {\Dsm[ (T_h u)- u]} \Big) \right|
     & \stackrel{\eqref{eq:weak_solution_bis}}{=}  \left| \langle f,   {(T_h u)- u}  \rangle \right|
     \le \|f\|_{L^2(\RN)}\| (T_h u )- u \|_{L^2(\RN)}\nonumber\\
     &  \stackrel{\eqref{eq:key1}}{\leq} 
     C (N, s, \mathrm{diam} \, \Omega)  | h|^s \| f\|_{L^2(\RN)}\|f\|_{\spazio(\Omega)'}. 
\end{align}
{\sc Step 2:} we rewrite the term $A$ in~\eqref{eq:decompose_normas} in a more convenient form.  
%\III To this aim, we start with 
%introducing a notation: given a couple $v,z$ of sufficiently regular functions, we 
%define the {\it commutator}\/ $\mathcal C (v, z)$ as
%%
%\begin{equation}
%\label{e:commutator}
%    \mathcal C (v, z): =  (- \Delta)^{s/2} ( v z)  - z \, (- \Delta)^{s/2}  v 
%    - v \, (- \Delta)^{s/2} z. 
%\end{equation} \EEE
%
Actually, we observe that
\begin{align}\label{e:tiacca}
  \| T_h u\|^2_{s} & = C(N, s) \| \Lsm (T_h u)\|^{2}_{L^2(\RN)}
     = C(N, s) \| \Lsm [\phi u_h] + \Lsm [(1-\phi) u]\|_{L^2(\RN)}^2\\
  & = C(N, s)  \| \Lsm u + \Lsm[\phi(u_h - u)]\|^2_{L^2(\RN)}\nonumber\\
  & = C(N, s) \| \Lsm u +\phi \Lsm [u_h-u] + (u_h-u)\Lsm \phi 
   + \mathcal{C}(\phi, u_h -u)  \|_{L^2(\RN)}^2\nonumber\\
 \nonumber
  & = C(N, s) \| T_h (\Lsm u) + (u_h-u)\Lsm \phi +\mathcal{C}(\phi, u_h -u) \|^2_{L^2(\RN)}.
\end{align}
Here we used the fact that $(-\Delta)^{s/2} u_h = [(-\Delta)^{s/2} u]_h$ 
to ensure the last equality. Indeed, this equation holds true for every $u \in H^s(\RN)$, 
since the definition of the fractional Laplacian in terms of the Fourier transform implies 
$$
  \mathfrak F(\Lsm u_h) = |\xi|^s \widehat{u_h} 
   = |\xi|^s e^{i h \cdot \xi} \widehat u = e^{i h \cdot \xi} \mathfrak F(\Lsm u) 
   = \mathfrak F([\Lsm u]_h)
$$ 
for $u \in H^s(\RN)$.
%
%
%where $\mathcal C (\phi, u)$ is the commutator 
%\begin{equation}
%\label{e:commutator}
%    \mathcal C (\phi, u): =  (- \Delta)^{s/2} ( \phi u)  - u \, (- \Delta)^{s/2}  \phi -
%     \phi \, (- \Delta)^{s/2} u. 
%\end{equation}
%
We have
\begin{align} \label{eq:est_1}
 A = \| T_h u\|^2_{s} - \|u\|^2_{s}
  & = C(N, s) \| T_h (\Lsm u) + (u_h-u)\Lsm \phi +\mathcal{C}(\phi, u_h -u) \|^2_{L^2(\RN)}
  -\|u\|^2_{s}\nonumber\\
  & = C(N, s) \Big[ \| T_h (\Lsm u) + (u_h-u)\Lsm \phi\|^2_{L^2(\RN)}
   -  \|\Lsm u\|^2_{L^2(\RN)}\nonumber\\
  & \mbox{}~~~~~~~~~~ + \|\mathcal{C}(\phi, u_h -u) \|^2_{L^2(\RN)}
 + { 2} \Big(T_h (\Lsm u) + (u_h-u)\Lsm \phi, \mathcal{C}(\phi, u_h -u)\Big) \Big] \nonumber\\
  & = C(N, s) \Big[ I_1+ I_2 + I_3\Big], 
\end{align}
provided that
\begin{align}
 I_1 &:= \| T_h (\Lsm u) + (u_h-u)\Lsm \phi\|^2_{L^2(\RN)} - \|\Lsm u\|^2_{L^2(\RN)}\label{eq:defiuno}\\
 I_2 &:=  \|\mathcal{C}(\phi, u_h -u) \|^2_{L^2(\RN)}\label{eq:defidue}\\
 I_3 &:= 2\Big( T_h (\Lsm u) + (u_h-u)\Lsm \phi,  
    \mathcal{C}(\phi, u_h -u)\Big). \label{eq:defitre}
\end{align}
{\sc Step 3:} we control the term $I_1$. First, we rewrite it as
\begin{align}
I_1 &=  \underbrace{\| T_h (\Lsm u) + (u_h-u)\Lsm \phi\|^2_{L^2(\RN)}- 
 \|T_h(\Lsm u)\|^2_{L^2(\RN)}\nonumber}_{\displaystyle {I_{1, 1}}} \\
 &
 \quad + \underbrace{ \|T_h(\Lsm u)\|^2_{L^2(\RN)}- 
 \|\Lsm u\|^2_{L^2(\RN)}}_{\displaystyle {I_{1, 2}}}.
\end{align}
To control $I_{1,1}$, we use the elementary identity
$\vert a + b \vert^2 - \vert a\vert^2 =  b\cdot (b+ 2a)$, which gives
\begin{align}\label{eq:A1}
 & I_{1,1} = \int_{\RN}\bigg[ \Big| T_h (\Lsm u)(x)
 + \Big(u_h(x)-u(x) \Big)(\Lsm \phi)(x)\Big|^2 
   -\vert T_h(\Lsm u)(x) |^2 \bigg] \,\d x\nonumber\\
 & = \int_{\RN}
 \Big[u_h(x)-u(x)\Big](\Lsm \phi)(x)\cdot \Bigg[\Big[u_h(x)-u(x)\Big](\Lsm \phi)(x)
 + 2T_h (\Lsm u)(x)\Bigg] \,\d x\nonumber\\
 & \le \| \Lsm \phi \|_{L^\infty (\R^N)}  
 \| u_h - u\|_{L^2(\RN)}\Big[
  \| \Lsm \phi \|_{L^\infty (\R^N)} \| u_h -u\|_{L^2(\RN)} 
  + 2 \| T_h (\Lsm u)\|_{L^2(\RN)}\Big].
\end{align}
Next, we use the convexity of the
real valued function $y \mapsto \vert y \vert^2$. More precisely, recalling that 
$0 \leq \phi \leq 1$, then for almost every $x\in \RN$ we have 
\begin{equation}
\begin{split}
\label{e:convessa}
\vert T_h(\Lsm u)(x)\vert^2 & = 
\Big| 
\phi(x) (\Lsm u)_h(x) 
   + \big[ 1-\phi(x) \big] \Lsm u(x)\Big|^2\\
   & \le \phi(x) \Big| (\Lsm u)_h(x)\Big|^2 + (1-\phi(x))\Big| \Lsm u(x)
  \Big|^2. \\
\end{split}
\end{equation}
Owing to~\eqref{eq:lax_milgram}, the above inequality implies 
$$ 
  \|T_h (\Lsm u)\|_{L^2(\RN)} \leq 2  \| \Lsm u  \|_{L^2(\RN)}
   \stackrel{\eqref{eq:lax_milgram}}{\leq}
  C (N, s, \mathrm{diam}\, \Omega)  \| f\|_{\spazio(\Omega)'},
$$
whence, recalling~\eqref{e:laplacianophi},~\eqref{e:plancherel} 
and~\eqref{eq:A1}, we obtain
\begin{equation} \label{eq:A1bis}
  I_{1,1}\le  C (N, s, \mathrm{Lip} \,\phi, \mathrm{diam}\, \Omega) \vert h\vert^{s} \| f\|^2_{\spazio(\Omega)'}.
\end{equation} 
To control $I_{1,2}$ we use~\eqref{eq:lax_milgram} and~\eqref{e:convessa} and we infer 
\begin{align}\label{eq:A2} 
  I_{1,2} & = \int_{\RN}\Big(\vert T_h(\Lsm u)(x)\vert^2  
    -  \vert\Lsm u(x)\vert^2\Big)\,\d x\nonumber\\
   & \le \int_{\RN}\phi(x) \Big( \vert (\Lsm u)_h(x)\vert^2 
   - \vert (\Lsm u)(x)\vert^2 \Big) \, \d x \nonumber\\
  & = \int_{\RN}\vert \phi(x-h) -\phi(x)\vert \vert (\Lsm u)(x)\vert^2\,\d x 
   \le \mathrm{Lip} \, \phi \, |h|
   \int_{\RN} \vert (\Lsm u)(x)\vert^2\,\d x \nonumber \\
   & 
   \stackrel{\eqref{eq:lax_milgram}}{\leq} C(N, s, \mathrm{Lip} \, \phi, \mathrm{diam} \, \Omega)   
   \vert h  \vert \|f  \|^2_{\spazio(\Omega)'}. \nonumber \phantom{\int}
\end{align}
Combining the above inequality with~\eqref{eq:A1bis} and recalling
that $s \in (0, 1)$ and $|h| \leq 1$, we arrive at 
\begin{equation}\label{eq:estI1}
 \begin{split}
 I_1 & \le \vert I_{1,1}\vert +\vert I_{1,2}\vert 
   \le  C(N, s, \mathrm{Lip} \, \phi, \mathrm{diam} \, \Omega)  
     \big(\vert h\vert^{s} +\vert h\vert \big) \|f\|^2_{\spazio(\Omega)'} \\ &
   \le  C(N, s, \mathrm{Lip} \, \phi, \mathrm{diam} \, \Omega)   
    \vert h\vert^{s} \| f\|^2_{\spazio(\Omega)'}.\\
  \end{split}
\end{equation}
{\sc Step 4:} we control the term $I_2$. We combine~\eqref{eq:lax_milgram}, 
\eqref{e:plancherel} and Lemma~\ref{l:interplp}. 
We fix $\sigma \in (0, 1)$ and we obtain
\begin{equation}\label{eq:commutator}
\begin{split} 
I_2 =  \| \mathcal{C}(\phi, u_h -u) \|^2_{L^2(\RN)} 
& \stackrel{\eqref{e:interplp}}{\leq}  C(N, s, \mathrm{Lip} \, \phi,  \sigma)  \| u_h - u\|^{2 \sigma}_{L^2(\RN)}
\| u - u_h \|^{2-2\sigma}_s \\ & \leq 
C(N, s, \mathrm{Lip} \, \phi,  \sigma)  \| u_h - u\|^{2\sigma}_{L^2(\RN)}
\| u  \|^{2-2\sigma}_s \\ & 
  \stackrel{\eqref{eq:lax_milgram}}{\leq}
 C(N, s, \mathrm{Lip} \, \phi, \mathrm{diam} \, \Omega, \sigma) 
  \| u_h - u\|^{2 \sigma}_{L^2(\RN)} 
  \| f \|^{2-2\sigma}_{\spazio (\Omega)'} \\
  & \stackrel{\eqref{eq:key1}}{\leq}
   C(N, s, \mathrm{Lip} \, \phi, \mathrm{diam} \, \Omega, \sigma)
  |h|^{2 \sigma s}
  \| f \|^{2\sigma}_{\spazio (\Omega)'}  \| f \|^{2-2\sigma}_{\spazio (\Omega)'}  \\
  & =  C(N, s, \mathrm{Lip} \, \phi, \mathrm{diam} \, \Omega, \sigma) 
  |h|^{2 \sigma s}
  \| f \|^{2}_{\spazio (\Omega)'}. \\
\end{split} 
\end{equation} 
{\sc Step 5}: we control the term $I_3$. First, we point out that 
\begin{equation}\label{e:itre}
  | I_3| \le 2 \sqrt{ I_2  (I_1 + \|\Lsm u\|^2_{L^2(\RN)})}.
\end{equation}  
By combining~\eqref{eq:estI1} and~\eqref{eq:commutator} and recalling that $|h| < 1$ 
we get 
\begin{equation}
\label{e:itre1}
         \sqrt{ I_2  I_1}
         \leq  C(N, s, \mathrm{Lip} \, \phi, \mathrm{diam} \, \Omega, \sigma)
         \sqrt{|h|^{2 \sigma s +s } \| f \|^4_{\spazio (\Omega)'}}
         \leq  C(N, s, \mathrm{Lip} \, \phi, \mathrm{diam} \, \Omega, \sigma)
 |h|^{ \sigma s} \| f \|^2_{\spazio (\Omega)'}.
\end{equation}   
Next, by combining~\eqref{eq:lax_milgram} and~\eqref{eq:commutator} we infer 
\begin{equation}
\label{e:itre2}  
   \sqrt{ I_2   \|\Lsm u\|^2_{L^2(\RN)}}
  \le
 C(N, s, \mathrm{Lip} \, \phi, \mathrm{diam} \, \Omega, \sigma) \vert h \vert^{\sigma s} 
 \| f\|^2_{\spazio(\Omega)'}.
\end{equation}
Finally, we plug~\eqref{e:itre1} and~\eqref{e:itre2} into~\eqref{e:itre} and we arrive at 
\begin{equation}
\label{e:itrefinale}
       I_3  \leq  C(N, s, \mathrm{Lip} \, \phi, \mathrm{diam} \, \Omega, \sigma)
   |h|^{ \sigma s} \| f \|^2_{\spazio (\Omega)'}.
\end{equation}  
{\sc Step 6:} we conclude the proof of Lemma~\ref{l:bootstrap1}. 
We plug~\eqref{eq:commutator},~\eqref{eq:estI1} and~\eqref{e:itrefinale}
into~\eqref{eq:est_1} and by recalling that $\sigma \in (0, 1)$ we obtain
$$
  A = \| T_h u\|^2_{s} - \|u\|^2_{s}\le  
  C(N, s, \mathrm{Lip} \, \phi, \sigma, \mathrm{diam} \, \Omega) 
  \vert h\vert^{\sigma s}\|f\|_{\spazio(\Omega)'}^2.
$$
Recalling~\eqref{eq:decompose_normas} and using~\eqref{e:scalar}
we then deduce
$$
  \| T_h u- u\|_{s}^2 
    \le C(N, s, \mathrm{Lip} \, \phi, \sigma, \mathrm{diam} \, \Omega) 
    \vert h\vert^{\sigma s} \Big(\|f\|_{\spazio(\Omega)'}^2
      + \| f\|_{L^2(\RN)} \|f\|_{\spazio(\Omega)'}\Big),
$$
whence, using the inequality $\|f\|_{\spazio(\Omega)'} \leq \| f \|_{L^2 (\R^N)}$,
we eventually arrive at~\eqref{e:bootstrap1}.

%%%%%%%%%%%%%%%%%%%%%%%%%%%%%%%%%%%%%%%%%%%%%%%%%%%%%%%%%%%%%%%%%%%%%%%%%%%%%%%%%%%%%%%%%%%%%
%
%     GEOMETRICAL PRELIMINARIES
%
%%%%%%%%%%%%%%%%%%%%%%%%%%%%%%%%%%%%%%%%%%%%%%%%%%%%%%%%%%%%%%%%%%%%%%%%%%%%%%%%%%%%%%%%%%%%%

\section{Geometric background material: Hausdorff distance, Lipschitz cone condition 
   and cut-off functions}
\label{sec:geo}

In this section we introduce some preliminary notions
related to distance between sets in $\RN$, regularity
properties of open sets, covering lemmas and cut-off functions.

%%%%%%%%%%%%%%%%%%%%%%%%%%%%%%%%%%%%%%%%%%%%%%%%%%%%%%%%%%%%%%%%%%%%%%%%%%%%%%%%%%%%%%%%%%%%%

\subsection{Hausdorff and related distances between sets in $\RN$}
\label{subsec:haus}
We start by recalling some definitions and we refer to 
\cite[\S~2.3]{SS02} for more details.
Let $E, F \subseteq \R^N$ be two sets. 
The {\it excess}\/ or {\it unilateral Hausdorff distance}\/ of $E$ 
with respect to $F$ is defined as
\begin{equation}\label{defi:exc}
    e(E,F) := \sup_{x \in E} d (x, F)
     = \sup_{x \in E \backslash F} d (x, F)
     = \sup_{x \in E \backslash F} d (x, \partial F).
\end{equation}
The (bilateral) {\it Hausdorff distance}\/ between $E$ and $F$ is then
given by
\begin{equation}\label{defi:hausd}
   d_H (E, F) := e(E,F) + e(F,E) 
    = \sup_{x \in E} d (x, F) + \sup_{y \in F} d(y, E). 
\end{equation}
We also introduce the notions of
{\it internal excess}:
\begin{equation}\label{defi:exci}
  e^c(F,E) := e\big( \R^N \backslash E, \R^N \backslash F \big)
   = \sup_{x \in F\backslash E} d (x,\R^N \backslash F)
   = \sup_{x \in F\backslash E} d (x, \partial F)
\end{equation}
and of {\it complementary Hausdorff distance}:
\begin{equation}\label{e:chaus}
   d_H^{c}(E, F) := e^c(F,E) + e^c(E,F)
    = d_H \big( \R^N \backslash E, \R^N \backslash F \big). 
\end{equation}
Next, given a set $E \subseteq \R^N$ and a real number $\varepsilon>0$,
we term $E^{- \varepsilon}$ the (possibly empty) set
\begin{equation}
\label{e:ristretto}
     E^{- \varepsilon} : = \big\{ x \in E: \; B_\varepsilon (x) \subseteq E \big\}. 
\end{equation}
Moreover, we denote by $E^{\varepsilon}$ the set 
\begin{equation}\label{e:allargato}
   E^{ \varepsilon} : = \big\{ x \in \R^N: \;
    d(x, E) < \varepsilon  \big\}. 
\end{equation}
The following result is well known. We provide a proof just for
the sake of completeness.
\begin{lemma}\label{l:hausdorff}
 Let $E, F$ be subsets of $\R^N$ and let $\varepsilon >0$. Then we have the following:
 \begin{itemize}
 \item[i)]
 if  
 \begin{equation}
 \label{e:piccolo}
     e^c (F, E) < \varepsilon, 
 \end{equation}
 then 
 $$
    F^{- \varepsilon} \subseteq E;
 $$
 \item[ii)] if
 \begin{equation}
 \label{e:piccolo2}
     e(E, F) < \varepsilon, 
 \end{equation}
 then 
 $$
    E \subseteq F^{ \varepsilon}.
 $$
 \end{itemize}
\end{lemma}
\begin{proof}
Let us first prove i). By contradiction, let us assume
there is $x \in F$ such that $B_\varepsilon (x) \subseteq F$ 
and $x \notin E$. Since $B_\varepsilon (x) \subseteq F$, then 
$$
    d (x, \R^N \backslash F) \geq \varepsilon. 
$$
Moreover, since $x \notin E$, then $x \in F \backslash E$. Hence, we have 
$$
   e^c (F, E) = \sup_{y \in F \backslash E}  d (y, \R^N \backslash F) 
    \ge d (x, \R^N \backslash F) \geq \varepsilon, 
$$
which contradicts~\eqref{e:piccolo} and proves i).

\smallskip

\noindent%
To prove ii), we assume again by contradiction that 
there is $x \in E$ such that $d (x, F) \ge \varepsilon$. This implies that
$$
  e(E, F)  = \sup_{y \in E} d (y, F) 
   \ge d (x, F) \ge \varepsilon, 
$$
which contradicts~\eqref{e:piccolo2}. The lemma is proved.
\end{proof}
It is worth noting that, if it is in particular
$d_H^{c}(E, F) < \varepsilon$, then we have both
$E^{- \varepsilon} \subseteq F$ and $F^{- \varepsilon} \subseteq E$.
Analogously, if $d_H(E, F) < \varepsilon$, 
then $E \subseteq F^{ \varepsilon}$ and $F \subseteq E^{ \varepsilon}$.

In addition to that, we observe that, if we define the distance
\begin{equation} \label{d:front}
   \dd(E, F) := e(E,F) + e^c(F,E),
\end{equation}
then it turns out that
$$
    e(F \triangle E, \partial F) \leq \dd(E, F) \leq 2 e(F \triangle E, \partial F),
$$
where $F\triangle E = (E \backslash F) \cup (F \backslash E)$. Also, 
\begin{equation} \label{intorni}
  \dd(E, F) < \varepsilon \Longrightarrow
   F^{ - \varepsilon} \subseteq E \subseteq F^{ \varepsilon}.
\end{equation}
In other words, if the distance $\dd(E,F)$ is smaller than~$\varepsilon$,
then the boundary of $E$ is included in the $\varepsilon$-neighbourhood of
the boundary of $F$.

%%%%%%%%%%%%%%%%%%%%%%%%%%%%%%%%%%%%%%%%%%%%%%%%%%%%%%%%%%%%%%%%%%%%%%%%%%%%%%%%%%%%%%%%%%%%%%%%%%%%%%%%%%%%%%%%%%%%

\subsection{Construction of cut-off functions}
\label{s:cutoff}

The following covering lemma is classical (cf., e.g., \cite[p.~49]{AmbrosioFuscoPallara}).
Also in this case, we provide a proof for completeness:
\begin{lemma}\label{l:vitali}
 Let $E \subseteq \R^N$ satisfy $E \subseteq B_R (y)$ for some $R>0$, $y \in \R^N$.
 Fix $r>0$. Then there are $x_1, \dots, x_k \in E$ such that:
 \begin{itemize}
 \item[i)] the cardinality $k$ satisfies
 \begin{equation}
 \label{e:kappa}
     k \leq \left( \frac{ 2R + r}{r} \right)^N;
 \end{equation}
 \item[ii)] the balls $B_r (x_i)$ cover $E$, namely
 \begin{equation}
 \label{e:covering}
     E \subseteq \bigcup_{i=1}^k B_r (x_i);
 \end{equation}
 \item[iii)] the balls $B_{r/2} (x_i)$ are pairwise disjoint, namely
 \begin{equation}
  \label{e:disjoint}
      B_{r/2} (x_i) \cap B_{r/2} (x_j) = \emptyset \quad \text{if $i \neq j$.}
 \end{equation} 
 \end{itemize} 
\end{lemma}
\begin{proof}
We choose $x_1 \in E$ and we set $E_1: = E \setminus B_r (x_1)$. Next, we choose $x_2 \in E_1$ and we set 
$E_2: = E_1 \setminus B_r (x_2)$. We iterate this procedure: since $E$ is bounded, after some finite number $k$
of steps we obtain $E_{k+1} = \emptyset$. Then, by construction, \eqref{e:covering} is satisfied.
To establish~\eqref{e:disjoint}, 
we point out that $|x_i - x_j | \ge r$ if $i \neq j$. 
To establish~\eqref{e:kappa}, we observe that the balls $B(x_1), \dots, B(x_k)$
are also contained in the ball $B_{R + r/2} (y)$. Hence, we deduce
\begin{align*}
   k \, \omega_N \left(\frac{r}{2} \right)^N & =
    \sum_{i=1}^k \Leb^N \Big(  B_{r/2} (x_i) \Big)=  \Leb^N \left( \bigcup_{i=1}^k  B_{r/2} (x_i) \right)\\
     & \leq 
     \Leb^N \Big( B_{R + r/2} (y) \Big) = \omega_N (R + r/2)^N \phantom{\int},
\end{align*}
which implies~\eqref{e:kappa}. In the above expression, $\omega_N$ denotes the measure of the unit ball in $\R^N$.
\end{proof}
\noindent%
We conclude this paragraph with a result that gives the existence of a suitable 
family of cut-off functions. The proof is very standard (see for instance the 
proof of~\cite[Lemma 9]{LMS2013}), so we omit it.  
\begin{lemma} \label{l:cutoff}
 Let $x_1, \dots, x_k$ belong to $\R^N$. Let $r>0$ and let $B_r (x_1), \dots, B_r (x_k)$
 be open balls such that~\eqref{e:disjoint} holds.
 Then there are $(k+1)$ Lipschitz continuous functions 
 $\phi_0, \dots, \phi_k: \R^N \to \R$ satisfying the following requirements:
 \begin{align}  \label{eq:gradient}
   & 0 \leq \phi_i (x) \leq 1, \quad 
      |\nabla \phi_i (x) | \leq \frac{C(N)}{r} \
        \text{for a.e. $x \in \R^N$ and every $ i=1, \dots, k$};\\
  \label{e:tetazero}
   & \phi_0(x) = 0 \
     \text{if $x \in \bigcup_{i=1}^k B_r (x_i)$}, 
       \qquad 
      \phi_0(x) = 1
       \ \text{if $x \in \R^N \setminus \bigcup_{i=1}^k B_{2r} (x_i)$}; \\ 
  \label{e:tetai}
   & \phi_i (x) = 0 
    \ \text{if $x \in \R^N \backslash B_{2r}(x_i)$}; \\
   %        \quad 
   %       \phi_0(x) = 1
   %       \; \text{if $x \in B_r(x_i)$}
   %       \quad \text{for every $ i=1, \dots, k$} \\
  \label{e:somma}
   & \sum_{i=0}^k \phi_i(x) = 1 \quad 
          \text{for every $x \in \R^N$.}
 \end{align} 
\end{lemma}

%%%%%%%%%%%%%%%%%%%%%%%%%%%%%%%%%%%%%%%%%%%%%%%%%%%%%%%%%%%%%%%%%%%%%%%%%%%%%%%%%%%%%%%%%%%%%%%%%%%%%%%%%%%%%%%%%%%%%%%%%%

\subsection{The Lipschitz cone condition}
\label{s:Lip}
We first introduce the regularity assumption we use. 
We fix $\theta \in ]0, \pi/2[$, $\rho >0$ and $\mathbf n$ a unitary vector in $\R^N$. 
We term $\mathcal C_{\rho, \theta} (\mathbf n)$ the open cone 
$$
    \mathcal C_{\rho, \theta} (\mathbf n) : =
    \left\{  \mathbf v \in \R^N: \; 0< | \mathbf v| < \rho, \;  \mathbf v \cdot \mathbf n > | \mathbf v | \cos \theta \right\}.  
$$
\begin{definition} \label{d:cone}
 Assume $\Omega \subseteq \R^N$ is an open set and fix $x_0 \in \R^N$, 
 $\theta \in ]0, \pi/2[$, $\rho >0$. 
 We term $\mathcal N _{\rho, \theta} (x_0, \Omega)$ the (possibly empty) set of unit vectors $ \mathbf n \in \R^N$
 such that 
 \begin{itemize}
 \item[i)] for every $ \mathbf v \in  \mathcal C_{\rho, \theta} (\mathbf n)$ and every $y \in B_{3 \rho} (x_0) \cap \Omega$, we have 
 $(y -  \mathbf v) \in \Omega$.
 \item[ii)] for every $ \mathbf v \in  \mathcal C_{\rho, \theta} (\mathbf n)$ and every $y \in B_{3 \rho} (x_0) \setminus \Omega$, we have 
 $(y +  \mathbf v) \in \R^N \setminus \Omega$.
 \end{itemize}
 We say that $\Omega$ satisfies the \emph{uniform $(\rho, \theta)$-Lipschitz cone condition} if
 $$
    \mathcal N _{\rho, \theta} (x_0, \Omega) \neq \emptyset \qquad \forall \; x_0 \in \R^N. 
 $$
\end{definition}
\noindent%
We have the following simple result. We refer to Figure~\ref{f:cone} for a representation.
\begin{lemma}\label{l:pallanelcono}
 Let $\mathbf n$ be a unit vector and let $\rho > 0$ and $\theta \in ]0,\pi/2[$. 
 Assume that $0 < t < \rho/2$ and let $\varepsilon = t \sin \theta$. 
 If $x = t \mathbf n$, then
 $$
     B_\varepsilon (x) \subseteq C_{\rho, \theta} (\mathbf n).
 $$
\end{lemma}
\begin{proof}
We can refer to Figure~\ref{f:cone} and infer that the inclusion holds true. 
For completeness, we also provide an analytic proof. 
\begin{figure}
\centering
\caption{} \label{f:cone}
\psfrag{C}{{\color{blue}$C_{\rho, \theta} (\mathbf n)$}}
\psfrag{B}{{\color{red}$B_\varepsilon (x)$}}
\psfrag{x}{$x$} 
\psfrag{n}{$\mathbf n$} 
\vspace{2mm}
\includegraphics[scale=0.5]{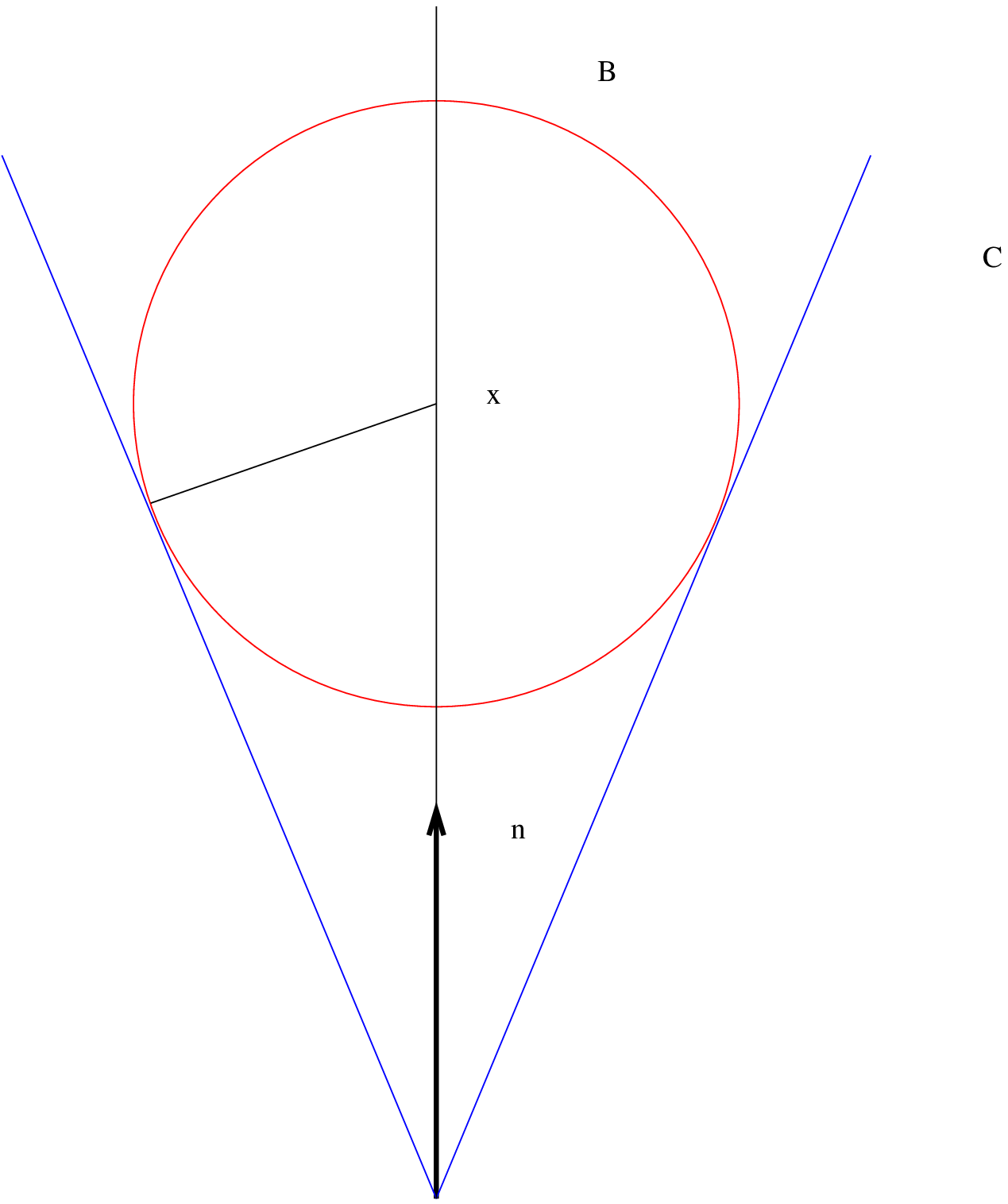}
\end{figure}
Assume that $\mathbf v \in  B_\varepsilon (x)$. Then
\begin{equation}
 \label{e:v}
    \mathbf v = t \mathbf n + \varepsilon \mathbf e
\end{equation}
for some $\mathbf e$ in the open unit ball centered at the origin. Since
$$
 0 < t (1 - \sin \theta) \leq 
 | \mathbf v| \leq t (1 + \sin \theta) < \rho, 
$$
recalling that $\varepsilon = t \sin \theta$,
we are left to show that
\begin{equation}\label{e:toverify}
   \mathbf v \cdot \mathbf n > |\mathbf v| \cos \theta.  
\end{equation}
Actually, using~\eqref{e:v}, we infer 
\begin{equation}\label{cono11}
   \mathbf v \cdot \mathbf n = t + \varepsilon  \, \mathbf e  \cdot \mathbf n,
\end{equation}
whence in particular $\mathbf v \cdot \mathbf n > 0$.
Moreover, using that $|\mathbf n|=1$ and $|\mathbf e|<1$, we infer
\begin{equation}\label{cono12}
  |\mathbf v|^2 = t^2 + 2 \varepsilon t \, \mathbf e  \cdot \mathbf n + \varepsilon^2 | \mathbf e |^2
   < t^2 + 2 \varepsilon t \, \mathbf e  \cdot \mathbf n + \varepsilon^2,
\end{equation}
Then, squaring both sides of~\eqref{e:toverify}, we are left to check that
$$
  ( \mathbf v \cdot \mathbf n )^2 - | \mathbf v |^2 \cos^2 \theta > 0.
$$  
Using \eqref{cono11} and \eqref{cono12} and subsequently that
$\varepsilon = t \sin \theta$, we then obtain
\begin{align*}
  ( \mathbf v \cdot \mathbf n )^2 - | \mathbf v |^2 \cos^2 \theta 
   & > t^2 + \varepsilon^2  ( \mathbf e  \cdot \mathbf n)^2 +2 \varepsilon t \, \mathbf e  \cdot \mathbf n \
    - \big( t^2 + 2 \varepsilon t \, \mathbf e  \cdot \mathbf n + \varepsilon^2 \big) \cos^2 \theta\\
  & = t^2 \sin^2 \theta + t^2 ( \mathbf e  \cdot \mathbf n)^2 \sin^2 \theta
   + 2 t^2 ( \mathbf e  \cdot \mathbf n) \sin^3 \theta - t^2 \sin^2 \theta \cos^2 \theta\\
  & = t^2 \sin^2 \theta \big( \sin \theta + ( \mathbf e  \cdot \mathbf n) \big)^2 \ge 0,
\end{align*}
whence follows \eqref{e:toverify}, as desired.
\end{proof}

%%%%%%%%%%%%%%%%%%%%%%%%%%%%%%%%%%%%%%%%%%%%%%%%%%%%%%%%%%%%%%%%%%%%%%%%%%%%%%%%%%%%%%%%%%%%%%%%%%%%%%%%%%%%%%%%%%%%%%%%%%
%
%     NEW SECTION
% 
%%%%%%%%%%%%%%%%%%%%%%%%%%%%%%%%%%%%%%%%%%%%%%%%%%%%%%%%%%%%%%%%%%%%%%%%%%%%%%%%%%%%%%%%%%%%%%%%%%%%%%%%%%%%%%%%%%%%%%%%%%

\section{Projection estimates}
\label{s:projection}

In this section we establish two preliminary results that are pivotal to the proof of 
Theorem~\ref{th:domain_perturb} and of Theorem~\ref{th:spectral_stability}.
\begin{lemma}
\label{l:tildeu}
Assume $\Omega$ is an open, bounded set satisfying a $(\rho, \theta)$-Lipschitz cone condition for some $\rho \in (0, 1/2]$, $\theta \in ]0, \pi/2[$.
Assume furthermore that $f \in L^2 (\R^N)$
and that $u \in \spazio (\Omega)$ is the weak solution of~\eqref{eq:DP}. If 
\begin{equation}
\label{e:eps}
    0 < \varepsilon < \frac{\rho \sin \theta }{2},
\end{equation} 
then there is $\tilde u \in \spazio (\Omega^{-\eps})$ such that
\begin{equation}
 \label{eq:est_norma2}
   \| \tilde u - u \|^2_{L^2(\RN)}\le 
   C(N, s,  \mathrm{diam} \, \Omega, \rho)   \left( \frac{\varepsilon}{\sin \theta}\right)^{2s} \| f \|^2_{\spazio(\Omega)'}
 \end{equation}
 and, for every $\sigma \in (0, 1)$,
 \begin{equation}
   \| \tilde u - u\|^2_{s}\le 
   C(N, s,  \mathrm{diam} \, \Omega, \rho, \sigma)
   \left( \frac{\varepsilon}{\sin \theta}\right)^{\sigma s} \| f\|_{L^2(\R^N)}\| f\|_{\spazio(\Omega)'}.
  \label{eq:est_normas}
 \end{equation}
\end{lemma}
Note that the assumption that $\rho \leq 1/2$ is not restrictive: indeed, Definition~\ref{d:cone} implies that, 
if $\rho_1 \leq \rho_2$ and $\Omega$ satisfies a $(\rho_2, \theta)$-Lipschitz cone condition, then it also 
satisfies a $(\rho_1, \theta)$-Lipschitz cone condition.
\begin{lemma}
\label{l:hatw}
Assume $\Omega$ is an open, bounded set satisfying a $(\rho, \theta)$-Lipschitz cone condition 
for some $\rho \in (0, 1/2]$, $\theta \in ]0, \pi/2[$. Let $\varepsilon$ satisfy 
\begin{equation}
\label{e:eps2}
    0 < \varepsilon < \frac{\rho \sin \theta }{2}.
\end{equation} 
Fix $f \in L^2 (\R^N)$ and let $w$ be the weak solution of 
\begin{equation}
\label{eq:DPeps}
\begin{cases}
\Ds w = f \ \hbox{ in } \Omega^\eps\\
w = 0 \ \hbox{ in } \R^N \setminus \Omega^\eps.
\end{cases}
\end{equation}
Then there is $\hat w \in \spazio (\Omega)$ such that
\begin{equation}
   \| \hat w - w \|^2_{L^2(\RN)}\le 
   C(N, s,  \mathrm{diam} \, \Omega, \rho)   
     \left( \frac{\varepsilon}{\sin \theta}\right)^{2s} \| f \|^2_{\spazio(\Omega^\eps)'}
  \label{eq:est_norma2w}
\end{equation}
and, for every $\sigma \in (0, 1)$,
   \begin{equation}
   \| \hat w  - w \|^2_{s}\le 
   C(N, s,  \mathrm{diam} \, \Omega, \rho,\sigma)
   \left( \frac{\varepsilon}{\sin \theta}\right)^{\sigma s} 
     \| f\|_{L^2(\R^N)}\| f\|_{\spazio(\Omega^\eps)'}.
  \label{eq:est_normasw}
   \end{equation}
\end{lemma}
%

%%%%%%%%%%%%%%%%%%%%%%%%%%%%%%%%%%%%%%%%%%%%%%%%%%%%%%%%%%%%%%%%%%%%%%%%%%%%%%%%%%%%%%%%%%%%%%%%%%%%%%%%%%%%%%%%%%%%%%%%%%

\subsection{Proof of Lemma~\ref{l:tildeu}}

%%%%%%%%%%%%%%%%%%%%%%%%%%%%%%%%%%%%%%%%%%%%%%%%%%%%%%%%%%%%%%%%%%%%%%%%%%%%%%%%%%%%%%%%%%%%%%%%%%%%%%%%%%%%%%%%%%%%%%%%%%

\subsubsection{Construction of $\tilde u$}
\label{sss:contildeu}

We fix $\Omega$, $\varepsilon$ and $u$ as in the statement of Lemma~\ref{l:tildeu}. 
We also fix a number $t$ such that
\begin{equation}
\label{eq:cond_h}
     \frac{ \varepsilon}{\sin \theta} < t < \frac{\rho}{2}. 
\end{equation}
We proceed according to the following steps. \\
{\sc Step 1:} we introduce the set 
\begin{equation}
\label{e:E}
    E : = \bigcup_{y \in \partial \Omega} B_{ \rho} (y). 
\end{equation}
We apply Lemma~\ref{l:vitali} with $R = \mathrm{diam} \, \Omega +1$, $r = \rho$ and we obtain that 
\begin{equation}
\label{e:coverE}
    E \subseteq  \bigcup_{i=1}^k B_{\rho} (x_i), \quad x_1, \dots, x_k \in \R^N,
\end{equation}
where every $x \in \R^N$ belongs to at most $5^N$ of the balls 
$B_{2\rho} (x_1), \dots, B_{2\rho} (x_k)$.   Owing to~\eqref{e:kappa}, 
the cardinality $k$ satisfies 
\begin{equation}
\label{e:kappa2}
  k \leq C(N, \mathrm{diam} \, \Omega, \rho).
\end{equation}
\begin{figure}
\begin{center}
\caption{The sets $\Omega^{-\varepsilon}$ (in red)
and $\Omega$ (in blue) and the balls covering the set $E$, see~\eqref{e:E} and~\eqref{e:coverE}.}
\psfrag{A}{{\color{blue} $\Omega$}} 
\psfrag{B}{{\color{red} $\Omega^{-\varepsilon}$}} 
\label{f:s1}
\includegraphics[scale=0.4]{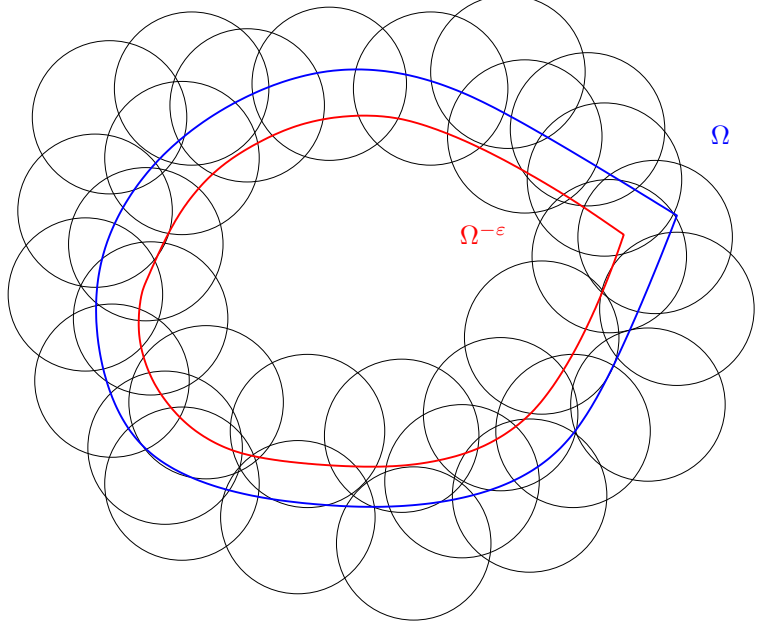}
\end{center}
\end{figure}
%
%
%   FINE FIGURA 2
%
%
{\sc Step 2:} we apply Lemma~\ref{l:cutoff}, again with $r=\rho$,
and we consider the functions $\phi_0, \dots, \phi_k$.  For every $i=1, \dots, k$ we fix a vector 
$\mathbf n_i \in  \mathcal N_{\rho, \theta} (x_i, \Omega)$
and we define the function $ u_{it}$ by setting 
\begin{equation}
\label{e:iacca}
        u_{it} (x) : = u (x + t \mathbf n_i). 
\end{equation}
{\sc Step 3:} finally, we define the function $\tilde u$ by setting 
\begin{equation}
\label{e:uacca}
         \tilde u (x) : = \phi_0 (x) u(x) + \sum_{i=1}^k  \phi_i (x) u_{it} (x).
\end{equation}

%%%%%%%%%%%%%%%%%%%%%%%%%%%%%%%%%%%%%%%%%%%%%%%%%%%%%%%%%%%%%%%%%%%%%%%%%%%%%%%%%%%%%%%%%%%%%%%%%%%%%%%%%%%%%%%%%%%%%%%%%%

\subsubsection{Proof of the inclusion $\tilde u \in \spazio (\Omega^{-\varepsilon})$ }
\label{sss:proofincl}

First, combining Lemma~\ref{l:phiu} with the definition~\eqref{e:uacca} of $\tilde u$, 
we conclude that $\tilde u \in H^s (\R^N).$ Hence, we are left to show that 
\begin{equation}\label{e:zero}
   \tilde u (x)=0 \quad \text{for a.e.\ $x \in \R^N \backslash \Omega^{-\varepsilon}$}. 
\end{equation}
 We fix $x \in \R^N \backslash \Omega^{-\varepsilon}$ and we separately consider two cases. \\
{\sc Case 1:} if $d (x, \Omega) \ge t$, then $x \notin \Omega$ and moreover 
 $(x + t \mathbf{n}_i) \notin \Omega$ for all $i=1,\dots, k$
 because $\mathbf{n}_i$ is a unit vector. Since $u \in \spazio (\Omega)$, 
 then $u \equiv 0$ in $\R^N \backslash \Omega$, whence 
 $$
     0 = u(x) = u(x + t \mathbf n_i) = u_{it} (x).
 $$
 This implies that $\tilde u (x) =0$.  \\
 {\sc Case 2:} we are left to consider the case when $d(x, \Omega) < t$: we have
 \begin{equation} \label{e:bordo}
     x \in \bigcup_{z \in \partial \Omega}  B_{2t} (z) 
       \subseteq \bigcup_{i=1}^k B_\rho (x_i). 
 \end{equation}
 To prove the above inclusion, we have combined the inequality $t < \rho/2$, 
 which follows from~\eqref{eq:cond_h}, with {\sc Step 1} in~\S~\ref{sss:contildeu}.  
 By combining~\eqref{e:bordo} with~\eqref{e:tetazero} we get $\phi_0(x) =0$. 
 Next, we fix $i\in\{1, \dots, k\}$ such that 
 $\phi_i(x) \neq 0$. Owing to~\eqref{e:tetai}, this implies that 
 $x \in B_{2 \rho } (x_i)$. We set $y : = x + t \mathbf{n}_i$. We want
 to show that $y \notin \Omega$. First, 
 we apply  Lemma~\ref{l:pallanelcono} and we conclude that
 \begin{equation}
 \label{e:pallatraslata}
     B_\varepsilon (x ) \subseteq y - C_{\rho, \theta} (\mathbf n_i).
 \end{equation}
 Next, we point out that  $y \in B_{3 \rho} (x_i)$ because $x \in B_{2 \rho } (x_i)$ 
 and $| t |\leq \rho/2$. Hence, we can use property i) in Definition~\ref{d:cone}:
 if $y \in \Omega$, then $y - C_{\rho, \theta} (\mathbf n) \subseteq \Omega$. 
 By recalling~\eqref{e:pallatraslata} and the definition of $\Omega^{-\varepsilon}$ we conclude that, 
 if $y \in \Omega$, then $B_\varepsilon (x) \subseteq \Omega$ and hence $x \in \Omega^{-\varepsilon}$.
 This contradicts our assumption and hence we can conclude that $y \notin \Omega$. This implies 
 $$
     0 = u(y) = u( x + t \mathbf{n}_i) = u_{it} (x),
 $$
 whence follows that $\tilde u(x)=0$. The proof of~\eqref{e:zero} is complete.

%%%%%%%%%%%%%%%%%%%%%%%%%%%%%%%%%%%%%%%%%%%%%%%%%%%%%%%%%%%%%%%%%%%%%%%%%%%%%%%%%%%%%%%%%%%%%%%%%%%%%%%%%%%%%%%%%%%%%%%%%%

\subsubsection{Proof of~\eqref{eq:est_norma2} and~\eqref{eq:est_normas}} 
\label{sss:proofe}

We first establish~\eqref{eq:est_norma2}. We combine
\eqref{e:somma} with~\eqref{e:uacca} and we conclude that, for every $x \in \R^N$,
there holds
\begin{equation}
\label{e:differenza}
     u (x) - \tilde u (x) =
     \sum_{i=1}^k \phi_i (x) \big[ u (x) - u_{it}(x) \big].
\end{equation} 
To control $\| u - \tilde u\|_{L^2 (\R^N)}$ we use~\eqref{eq:key1}. More precisely, we 
first recall~\eqref{e:kappa2} and  conclude that 
\begin{equation}
\label{e:massimo}
   \| u - \tilde u\|_{L^2 (\R^N)} \leq
   k \max_{i=1, \dots, k}  \| \phi_i [ u  - u_{it} ] \big] \|_{L^2 (\R^N)}
   \stackrel{\eqref{e:kappa2}}{\leq} C(N, \mathrm{diam} \, \Omega, \rho) \max_{i=1, \dots, k}
   \| \phi_i [ u  - u_{it} ] \|_{L^2 (\R^N)}.
\end{equation}
Next, we set $h = t \mathbf n_i$ (we do not highlight the dependence of
$h$ on the index $i$, for simplicity) and 
recall the definition~\eqref{eq:local_transl} of $T_h v$. Then we infer that 
$T_h v - v = \phi [v_h - v]$. We now apply Theorem~\ref{t:main} to the function 
$\phi_i [ u  - u_{it} ]$, 
for every $i=1, \dots, k$. The hypotheses of Theorem~\ref{t:main} 
are satisfied because $\phi_i$ is a cut-off function as in the statement of Lemma~\ref{l:cutoff} and
consequently satisfies~\eqref{e:hypphi} and also~\eqref{e:supporto}, 
since $r=\rho \leq 1/2$. Also, 
the analysis in~\S~\ref{sss:proofincl} shows that 
$\phi_i u_{it}\in \spazio (\Omega^{-\eps}) \subseteq \spazio (\Omega)$,
whence condition~\eqref{e:hyprodotto} is also satisfied. By 
combining~\eqref{e:massimo} with~\eqref{eq:key1}
we arrive at the inequality
$$
    \| u - \tilde u\|_{L^2 (\R^N)} 
    \leq C (N, s, \mathrm{diam} \, \Omega, \rho, \theta) 
      t^{ s } \| f \|_{\spazio(\Omega)'}. 
    = C (N, s, \mathrm{diam} \, \Omega, \rho, \theta) 
      \vert h\vert^{ s } \| f \|_{\spazio(\Omega)'}. 
$$
Finally, we point out that the above inequality holds for every 
$t$ satisfying~\eqref{eq:cond_h} and we eventually arrive at~\eqref{eq:est_norma2}.
 
The proof of~\eqref{eq:est_normas} relies on~\eqref{e:bootstrap1} and is entirely analogous. The only new point is that 
we have to use the inequality $\mathrm{Lip} \, \phi \leq C(N,  \rho)$, which follows from~\eqref{eq:gradient}. 
Details are omitted for brevity.

%%%%%%%%%%%%%%%%%%%%%%%%%%%%%%%%%%%%%%%%%%%%%%%%%%%%%%%%%%%%%%%%%%%%%%%%%%%%%%%%%%%%%%%%%%%%%%%%%%%%%%%%%%%%%%%%%%%%%%%%%%

\subsection{Proof of Lemma~\ref{l:hatw}}

%%%%%%%%%%%%%%%%%%%%%%%%%%%%%%%%%%%%%%%%%%%%%%%%%%%%%%%%%%%%%%%%%%%%%%%%%%%%%%%%%%%%%%%%%%%%%%%%%%%%%%%%%%%%%%%%%%%%%%%%%%

\subsubsection{Construction of $\hat w$}

We fix $\Omega$, $\varepsilon$ and $w$ as in the statement of Lemma~\ref{l:hatw}. 
We also fix a number $t$ satisfying~\eqref{eq:cond_h}. We proceed as in {\sc Step 1}
and {\sc Step 2} in~\S~\ref{sss:contildeu} and we define the function $w_{it}$ by setting  
\begin{equation}
\label{e:iacca2}
        w_{it} (x) : = w (x + t \mathbf n_i) . 
\end{equation}
Finally, we define the function $\hat w$ by setting 
\begin{equation}
\label{e:what}
         \hat w (x) : = \phi_0 (x) w(x) + \sum_{i=1}^k  \phi_i (x) w_{it} (x). 
\end{equation}

%%%%%%%%%%%%%%%%%%%%%%%%%%%%%%%%%%%%%%%%%%%%%%%%%%%%%%%%%%%%%%%%%%%%%%%%%%%%%%%%%%%%%%%%%%%%%%%%%%%%%%%%%%%%%%%%%%%%%%%%%%

\subsubsection{Proof of the inclusion $\hat w \in \spazio (\Omega)$}

We combine Lemma~\ref{l:phiu} with the definition~\eqref{e:what} of $\hat w$ and we conclude 
that $\hat w \in H^s (\R^N).$ Hence, we are left to show that 
\begin{equation}
\label{e:fuori}
\hat w (x) = 0 \quad \text{for a.e. $x \in \R^N \backslash \Omega$.}  
\end{equation}
We fix $x \notin \Omega$ and we separately consider two cases.\\
{\sc Case 1:} if $d (x, \Omega^\varepsilon) \ge t$, then $x \notin \Omega^\varepsilon$ and moreover 
$(x + t \mathbf{n}_i) \notin \Omega^\varepsilon$ because
$\mathbf{n}_i$ is a unit vector. Since $w \in \spazio (\Omega^\varepsilon)$, 
then $w \equiv 0$ in $\R^N \backslash \Omega^\varepsilon$, whence 
$$
    0 = w(x) = w(x + t \mathbf n_i) = w_{it} (x).
$$
This implies that $\hat w (x) =0$. \\
{\sc Case 2:} we are left to consider the case when $d(x, \Omega^\varepsilon)< t$. 
By recalling definition~\eqref{e:allargato}, 
this implies $d(x, \Omega)< t + \varepsilon < 2t$.
Thus, we have
$$
    x \in \bigcup_{z \in \partial \Omega}  B_{2t} (z) \subseteq \bigcup_{i=1}^k 
    B_\rho (x_i). 
$$
Combining the above formula with~\eqref{e:tetazero} we deduce that $\phi_0(x) =0$. 
Next, we fix $i\in\{1, \dots, k\}$ such that 
$\phi_i(x) \neq 0$. Owing to~\eqref{e:tetai}, this implies that 
$x \in B_{2 \rho } (x_i)$. We set $y : = x + t \mathbf{n}_i$ and we want to show
that $y \notin \Omega^\varepsilon$. 
Since $x \in B_{2 \rho } (x_i)$, we can use property ii) in Definition~\ref{d:cone}: 
since $x \notin \Omega$, then $x + C_{\rho, \theta} (\mathbf{n}_i) \subseteq \R^N \setminus \Omega$. 
Next, we apply Lemma~\ref{l:pallanelcono} and we conclude that
$$
   B_\varepsilon (y ) \subseteq x + C_{\rho, \theta} (\mathbf n)  \subseteq \R^N \setminus \Omega.
$$
This means that $d (y,  \Omega) \ge \varepsilon$ and hence that $y \notin \Omega^\varepsilon$.  
Consequently we have
$$
   0 = w(y) = w( x + t \mathbf{n}_i) = w_{it} (x),
$$
whence we obtain $\hat w(x)=0$. The proof of~\eqref{e:fuori} is complete.

%%%%%%%%%%%%%%%%%%%%%%%%%%%%%%%%%%%%%%%%%%%%%%%%%%%%%%%%%%%%%%%%%%%%%%%%%%%%%%%%%%%%%%%%%%%%%%%%%%%%%%%%%%%%%%%%%%%%%%%%%%

\subsubsection{Proof of~\eqref{eq:est_norma2w} and~\eqref{eq:est_normasw}}

We proceed as in~\S~\ref{sss:proofe} and we  apply estimates~\eqref{eq:key1} and~\eqref{e:bootstrap1} 
in the domain $\Omega^\eps$. The details are omitted.

%%%%%%%%%%%%%%%%%%%%%%%%%%%%%%%%%%%%%%%%%%%%%%%%%%%%%%%%%%%%%%%%%%%%%%%%%%%%%%%%%%%%%%%%%%%%%%%%%%%%%%%%%%%%%%%%%%%%%%%%%%
%
%     NEW SECTION
% 
%%%%%%%%%%%%%%%%%%%%%%%%%%%%%%%%%%%%%%%%%%%%%%%%%%%%%%%%%%%%%%%%%%%%%%%%%%%%%%%%%%%%%%%%%%%%%%%%%%%%%%%%%%%%%%%%%%%%%%%%%%

\section{Domain perturbation estimates}
\label{ss:proof_d_perturbation}

This section aims at establishing the following result, which can be regarded as a 
weaker version of Theorem~\ref{th:domain_perturb}:
\begin{lemma}
\label{l:bootdp}
Under the same assumptions as in the statement of Theorem~\ref{th:domain_perturb}, for every $\sigma \in (0,1)$ we have 
\begin{equation}
\label{e:bootdp}
\| u_a - u_b\|_{s}
   \le C(N, s, \mathrm{diam} \, \Omega, \rho, \theta, \sigma)  
   \| f \|_{L^2(D)}^{1/2} \| f \|_{\Xzs(D)'}^{1/2} 
    \dd(\Omega_b,\Omega_a)^{s \sigma/2}.
    \end{equation}
\end{lemma}

%%%%%%%%%%%%%%%%%%%%%%%%%%%%%%%%%%%%%%%%%%%%%%%%%%%%%%%%%%%%%%%%%%%%%%%%%%%%%%%%%%%%%%%%%%%%%%%%%%%%%%%%%%%%%%%%%%%%%%%%%%

\subsection{Notation and preliminary results}

Let $D$, $\Omega_a$ and $\Omega_b$ be as in the statement of Theorem~\ref{th:domain_perturb}.  We recall that 
we term $u_a $ and $u_b$ the solutions of~\eqref{eq:DP}
when $\Omega = \Omega_a$ and $\Omega = \Omega_b$, respectively. Also, we recall 
that the sets $\Omega^{-\varepsilon}$ and $\Omega^\varepsilon$
are defined as in~\eqref{e:ristretto} and~\eqref{e:allargato}, respectively,
and we term $u^{-\varepsilon}$ and $u^{\varepsilon}$ the solutions
of the Poisson problem~\eqref{eq:DP} when $\Omega = \Omega^{-\varepsilon}$ 
and $\Omega = \Omega^\varepsilon$, respectively. 

We introduce some additional notation. Given two bounded subdomains
$\Omega$ and $\tilde \Omega$ of $D$
with $\Omega\subseteq\tilde\Omega$,
we denote with $P_{\tilde\Omega\to\Omega}$
the orthogonal projection
\begin{equation}
\label{eq:proj}
    P_{\tilde\Omega\to\Omega}: \spazio (\tilde\Omega) \to \spazio (\Omega)
\end{equation}
with respect to the scalar product~\eqref{e:scalarpoincare}. Namely, 
for $u\in \Xzs(\tilde\Omega)$, this is characterized by
$$
  [u-P_{\tilde\Omega\to\Omega}u,v]_{s} = 0
    \ \  \forall\, v\in \Xzs(\Omega).
$$
Recall also that
$$
    \| w - P_{\tilde\Omega\to\Omega} (w) \|_s = \min_{v \in \spazio (\Omega) }  \| w - v\|_s
$$
for $w \in \Xzs(\tilde\Omega)$. 
We have the following simple, albeit important,
property:
\begin{lemma}\label{l:proiezioni}
 Assume that $\Omega_a \subseteq \Omega$ and that $u_a$ and $u$ solve \eqref{eq:DP}
 respectively in $\Omega_a$ and in $\Omega$. 
 Then
 $$
     P_{\Omega\to \Omega_a} (u) = u_a,
 $$
 and $P_{\Omega \to \Omega_a}$ is linear. 
\end{lemma}
\begin{proof}
 Since $\spazio(\Omega_a)\subset \spazio(\Omega)$ and since $u$
 and $u_a$ are weak solutions of
 \eqref{eq:DP} in $\Omega$ and in $\Omega_a$, respectively,
 we have, for all $v\in \spazio(\Omega_a)$,
 $$
   [u,v]_s = \langle f,v \rangle = [u_a,v]_s,
 $$
 whence $[u-u_a,v]_s=0$ for all $v\in \spazio(\Omega_a)$, that is the thesis.
\end{proof}

%%%%%%%%%%%%%%%%%%%%%%%%%%%%%%%%%%%%%%%%%%%%%%%%%%%%%%%%%%%%%%%%%%%%%%%%%%%%%%%%%%%%%%%%%%%%%%%%%%%%%%%%%%%%%%%%%%%%%%%%%%

\subsection{Proof of Lemma~\ref{l:bootdp}: conclusion}

First, we fix $\varepsilon >0$ such that 
\begin{equation}
\label{e:inmezzo}
    \dd(\Omega_b, \Omega_a)  < \eps < \frac{\rho \sin \theta}{2}. 
\end{equation}
We recall~\eqref{intorni} and we conclude that 
\begin{equation}\label{e:dentro}
   \Omega_a^{- \varepsilon} \subseteq \Omega_b \subseteq \Omega_a^\varepsilon.
\end{equation}
By using Lemma~\ref{l:proiezioni}, we have 
$$
  u_b = P_{\Omega_{a}^{\eps}\to \Omega_b}(u_{a}^{\eps}),
$$
where $u^\varepsilon_a$ denotes the weak solution of \eqref{eq:DP} 
in $\Omega_a^\varepsilon$.
Hence we obtain the following chain of inequalities:
 \begin{equation} \label{chain11}
  \|u^\varepsilon_a - u_b\|_s = 
    \min_{v \in \spazio (\Omega_b) }
     \|u^\varepsilon_a - v\|_s
     \leq  \|u^\varepsilon_a - u^{-\varepsilon}_a\|_s \leq 
     \|u^\varepsilon_a - u_a\|_s + \|u_a-  u^{-\varepsilon}_a\|_s.
\end{equation}
Note that 
to establish the first inequality we used the inclusion 
$\spazio (\Omega_a^{-\varepsilon}) \subseteq \spazio (\Omega_b)$
following from \eqref{e:dentro}.
By using~\eqref{chain11} we infer 
\begin{equation}
\label{e:chain2}
\begin{split}
    \|u_a - u_b\|_s \leq \|u_a - u^\varepsilon_a\|_s +  
    \|u^\varepsilon_a - u_b\|_s
    \leq 2 \|u_a - u^\varepsilon_a\|_s +  
    \|u_a - u^{-\varepsilon}_a \|_s. 
\end{split}
\end{equation}
Applying again Lemma~\ref{l:proiezioni}, we deduce
$$
     \|u_a - u^{-\varepsilon}_a \|_s = \min_{v \in \spazio (\Omega_a^{-\eps})} 
     \| u_a - v \|_s. 
$$
By using Lemma~\ref{l:tildeu} we conclude that  
\begin{equation} \label{e:lemmasetteuno}
   \begin{split}
   \|u_a - u^{-\varepsilon}_a \|_s 
   & \leq C(N, s, \mathrm{diam} \, \Omega, \rho, \theta, \sigma) 
   \, \eps^{\sigma s/2} \| f \|^{1/2}_{L^2 (D)}  \| f \|^{1/2}_{\spazio (\Omega_a)'} \\
   & \leq  C(N, s, \mathrm{diam} \, \Omega, \rho, \theta, \sigma ) \, 
   \eps^{\sigma s/2} \| f \|^{1/2}_{L^2 (D)}  \| f \|^{1/2}_{\spazio (D)'}.  \\
   \end{split}
\end{equation}
In the previous estimate we used the inequality $\| f \|_{\spazio (\Omega_a)'} \leq \| f \|_{\spazio (D)'}$, which holds  
for $f \in L^2(\RN)$ and can be established by arguing as follows. From the inclusion $\Omega_a \subset D$ we infer that, for every $v \in \spazio (\Omega_a)$,  $\|v\|_{\spazio (D)} = \|v\|_{\spazio (\Omega_a)}$. This implies 
$$
 \| f \|_{\spazio (\Omega_a)'}
 = \sup_{v \in \spazio (\Omega_a)} \dfrac{\int_{\RN} f(x) v(x) \, \d x}{\|v\|_{\spazio (\Omega_a)}}
 \leq \sup_{v \in \spazio (\Omega_a)} \dfrac{\|f\|_{\spazio (D)'}  \|v\|_{\spazio (D)}}{\|v\|_{\spazio (\Omega_a)}}
 = \|f\|_{\spazio (D)'}.
$$  
%
%(LASCIARE LE DUE RIGHE SOPRA? IN FONDO E' OVVIO. QUI HO ANCHE SOSTITUITO 
%LA NORMA DI $f$ IN $L^2(\Omega_a)$ CON LA NORMA IN $L^2(D)$. QUESTO E'
%IN ACCORDO COL CAMBIAMENTO CHE HO FATTO IN \eqref{e:bootstrap1} DOVE HO
%PREFERITO CAMBIARE $\Omega$ CON $\R^N$. MI PARE INFATTI CHIARO DALLA DIMOSTRAZIONE
%CHE IN \eqref{e:bootstrap1} VADA LA NORMA IN $L^2(\R^N)$. POI NATURALMENTE
%SE $f$ E' NULLA (O SI PUO' SUPPORRE NULLA) FUORI DA $D$, SCRIVERE $\R^N$
%O $D$ E' LO STESSO).
%
By applying once more Lemma~\ref{l:proiezioni} we get 
$$
     \| u^{\varepsilon}_a -u_a \|_s = \min_{v \in \spazio (\Omega_a)} 
     \| u^{\varepsilon}_a - v \|_s, 
$$
which combined with  Lemma~\ref{l:hatw} gives 
\begin{align}\label{e:lemmasettedue}
   \|u^\varepsilon_a - u_a \|_s & \leq C(N, s, \mathrm{diam} \, \Omega, \rho, \theta, \sigma) 
   \, \eps^{\sigma s/2} \| f \|^{1/2}_{L^2 (D)}  \| f \|^{1/2}_{\spazio (\Omega^\eps_a)'} \\
 \nonumber
  & \leq  C(N, s, \mathrm{diam} \, \Omega, \rho, \theta, \sigma ) \, 
   \eps^{\sigma s/2} \| f \|^{1/2}_{L^2 (D)}  \| f \|^{1/2}_{\spazio (D)'}.
\end{align}
By plugging~\eqref{e:lemmasetteuno} and~\eqref{e:lemmasettedue} into~\eqref{e:chain2} we 
arrive at
$$
  \|u_a - u_b\|_s \leq C(N, s, \mathrm{diam} \, \Omega, \rho, \theta, \sigma ) \, 
   \eps^{\sigma s/2} \| f \|^{1/2}_{L^2 (D)}  \| f \|^{1/2}_{\spazio (D)'} .
$$
We recall that the above inequality holds for every $\eps$ satisfying~\eqref{e:inmezzo} 
and we conclude the proof of~\eqref{e:bootdp}.

%%%%%%%%%%%%%%%%%%%%%%%%%%%%%%%%%%%%%%%%%%%%%%%%%%%%%%%%%%%%%%%%%%%%%%%%%%%%%%%%%%%%%%%%%%%%%%%%%%%%%%%%%%%%%%%%%%%%%%%%%%
%
%     NEW SECTION
% 
%%%%%%%%%%%%%%%%%%%%%%%%%%%%%%%%%%%%%%%%%%%%%%%%%%%%%%%%%%%%%%%%%%%%%%%%%%%%%%%%%%%%%%%%%%%%%%%%%%%%%%%%%%%%%%%%%%%%%%%%%%

\section{Regularity estimates}
\label{ss:proof_regularity}

In this section we establish the following result, which can be be 
regarded as a weaker version of Theorem~\ref{th:regularity}:
\begin{lemma}
\label{l:bootreg}
 Under the same assumptions as in the statement of Theorem~\ref{th:regularity}, 
 for every $\sigma \in (0, 1)$ we have 
 \begin{equation}
 \label{e:bootreg}
        u \in B^{3 \sigma s/2}_{2, \infty}(\R^N), \qquad 
        \| u \|_{B^{3 \sigma s/2}_{2, \infty}(\R^N)}
        \leq  C(N, s, \mathrm{diam} \, \Omega, \rho, \theta, \sigma)  
      \|f \|_{H^{-s} (\R^N) }^{1/2}
    \|f   \|_{L^2(\RN)}^{1/2}. 
 \end{equation}
\end{lemma}
The proof is based on an argument similar to that given
in the proof of~\cite[Proposition 2.3]{SS02} combined  
with the use of Lemma~\ref{lemma:besov} below.

%%%%%%%%%%%%%%%%%%%%%%%%%%%%%%%%%%%%%%%%%%%%%%%%%%%%%%%%%%%%%%%%%%%%%%%%%%%%%%%%%%%%%%%%%%%%%%%%%%%%%%%%%%%%%%

\subsection{Preliminary results}

The following result is classical, but we provide a proof for the sake 
of completeness and for future reference. 
\begin{lemma}%[Besov Regularity]
\label{lemma:besov}
 Assume that $v \in L^2(\RN)$ satisfies 
  \begin{equation}
  \label{eq:poisson_wholespace}
  \Ds v = g\,\,\,\hbox{ in } \RN,
  \end{equation}
 If $g\in B^{r}_{2, \infty}(\R^N)$ for some $r > 0$, then 
 \begin{equation}
 \label{eq:reg_besov_wholespace}
  v \in B^{r+2s}_{2, \infty} ( \R^N) \quad 
  \text{and $\|v\|_{ B^{r+2s}_{2, \infty} (\R^N)} \le C(N, s, r) 
 \big[ \| v \|_{L^2 (\R^N)} + \|g\|_{B^r_{2, \infty} (\R^N)} \big]$}.
 \end{equation}
\end{lemma}
\begin{proof}
The basic idea of the proof can be outlined as follows: first, 
we observe that~\eqref{eq:poisson_wholespace} implies that 
\begin{equation}
\label{e:b:poissonidentita}
      v + \Ds v = \ell: = v+  g \,\,\,\hbox{ in } \RN.    
\end{equation}
Next, by using the Fourier transform, we show that the regularity 
properties of the above equation are basically the same as those 
of the equation~\eqref{e:b:stein}. Finally, we apply Lemma~\ref{l:stein} 
and we conclude.  

The details of this procedure are organized into a number of steps. \\
{\sc Step 1:} we show that $v \in B^r_{2, \infty} (\R^N)$. 

First of all, we show that $v$ has some fractional Sobolev regularity. 
More precisely, we fix $\varepsilon = \min \{ r, s \}$, 
and we show that $v \in H^{r + 2s - \varepsilon} (\R^N)$ and that 
\begin{equation}
\label{e:b:accaesse}
     \| v \|_{H^{r +2s-\varepsilon} (\R^N)} \leq C(N, s, r) \big[ \| v \|_{L^2 (\R^N)} +
     \| g \|_{B^r_{2, \infty} (\R^N)} \big].
\end{equation}
To establish~\eqref{e:b:accaesse}, we first use the inclusion 
property~\eqref{e:b:inclusion} and we conclude that 
 ${g \in H^{r-\varepsilon} (\R^N)}$ and that 
 \begin{equation}
 \label{e:b:ineq1}
     \| g \|_{H^{r-\varepsilon} (\R^N)}  
   \stackrel{\eqref{e:b:inclusion}}{\leq}  C(N, s, r) \| g \|_{B^r_{2, \infty} (\R^N)}.
 \end{equation}
Next, we point out that proving that $v \in H^{r + 2s-\varepsilon} (\R^N)$ amounts to show that 
$$
  (1 + |\xi|^2)^{(r + 2s-\varepsilon)/2} \hat v \in L^2 (\R^N). 
$$
We recall~\eqref{def:fract_laplbis} and we infer the following chain of equalities:
\begin{equation}
\label{e:b:chain}
\begin{split}
         (1 + |\xi|^2)^{(r + 2s-\varepsilon)/2} | \hat v |
         & = \frac{(1 + |\xi|^2)^{(r + 2s-\varepsilon)/2} }{ 1 + |\xi|^{r + 2s-\varepsilon}} 
    (1 + |\xi|^{r + 2s-\varepsilon} )| \hat v |  \stackrel{\eqref{def:fract_laplbis}}{=} 
         \frac{(1 + |\xi|^2)^{(r + 2s-\varepsilon)/2} }{ 1 + |\xi|^{r + 2s-\varepsilon}}  
    ( |\hat v| + |\xi|^{r-\varepsilon} |\hat g| ) \\
         &
         \leq  \frac{(1 + |\xi|^2)^{(r + 2s-\varepsilon)/2} }{ 1 + |\xi|^{r + 2s-\varepsilon}} \Big[ 
         |\hat v| + (1 + |\xi|^2)^{(r-\varepsilon)/2} |\hat g| \Big] . 
\end{split}
\end{equation}
Next, we recall that $\varepsilon = \min \{ r, s \}$ and, since 
$$
    \left| \frac{(1 + |\xi|^2)^{(r + 2s- \varepsilon )/2} }{ 1 + |\xi|^{r + 2s- \varepsilon }} \right| 
   \leq C(N, s, \varepsilon)= C(N, s, r) 
    \quad \text{for every $\xi \in \R^N$},
$$
then by combining~\eqref{e:b:ineq1} and~\eqref{e:b:chain} we conclude that 
 $v \in H^{r + 2s-\varepsilon} (\R^N)$ and that the inequality~\eqref{e:b:accaesse} is satisfied. \\
We now turn to the proof of the Besov regularity of $v$.  We recall that 
$\varepsilon = \min \{ r, s \}$ and, owing to~\eqref{e:b:inclusion2}, we conclude that 
$$
    v \in H^{r + 2s-\varepsilon} (\R^N) \subset H^r (\R^N)  
     \stackrel{\eqref{e:b:inclusion}}{\subset}  B^r_{2, \infty} (\R^N)
$$
and, by using~\eqref{e:b:accaesse}, that 
\begin{equation}
\label{e:b:stepdue}
    \| v \|_{B^r_{2, \infty} (\R^N)}  \stackrel{\eqref{e:b:inclusion}}{\leq}  C(N, s, r) 
    \| v \|_{H^{r + 2s-\varepsilon} (\R^N) }
    \stackrel{\eqref{e:b:accaesse}}{\leq} C(N, s, r) 
         \big[ \| v \|_{L^2 (\R^N)} +  \| g \|_{B^r_{2, \infty} (\R^N)} \big].
\end{equation}
{\sc Step 2:} we conclude the proof of the lemma in the case when $r + 2s \leq 1$.
First, we point out that by using~\eqref{def:fract_laplbis} again we infer from~\eqref{e:b:poissonidentita} the equality
\begin{equation}
\label{e:b:fourier}
     \hat v (\xi)+ |\xi|^{2s} \hat v (\xi) = \hat \ell(\xi) \quad \text{for a.e. $\xi \in \R^N$}.
\end{equation}
Note that, owing to~\eqref{e:b:stepdue}, $\ell = v + g \in B^r_{2, \infty} (\R^N)$ 
and 
\begin{equation}
\label{e:steptre}
   \| \ell \|_{B^r_{2, \infty} (\R^N)}
   \leq  C(N, s, r) \big[ \| v \|_{L^2 (\R^N)} + \| g \|_{B^r_{2, \infty} (\R^N)} \big]. 
\end{equation}
Since by assumption $v \in L^2 (\R^N)$, owing to the Plancherel Theorem 
and to definition~\eqref{e:b:definition}, proving that 
$v \in B^{r+2s}_{2, \infty} (\R^N)$ amounts to show that 
$$
\sup_{h \in \R^N \setminus \{0 \}}
\frac{\| \hat v_{2h} -2 \hat v_h+  \hat v \|_{L^2(\R^N  )}}{|h|^{r+2s}} < + \infty. 
$$
By directly computing $\hat{v}_h$ and using~\eqref{e:b:fourier} we obtain 
\begin{equation}
\begin{split}
\label{e:b:trasformate}
      \hat v_{2h}(\xi) -2 \hat v_h(\xi)+  \hat v (\xi)  & =
       (e^{i 2 \xi \cdot h} -2 e^{i \xi \cdot h} + 1) \hat v (\xi) \\
       &\stackrel{\eqref{e:b:fourier}}{=}
       \frac{e^{i 2 \xi \cdot h} -2 e^{i \xi \cdot h} + 1 }{1 + |\xi|^{2s}} \hat \ell (\xi)=
        \frac{ (1+ |\xi|^2)^{s}}{1 + |\xi|^{2s}} 
        \frac{1}{ (1+ |\xi|^2)^{s}} (e^{i 2 \xi \cdot h} -2 e^{i \xi \cdot h} + 1) \hat \ell (\xi).\\
        \end{split}
        \end{equation}
Owing to~\eqref{e:bessel}, 
\begin{equation}
\label{e:b:fourier2}
      \frac{1}{ (1+ |\xi|^2)^{s}} (e^{i 2 \xi \cdot h} -2 e^{i \xi \cdot h} + 1) \hat \ell (\xi)=
      \hat u_{2h} (\xi) - 2 \hat u_h (\xi) + \hat u(\xi)
\end{equation}
provided that $u$ solves the equation 
\begin{equation}
\label{e:b:stein:elle}
    (I - \Delta)^s u= \ell \quad \text{in $\R^N$.}
\end{equation}
Owing to Lemma~\ref{l:stein}, since $\ell \in B^{r}_{2, \infty} (\R^N)$, then 
$u \in B^{r + 2s}_{2, \infty} (\R^N)$. Moreover,
\begin{equation}
\label{e:b:fourier3}
\begin{split}
      \| \hat u \|_{L^2(\R^N)}
      + \sup_{h \in \R^N \setminus \{0 \}}
           \frac{\| \hat u_{2h} - 2 \hat u_h+  \hat u \|_{L^2(\R^N  )}}{|h|^{r+2s}} &
      = 
      \| u \|_{B^{r + 2s}_{2, \infty} (\R^N)} \\ &
      \stackrel{\eqref{e:stein}}{\leq} C(N, s, r) \| \ell \|_{B^{r}_{2, \infty} (\R^N)}  \\ &
      \stackrel{\eqref{e:steptre}}{\leq} C(N, s, r) \big[ \| v \|_{L^2 (\R^N)} + \| g \|_{B^r_{2, \infty} (\R^N)} \big].\phantom{\int}
      \\
      \end{split}
\end{equation}
To conclude, we point out that 
\begin{equation}
\label{e:b:bound}
    \frac{ (1+ |\xi|^2)^{s}}{1 + |\xi|^{2s}} 
    \leq 1 \quad \text{for every $\xi \in \R^N$}.
\end{equation}
By combining~\eqref{e:b:norm},~\eqref{e:b:trasformate},~\eqref{e:b:fourier2},
\eqref{e:b:fourier3} and~\eqref{e:b:bound} we eventually arrive at~\eqref{eq:reg_besov_wholespace}. \\
{\sc Step 3:} we conclude the proof by dealing with the case when $r + 2s >1$. 
We recall~\eqref{e:b:higher}, we fix $j =1, \dots N$ and we term $w$ the distributional derivative  
$$
    w : = \frac{\partial v}{\partial x_{ j }}.
$$
Next, we point out that 
$$
    \hat w_{2h} (\xi) - 2 \hat w_h(\xi) + \hat w(\xi) = i \xi_{ j } 
      (e^{i 2 \xi \cdot h} -2 e^{i \xi \cdot h} + 1 ) \hat v(\xi)
$$  
and by arguing as in~\eqref{e:b:trasformate} and~\eqref{e:b:fourier2} we conclude that 
$$
    \hat w_{2h} (\xi) - 2 \hat w_h(\xi) + \hat w(\xi) = 
      \frac{ (1+ |\xi|^2)^{s}}{1 + |\xi|^{2s}} 
      \left( \hat z_{2h} (\xi) - 2 \hat z_h(\xi) + \hat z(\xi) \right),
$$
provided that 
$$
     z =  \frac{\partial u}{\partial x_{ j }}
$$
and $u$ solves~\eqref{e:b:stein:elle}. 

If $1< r + 2s \leq 2$, then by following the same argument as in {\sc Step 2} we conclude the proof of the lemma. 

If $r + 2s > 2$ we iterate the above argument and we eventually arrive at~\eqref{eq:reg_besov_wholespace}. 
\end{proof}

%%%%%%%%%%%%%%%%%%%%%%%%%%%%%%%%%%%%%%%%%%%%%%%%%%%%%%%%%%%%%%%%%%%%%%%%%%%%%%%%%%%%%%%%%%%%%%%%%%%%%%%%%%%%%%%%%%%%%%%%%%

\subsection{Proof of Lemma~\ref{l:bootreg}}

We fix $f\in L^2(\RN)$ and $h \in \R^N$. As usual we term $u$ the weak solution of~\eqref{eq:DP} and we define 
the functions $u_h$ and $f_h$ as in~\eqref{eq:local_transl}. 
We have now all the ingredients
required to prove Lemma~\ref{l:bootreg}. 
Owing to the translation invariance of the fractional
Laplacian,  $u_h \in \spazio (\Omega -h)$ is the weak solution of 
\begin{equation*}
 \begin{cases}
  \Ds u_h = f_h \,\,\,\hbox{ in } \,\Omega-h,\\
   u_h = 0 \,\,\,\hbox{ in } \,\RN\setminus (\Omega-h). 
 \end{cases}
\end{equation*}
Here and in the following we use the notation 
$$
   \Omega - h : = \big\{ x \in \R^N: x + h \in \Omega \big\}.  
$$
Note that, if $|h|$ is sufficiently small (which is not restrictive 
for our purposes, as it will be clear in the following),  then
\begin{equation}\label{eq:haus}
  \dd (\Omega-h, \Omega)
   = e( \Omega-h, \Omega ) + e^c (\Omega , \Omega-h) \leq 2 |h| \leq \frac{\rho \sin \theta}{2}
   .
\end{equation}
We term $v_h$ the weak solution
of 
\begin{equation*}
 \begin{cases}
  \Ds v_h = f_h \,\,\,\hbox{ in } \,\Omega,\\
    v_h = 0 \,\,\,\hbox{ in } \,\RN\setminus \Omega. 
 \end{cases}
\end{equation*}
Owing to~\eqref{eq:haus}, the sets $\Omega_a = \Omega$ and 
$\Omega_b = \Omega - h $ satisfy~\eqref{eq:hypo_dist}. By 
applying Lemma~\ref{l:bootdp}, we conclude that for every $\sigma \in (0, 1)$ we have 
\begin{equation}\label{eq:dependence_h}
  \| u_h - v_h\|_{s}\le C(N, s, \mathrm{diam} \, \Omega, \rho, \theta, \sigma)  
    \|f_{h}\|_{H^{-s}(\R^N)}^{1/2}
   \|f_h\|_{L^2(\RN)}^{1/2}\vert h\vert^{\sigma s/2}. 
\end{equation}
Next, we consider the function $w:=u-v_h$, which satisfies  
$w\in \spazio (\Omega)$. 
Moreover, by linearity, $w$ is the weak solution of 
\begin{equation*}
 \begin{cases}
  \Ds w = f- f_h \,\,\,\hbox{ in }\, \Omega,\\
   w = 0 \,\,\,\hbox{ in }\, \RN\setminus \Omega.
 \end{cases}
\end{equation*}
Using~\eqref{eq:lax_milgram}, we infer 
\begin{align} \label{eq:LM_difference}
  \| u - v_h\|_{s}
      \stackrel{\eqref{eq:lax_milgram}}{\leq}  C(N, s, \mathrm{diam} \, \Omega) \|f- f_h\|_{\spazio(\Omega)'} 
       \stackrel{\eqref{e:embeddingchian}}{\leq}  
     C(N, s, \mathrm{diam} \, \Omega) \|f- f_h\|_{H^{-s} (\R^N)}.  
\end{align}
Next, we control $\|f- f_h\|_{H^{-s} (\R^N)}$. 
We first fix $R>0$ (to be determined later) and we point out that 
\begin{align}\label{eq:trans1}
  &  \|f- f_h\|^2_{H^{-s}(\RN)}  = C(N)
    \int_{\RN}(1+ |\xi|^2 )^{-s} \vert 1 - e^{i\xi\cdot h}\vert^2 \vert \hat f\vert^2 \,\hbox{d}\xi \\
 \nonumber   
  & = C(N) \Big( \underbrace{\int_{|\xi|\leq R}(1+ |\xi|^2 )^{-s} \vert 1 - e^{i\xi\cdot h}\vert^2 \vert \hat f\vert^2 \,\hbox{d}\xi}_{I_1} + 
    \underbrace{\int_{|\xi| > R}(1+ |\xi|^2 )^{-s} \vert 1 - e^{i\xi\cdot h}\vert^2 \vert \hat f\vert^2 \,\hbox{d}\xi}_{I_2} \Big). 
\end{align}
Next, we introduce the decomposition 
\begin{equation}
\label{e:fourier:split}
   \vert 1 - e^{i\xi\cdot h}\vert^2 = 
   \vert 1 - e^{i\xi\cdot h}\vert^{2-s} \vert 1 - e^{i\xi\cdot h}\vert^{s}
    \stackrel{\eqref{elem:11},\eqref{elem:12}}{\leq}  8 \vert\xi\vert^{s}  \vert h \vert^{s},
\end{equation}
which gives 
\begin{equation}
\label{e:primopezzo}
\begin{split}
    I_1 &  \stackrel{\eqref{e:fourier:split}}{\leq}  C(N) \int_{|\xi|\leq R}
    (1+ |\xi|^2 )^{-s}
    \vert\xi\vert^{s}  \vert h \vert^{s}  \vert \hat f\vert^2 \,\hbox{d}\xi
    \leq C(N) R^s  |h|^s \int_{\R^N} (1+ |\xi|^2 )^{-s}  |\hat f |^2 \hbox{d}\xi \\
    &
      \; \; \leq C(N) R^s  |h|^s \| f\|^2_{H^{-s} (\R^N)}. \\
      \end{split}
\end{equation}
On the other hand, $(1+ |\xi|^2)^{-s} \leq |\xi|^{-2s}$, whence 
\begin{equation}
\label{e:secondopezzo}
  I_2  \stackrel{\eqref{e:fourier:split}}{\leq}
   C(N) \int_{|\xi| > R}|\xi|^{-2s} | \xi|^s |h |^s \vert \hat f\vert^2 \,\hbox{d}\xi 
   = C(N) \int_{|\xi| > R}|\xi|^{-s} | h|^s \vert \hat f\vert^2 \,\hbox{d}\xi \leq 
    R^{-s} |h|^s  \| f \|^2_{L^2 (\R^N)}. 
\end{equation}
By choosing $R$ in such a way that 
$R^s = \| f\|_{L^2(\RN)} / \| f\|_{H^{-s}(\RN)}$, plugging this equality into~\eqref{e:primopezzo} and~\eqref{e:secondopezzo}  and by recalling~\eqref{eq:trans1} 
we eventually get
\begin{equation}
\label{eq:translation_s} 
  \|f - f_h\|_{H^{-s}(\RN)}\le C (N) \|f\|_{L^2(\RN)}^{1/2}\|f\|_{H^{-s}(\R^N)}^{1/2} 
\vert h\vert^{s/2}.
\end{equation}
By combining~\eqref{eq:dependence_h},~\eqref{eq:LM_difference} and~\eqref{eq:translation_s}
we arrive at
\begin{equation}\label{eq:regularity1}
\begin{split}
  \| u - u_h\|_{s} & \le
   \| u - v_{h}\|_s + \| v_h - u_h\|_s \\
  &   \leq  C(N, s, \mathrm{diam} \, \Omega) \|f\|_{L^2(\RN)}^{1/2} \|f\|_{H^{-s}(\RN)}^{1/2} \vert h\vert^{s/2} \\
  & \mbox{}~~~~~
  + C(N, s, \mathrm{diam} \, \Omega, \rho, \theta, \sigma)  
    \|f_{h}\|_{H^{-s}(\R^N)}^{1/2}
   \|f_h\|_{L^2(\RN)}^{1/2}\vert h\vert^{\sigma s/2} \\
  & 
    \leq
    C(N, s, \mathrm{diam} \, \Omega, \rho, \theta, \sigma)  
    \|f \|_{H^{-s} (\R^N) }^{1/2}
   \|f \|_{L^2(\RN)}^{1/2}
   \big[ \vert h\vert^{\sigma s/2}  + |h|^{ s/2 } \big] \\ &  \leq 
   C(N, s, \mathrm{diam} \, \Omega, \rho, \theta, \sigma)  
    \|f \|_{H^{-s} (\R^N) }^{1/2}
   \|f \|_{L^2(\RN)}^{1/2}
   \vert h\vert^{\sigma s/2},
   \end{split}
\end{equation}
where to establish the last inequality we used that $|h| \leq1$ and $\sigma \in (0, 1)$.

We now set $z:=\Dsm u$ and we point out that, thanks to \eqref{eq:semi_norm},
\eqref{eq:regularity1} implies  
$$
  \| z - z_h\|_{L^2(\RN)} \le C(N, s, \mathrm{diam} \, \Omega, \rho, \theta, \sigma)  
    \|f \|_{H^{-s} (\R^N) }^{1/2}
   \|f \|_{L^2(\RN)}^{1/2}
   \vert h\vert^{\sigma s/2}.
$$
We point that $\sigma s /2 \in (0, 1)$ since $s, \sigma \in (0, 1)$ and we recall that in this case the Besov
norm can be characterized as in~\eqref{e:b:equivalent}. We conclude that the above inequality implies
\begin{equation}
\label{e:b:zeta}
  z \in B^{\sigma s/2}_{2,\infty} (\RN), \quad \| z \|_{B^{\sigma s/2}_{2,\infty} (\RN) }
  \leq  \| z \|_{L^2 (\R^N)} + C(N, s, \mathrm{diam} \, \Omega, \rho, \theta, \sigma)  
    \|f \|_{H^{-s} (\R^N) }^{1/2}
   \| f \|_{L^2(\RN)}^{1/2} .
\end{equation}
Note that 
\begin{equation*}
\begin{split}
     \| z \|_{L^2 (\R^N)} = C(N,s) \| u \|_s & \stackrel{\eqref{eq:lax_milgram}}{\leq}
     C(N, s, \mathrm{diam} \, \Omega) \| f \|_{\spazio (\Omega)'}
      \stackrel{\eqref{e:embeddingchian}}{\leq}
       C(N, s, \mathrm{diam} \, \Omega) \| f \|_{H^{-s} (\R^N)} \\
     & \stackrel{\eqref{e:embeddingchian}}{\leq}
      C(N, s, \mathrm{diam} \, \Omega) \| f \|^{1/2}_{H^{-s} (\R^N)}
      \|f \|_{L^2(\RN)}^{1/2},
      \end{split}
\end{equation*}
whence from~\eqref{e:b:zeta} we infer  
\begin{equation}
\label{e:b:zeta2}
    \| z \|_{B^{\sigma s/2}_{2,\infty} (\RN) }
  \leq  C(N, s, \mathrm{diam} \, \Omega, \rho, \theta, \sigma)  
    \|f \|_{H^{-s} (\R^N) }^{1/2}
   \|f \|_{L^2(\RN)}^{1/2}. 
\end{equation}
Finally, we recall that $z:=\Dsm u$  and we apply 
 Lemma~\ref{lemma:besov}. We conclude
that 
\begin{align}
\label{e:spaziosharp}
  u \in B^{( \sigma s/2 )+ s}_{2,\infty} (\RN) , \quad \| u\|_{B^{(\sigma s/2) + s}_{2,\infty} (\RN) }
  & \stackrel{\eqref{eq:reg_besov_wholespace}}{\leq} C(N,s,\sigma) (\|u\|_{L^2(\RN)} + \| z \|_{B^{\sigma s/2}_{2,\infty}  (\RN) }) \\
  & \stackrel{\eqref{e:b:zeta2}}{\leq} C(N, s, \mathrm{diam} \, \Omega, \rho, \theta, \sigma)  
    \|f \|_{H^{-s} (\R^N) }^{1/2}
   \|f   \|_{L^2(\RN)}^{1/2}. 
 \nonumber
\end{align}
Since $\sigma \in (0, 1)$, then  $B^{( \sigma s/2 )+ s}_{2,\infty} (\R^N) \subseteq B^{3 \sigma s/2}_{2,\infty} (\R^N)$ and 
the inclusion is continuous. Hence, from~\eqref{e:spaziosharp} we infer~\eqref{e:bootreg} 
which concludes the proof of the lemma.

%%%%%%%%%%%%%%%%%%%%%%%%%%%%%%%%%%%%%%%%%%%%%%%%%%%%%%%%%%%%%%%%%%%%%%%%%%%%%%%%%%%%%%%%%%%%%%%%%%%%%%%%%%%%%%%%%%%%%%%%%%
%
%     NEW SECTION
% 
%%%%%%%%%%%%%%%%%%%%%%%%%%%%%%%%%%%%%%%%%%%%%%%%%%%%%%%%%%%%%%%%%%%%%%%%%%%%%%%%%%%%%%%%%%%%%%%%%%%%%%%%%%%%%%%%%%%%%%%%%%

\section{Conclusion of the bootstrap argument}
\label{ss:boot2}

%%%%%%%%%%%%%%%%%%%%%%%%%%%%%%%%%%%%%%%%%%%%%%%%%%%%%%%%%%%%%%%%%%%%%%%%%%%%%%%%%%%%%%%%%%%%%%%%%%%%%%%%%%%%%%%%%%%%%%%%%%

\subsection{Proof of Theorem~\ref{t:main}}

First, we point out that we have already given the proof of~\eqref{eq:key1} in~\S~\ref{s:dimkey1}, so we are left to prove~\eqref{eq:key2}.
To this end, we proceed as in~\S~\ref{s:partialproof} and we point out that 
\begin{equation}
\label{e:apiub}
    \| T_h u - u \|_s^2 \leq |A| + |B|,
\end{equation}
where $A$ and $B$ are as in~\eqref{eq:decompose_normas}. Owing to~\eqref{e:scalar}, 
\begin{equation}
\label{e:bi}
   |B| \stackrel{\eqref{e:scalar}}{\leq} C (N, s, \mathrm{diam} \, \Omega)  | h|^s \| f\|_{L^2(\RN)}\|f\|_{\spazio(\Omega)'}
   \stackrel{\eqref{e:embeddingchian}}{\leq}
   C (N, s, \mathrm{diam} \, \Omega)  | h|^s \| f\|_{L^2(\RN)}\|f\|_{H^{-s}(\R^N)}.
\end{equation}
Next, we recall~\eqref{eq:est_1} and we decompose $A$ as 
\begin{equation}
\label{e:a}
   A= C(N, s) \big[ I_1 + I_2+ I_3 \big], 
\end{equation} 
where $I_1$, $I_2$ and $I_3$ are defined as in~\eqref{eq:defiuno},~\eqref{eq:defidue} and~\eqref{eq:defitre}, respectively. 
Owing to~\eqref{eq:estI1}, 
\begin{equation}
\label{e:stimaiuno}
       |I_1| \stackrel{\eqref{eq:estI1}}{\leq} C (N, s, \mathrm{Lip} \, \phi, \mathrm{diam} \, \Omega)  | h|^s \|f\|^2_{\spazio(\Omega)'} \stackrel{\eqref{e:embeddingchian}}{\leq}
  C (N, s, \mathrm{Lip} \, \phi, \mathrm{diam} \, \Omega) 
   | h|^s \| f\|_{L^2(\RN)}\|f\|_{H^{-s}(\R^N)}.
\end{equation}
To control $I_2$, we first choose $\sigma_1 \in ( 2/3 , 2/(3s) ) \neq \emptyset$ so that 
\begin{equation}
\label{e:sigmauno1}
      \frac{3}{2} \sigma_1 s <1. 
\end{equation}
We apply Lemma~\ref{l:bootreg} and we recall that, when $r\in (0,1)$, the $B^r_{2,\infty}$-norm 
can be characterized as in~\eqref{e:b:equivalent}. We conclude that 
\begin{align} \label{e:stimaidue1}
    \sup_{h \in \R^N \setminus \{ 0 \}} \frac{\| u- u_h \|_{ L^2(\RN) }}{|h|^{3 \sigma_1 s/2} }
      & \stackrel{\eqref{e:b:norm}, \eqref{e:b:equivalent}}{\leq}
         C(N, s) \| u \|_{B^{3 \sigma_1 s/2}_{2, \infty} (\R^N)} \\
  \nonumber       
      & \stackrel{\eqref{e:bootreg}}{\leq}
       C (N, s, \mathrm{diam} \, \Omega, \rho, \theta, \sigma_1) 
       \|f \|_{H^{-s} (\R^N) }^{1/2}
   \|f   \|_{L^2(\RN)}^{1/2}.
\end{align}
Next, we choose 
 \begin{equation}
 \label{e:sigma21}
          \sigma_2 = \frac{2}{3\sigma_1} \in (s,1)
 \end{equation}
 and by proceeding as in~\eqref{eq:commutator} we obtain
 \begin{equation} \label{eq:commutator2}
\begin{split} 
I_2  & =  \| \mathcal{C}(\phi, u_h -u) \|^2_{L^2(\RN)} 
 \stackrel{\eqref{e:interplp}}{\leq}
  C(N, s, \mathrm{Lip} \, \phi,  \sigma_2)  \| u_h - u\|^{2 \sigma_2}_{L^2(\RN)}
\| u - u_h \|^{2-2\sigma_2}_s \\ & \leq 
C(N, s, \mathrm{Lip} \, \phi,  \sigma_2)  \| u_h - u\|^{2\sigma_2}_{L^2(\RN)}
\| u  \|^{2-2\sigma_2}_s \\ & 
  \stackrel{\eqref{eq:lax_milgram}}{\leq}
 C(N, s, \mathrm{Lip} \, \phi, \mathrm{diam} \, \Omega, \sigma_2) 
  \| u_h - u\|^{2 \sigma_2}_{L^2(\RN)} 
  \| f \|^{2-2\sigma_2}_{\spazio (\Omega)'} \\
  & \stackrel{\eqref{e:stimaidue1}}{\leq} 
   C(N, s, \mathrm{Lip} \, \phi, \mathrm{diam} \, \Omega, \rho, \theta, \sigma_1, \sigma_2)
  |h|^{3 \sigma_1 \sigma_2 s}
  \|f \|_{H^{-s} (\R^N) }^{\sigma_2}
   \|f   \|_{L^2(\RN)}^{\sigma_2} 
  \| f \|^{2- 2\sigma_2}_{\spazio (\Omega)'}   \\
  & \stackrel{\eqref{e:sigma21}}{\leq}
     C(N, s, \mathrm{Lip} \, \phi, \mathrm{diam} \, \Omega, \rho, \theta, \sigma_1, \sigma_2)
  |h|^{2 s}
 \|f \|_{H^{-s} (\R^N) }^{\sigma_2}
   \|f   \|_{L^2(\RN)}^{\sigma_2} 
  \| f \|^{2- 2\sigma_2}_{\spazio (\Omega)'}
   \\ & \stackrel{\eqref{e:embeddingchian}}{\leq}
   C(N, s, \mathrm{Lip} \, \phi, \mathrm{diam} \, \Omega, \rho, \theta, \sigma_1, \sigma_2)
  |h|^{2 s}
    \|f \|_{H^{-s} (\R^N) }^{\sigma_2}
   \|f   \|_{L^2(\RN)}^{\sigma_2}
   \|f \|_{H^{-s} (\R^N) }^{1-\sigma_2}
   \|f   \|_{L^2(\RN)}^{1-\sigma_2} \\
   & \; =  C(N, s, \mathrm{Lip} \, \phi, \mathrm{diam} \, \Omega, \rho, \theta, \sigma_1, \sigma_2)
  |h|^{2 s}
    \|f \|_{H^{-s} (\R^N) }
   \|f   \|_{L^2(\RN)}. \\
\end{split} 
\end{equation}
%
% \goro (THE FOLLOWING SEEMS UNNECESSARY:) \egoro 
%If $s \in (0, 1/2]$, we set 
%$$
%    \sigma_1 : = \frac{5}{6}, \quad \sigma_2: = \frac{4}{5}
%$$
%and by direct check we verify that conditions~\eqref{e:sigmauno1} and~\eqref{e:sigma21} are both satisfied. If $s \in (1/2, 1)$, we set
%$$
%     \sigma_1 : = \frac{s+1}{3s}, \quad \sigma_2: = \frac{2s}{s+1}
%$$
%and again by direct check we verify that conditions~\eqref{e:sigmauno1} and~\eqref{e:sigma21} are both satisfied. 
%
We point out that $\sigma_1$, and consequently $\sigma_2$, can be chosen in such a way that they depend 
only on~$s$, and we simplify the above estimate to
\begin{equation}
\label{e:iduefinale}
    I_2 \leq C(N, s, \mathrm{Lip} \, \phi, \mathrm{diam} \, \Omega, \rho, \theta)
  |h|^{2 s}
    \|f \|_{H^{-s} (\R^N) }
   \|f   \|_{L^2(\RN)}.
\end{equation}
To control $I_3$, we recall~\eqref{e:itre} and we obtain
\begin{equation}
\label{e:itrefinale2}
\begin{split}
          |I_3|   &
          \stackrel{\eqref{e:itre}}{\leq}
          2 \sqrt{I_2 \big(| I_1| + \| (-\Delta)^{s/2} u \|^2_{L^2(\R^N)} \big)}
          \\ 
          & 
          \stackrel{\eqref{e:iduefinale}}{\leq}
          C(N, s, \mathrm{Lip} \, \phi, \mathrm{diam} \, \Omega, \rho, \theta)
  |h|^{s}
    \|f \|^{1/2}_{H^{-s} (\R^N) }
   \|f   \|^{1/2}_{L^2(\RN)}
      \sqrt{| I_1| + \| (-\Delta)^{s/2} u \|^2_{L^2(\R^N)} } \\ &
      \stackrel{\eqref{eq:lax_milgram},\eqref{e:stimaiuno}}{\leq}
     C(N, s, \mathrm{Lip} \, \phi, \mathrm{diam} \, \Omega, \rho, \theta)
  |h|^{s} 
      \|f \|_{H^{-s} (\R^N) }
   \|f   \|_{L^2(\RN)}  \sqrt{ |h|^s + 1 } \\ & \; \;  \leq 
    C(N, s, \mathrm{Lip} \, \phi, \mathrm{diam} \, \Omega, \rho, \theta)
  |h|^{s} 
      \|f \|_{H^{-s} (\R^N) }
   \|f   \|_{L^2(\RN)}. \phantom{\int} \\
      \end{split}
\end{equation}
By combining~\eqref{e:apiub},~\eqref{e:bi},~\eqref{e:a},~\eqref{e:stimaiuno},~\eqref{e:iduefinale} and~\eqref{e:itrefinale2} 
we eventually arrive at~\eqref{eq:key2} and this concludes the proof of Theorem~\ref{t:main}.

%%%%%%%%%%%%%%%%%%%%%%%%%%%%%%%%%%%%%%%%%%%%%%%%%%%%%%%%%%%%%%%%%%%%%%%%%%%%%%%%%%%%%%%%%%%%%%%%%%%%%%%%%%%%%%%%%%%%%%%%%%

\subsection{Proof of Theorem~\ref{th:domain_perturb}}

%%%%%%%%%%%%%%%%%%%%%%%%%%%%%%%%%%%%%%%%%%%%%%%%%%%%%%%%%%%%%%%%%%%%%%%%%%%%%%%%%%%%%%%%%%%%%%%%%%%%%%%%%%%%%%%%%%%%%%%%%%

\subsubsection{Preliminary results}

First, we establish a sharper version of Lemma~\ref{l:tildeu}.
\begin{lemma}
\label{l:tildeusharp}
Let $f \in L^2 (\R^N)$. Under the same assumptions as in the statement of\/ {\rm Lemma~\ref{l:tildeu}}, there is
$\tilde u \in \spazio(\Omega^{-\eps})$ such that~\eqref{eq:est_norma2} holds and moreover  
\begin{equation}
   \| \tilde u - u\|^2_{s}\le 
   C(N, s,  \mathrm{diam} \, \Omega, \rho)
   \left( \frac{\varepsilon}{\sin \theta}\right)^{s} \| f\|_{L^2(\R^N)}\| f\|_{H^{-s}(\R^N)}.
  \label{e:tildeusharp2} 
\end{equation}
\end{lemma}
\begin{proof}
We take the same function $\tilde u$ as in the proof of Lemma~\ref{l:tildeu} (see in particular~\S~\ref{sss:contildeu}). 
Owing to the analysis in~\S~\ref{sss:proofincl}, $\tilde u \in \spazio (\Omega^{-\eps})$ and hence we are 
left to establish~\eqref{e:tildeusharp2}. To this aim, we proceed as in~\S~\ref{sss:proofe} and 
we combine~\eqref{e:differenza} and~\eqref{eq:key2}. We get,
for $h=h_i=t \mathbf n_i$,
\begin{equation*}
\begin{split}
    \| u - \tilde u\|_{s} & \stackrel{\eqref{e:differenza}}{\leq}  
   k \max_{i=1, \dots, k}  \| \phi_i [ u  - u_{it} ] \big] \|_{s}
   \stackrel{\eqref{e:kappa2}}{\leq} C(N, \mathrm{diam} \, \Omega,  \rho) \max_{i=1, \dots, k}
   \| \phi_i [ u  - u_{it} ] \|_{s} \\ &
   \stackrel{\eqref{eq:local_transl},\eqref{e:iacca}}{=}
   C(N, \mathrm{diam} \, \Omega, \rho) \max_{i=1, \dots, k}
   \|T_{h_i} u - u \|_{s} \\ &
   \stackrel{\eqref{eq:key2},\eqref{eq:gradient}}{\leq}
    C(N, s, \mathrm{diam} \, \Omega, \rho) 
    \sqrt{ \vert h\vert^{ s} \| f \|_{L^2(\R^N)}
   \| f\|_{H^{-s}(\R^N)} } \\ &
   %\stackrel{\eqref{e:eps}}{\leq}
 \leq C(N, \mathrm{diam} \, \Omega,  \rho) 
   \sqrt{ \left( \frac{\varepsilon}{\sin \theta}\right)^{s} \| f \|_{L^2(\R^N)}
   \| f\|_{H^{-s}(\R^N)} }, \\
   \end{split}
\end{equation*}
owing to the arbitrariness of $t \in (\varepsilon/\sin \theta, \rho/2)$.  
This establishes~\eqref{e:tildeusharp2}. 
\end{proof}
We now state a sharper version of Lemma~\ref{l:hatw}.
\begin{lemma}
\label{l:hatwsharp} Let $f \in L^2 (\R^N)$. Under the same 
assumptions as in the statement of Lemma~\ref{l:hatw}, there 
is $\hat w \in \spazio (\Omega)$ such that~\eqref{eq:est_norma2w} holds and moreover
\begin{equation}
\label{e:hatwsharp}
        \| \hat w - w \|^2_{s}\le 
   C(N, s,  \mathrm{diam} \, \Omega, \rho)
   \left( \frac{\varepsilon}{\sin \theta}\right)^{s} 
   \| f\|_{L^2(\Omega^\eps)}\| f\|_{H^{-s}(\R^N)}.
\end{equation}
\end{lemma}
\begin{proof}
We take the same function $\hat w$ as in the proof of Lemma~\ref{l:hatw}, 
namely we define $\hat w$ as in~\eqref{e:what}. By arguing 
as in the proof of Lemma~\ref{l:tildeusharp} and applying~\eqref{eq:key2}
 we arrive at~\eqref{e:hatwsharp}. The details are omitted. 
\end{proof}

%%%%%%%%%%%%%%%%%%%%%%%%%%%%%%%%%%%%%%%%%%%%%%%%%%%%%%%%%%%%%%%%%%%%%%%%%%%%%%%%%%%%%%%%%%%%%%%%%%%%%%%%%%%%%%%%%%%%%%%%%%

\subsubsection{Proof of Theorem~\ref{th:domain_perturb}: conclusion}

We proceed as in the proof of Lemma~\ref{l:bootdp}, but we apply 
Lemma~\ref{l:tildeusharp} and~\ref{l:hatwsharp}
instead of Lemma~\ref{l:tildeu} and~\ref{l:hatw}, respectively. 
In particular, in place of~\eqref{e:lemmasetteuno} we get
\begin{equation}
\label{e:lemmanoveuno}
     \| u_a - u_a^{-\eps}\|_s \leq 
      C(N, s,  \mathrm{diam} \, \Omega,  \rho, \theta)
  \eps^{ s/2 } \| f\|^{1/2}_{L^2(\R^N)}\| f\|^{1/2}_{H^{-s}(\R^N)},
\end{equation}
and, in place of~\eqref{e:lemmasettedue},
\begin{equation}
\label{e:lemmanovedue}
     \| u_a^{\eps} - u_a\|_s \leq 
      C(N, s,  \mathrm{diam} \, \Omega, \rho, \theta)
  \eps^{ s/2 } \| f\|^{1/2}_{L^2(\R^N)}\| f\|^{1/2}_{H^{-s}(\R^N)}.
\end{equation}
We plug~\eqref{e:lemmanoveuno} and~\eqref{e:lemmanovedue} into~\eqref{e:chain2}, we recall 
that $\eps$ can be any number satisfying~\eqref{e:inmezzo} and we eventually 
arrive at~\eqref{e:goal2}.

%%%%%%%%%%%%%%%%%%%%%%%%%%%%%%%%%%%%%%%%%%%%%%%%%%%%%%%%%%%%%%%%%%%%%%%%%%%%%%%%%%%%%%%%%%%%%%%%%%%%%%%%%%%%%%%%%%%%%%%%%%

\subsection{Proof of Theorem~\ref{th:regularity}}
We proceed as in the proof of Lemma~\ref{l:bootreg}, but we apply~\eqref{e:goal2} instead 
of~\eqref{e:bootdp}. In particular, we can replace~\eqref{eq:dependence_h} with
\begin{equation}
\label{e:dependenceh2}
      \| u_h - v_h \|_s \leq 
      C(N, s,  \mathrm{diam} \, \Omega, \rho, \theta)
      \| f\|^{1/2}_{L^2(\R^N)}\| f\|^{1/2}_{H^{-s}(\R^N)} |h|^{s/2} 
\end{equation}
and hence we can improve~\eqref{eq:regularity1} to
$$
     \| u - u_h \|_s \leq 
      C(N, s,  \mathrm{diam} \, \Omega, \rho, \theta)
      \| f\|^{1/2}_{L^2(\R^N)}\| f\|^{1/2}_{H^{-s}(\R^N)} |h|^{s/2} .
$$
By arguing as in the proof of Lemma~\ref{l:bootreg} from the above inequality 
we infer that $z := (-\Delta)^{s/2} u$ belongs to the Besov 
space $B^{s/2}_{2, \infty}$. Thus, by applying Lemma~\ref{lemma:besov}, we 
eventually arrive at~\eqref{eq:regularity}.

%%%%%%%%%%%%%%%%%%%%%%%%%%%%%%%%%%%%%%%%%%%%%%%%%%%%%%%%%%%%%%%%%%%%%%%%%%%%%%%%%%%%%%%%%%%%%%%%%%%%%%%%%%%%%%%%%%%%%%%%%%

\subsection{An explicit example}
\label{ss:exantonio}

In this paragraph we discuss the $H^s$ (and henceforth Besov) regularity 
of the solution of~\eqref{eq:DP} in a specific example. Let us fix 
$s \in (0, 1)$ and consider the Poisson problem 
\begin{equation}
\label{ex:DP}
  \begin{cases} 
   (-\Delta)^s u = 1 & \text{in $B_1 (0)$}, \\
   u = 0 & \text{in $\R^N \setminus B_1(0)$}. 
  \end{cases}
\end{equation}
In the above expression, $B_1(0)$ is the unit ball, centered at the 
origin, of~$\R^N$. The solution $u$ is then given by
\begin{equation}
\label{ex:DPsol}
   u(x) =   \begin{cases} 
   C(N, s) (1 - |x|^2)^s & \text{if $|x| <1$}, \\
   0 & \text{elsewhere}. 
   \end{cases}
\end{equation}
A proof of the above fact is given by Getoor~\cite[Theorem 5.2]{Getoor}
(cf.~also~\cite{Bogdan} and~\cite{ROS2}). We have the following regularity result:
\begin{lemma}
\label{l:example}
 Assume $N=1$. Let $u$ be the solution of~\eqref{ex:DP}. Then 
 \begin{equation}
 \label{ex:accaesse}
     u \in H^r (\R) \quad \text{for every $r <s +  \frac{1}{2}$}.
 \end{equation}
\end{lemma}
Note that, owing to~\eqref{e:b:inclusion2}, the above lemma implies 
in particular
\begin{equation}
\label{ex:besov}
     u \in B^r_{2, \infty} (\R) \quad \text{for every $r <s +  \frac{1}{2}$}.
\end{equation}
We now compare this result with the regularity provided by Theorem~\ref{th:regularity}. We set 
$$
    f(x) : = \begin{cases}
    1 & \text{if $|x| <1$}, \\
    0 & \text{elsewhere}
    \end{cases}
$$ 
and we point out that $f \in L^2 (\R)$. Theorem~\ref{th:regularity} implies that 
$u \in B^{3s/2}_{2, \infty}(\R)$. Since $s <1$, then $3s/2< s + 1/2$, whence, in particular,  
$$
    B^r_{2, \infty} (\R) \subset    B^{3s/2}_{2, \infty}(\R)
    \quad \text{if $3s/2 < r < s + 1/2$}.
$$ 
This means that Lemma~\ref{l:example} is consistent with Theorem~\ref{th:domain_perturb} since the regularity result established in Lemma~\ref{l:example} is stronger than the regularity provided by Theorem~\ref{th:domain_perturb}.
\begin{proof}[Proof of Lemma~\ref{l:example}]
We use the explicit formula~\eqref{ex:DPsol} and we proceed according to the following steps. \\
{\sc Step 1:} we make some preliminary considerations. First, we point out that establishing~\eqref{ex:accaesse} 
amounts to show that $(1 + |\xi|^2)^{r/2} \hat u \in L^2 (\R)$, or, equivalently,
\begin{equation}
\label{ex:fourier1}
   \int_\R  (1 + |\xi|^2)^{r} \hat u^2 (\xi) \,\d \xi < + \infty. 
\end{equation}
Since 
$$
    \left| \frac{(1 + |\xi|^2)^{r} }{1 + |\xi|^{2r}}  \right| < C(r) \quad 
    \text{for every $\xi \in \R$} , 
$$
then establishing~\eqref{ex:fourier1} is equivalent to proving
\begin{equation}
\label{ex:fourier2}
   \int_\R  (1 + |\xi|^{2r}) \hat u^2 (\xi) \, \d \xi < + \infty. 
\end{equation}
Since $u \in L^2(\R)$, then $\hat u \in L^2(\R)$ and hence~\eqref{ex:fourier2} holds if and only if 
\begin{equation}
\label{ex:fourier3}
   \int_\R  |\xi|^{2r} \hat u^2 (\xi) \, \d \xi < + \infty. 
\end{equation}
Finally, we point out that 
\begin{equation}
\label{ex:fourier41}
   \int_\R  |\xi|^{2r} \hat u^2 (\xi) \, \d \xi = 
   \int_{-2}^2  |\xi|^{2r} \hat u^2 (\xi) \, \d \xi +  \int_{|\xi| >2}  |\xi|^{2r} \hat u^2 (\xi) \, \d \xi       \leq 
   C \| \hat u \|^2_{L^2 (\R)} +   \int_{|\xi|>2}  |\xi|^{2r} \hat u^2 (\xi) \, \d \xi 
\end{equation}
and this implies that to establish~\eqref{ex:accaesse} it suffices to show that 
\begin{equation}
\label{ex:fourier4}
   \int_{|\xi|>2}  |\xi|^{2r} \hat u^2 (\xi) \, \d \xi < + \infty \quad \text{for every $r <s +  \frac{1}{2}$}.  
\end{equation}
{\sc Step 2:} we compute the Fourier transform of $u$. To this end, we note 
that $u$ is smooth on the interval $(-1, 1)$ and satisfies  
\begin{equation}
\label{eq:equation}
  (1-x^2) u'(x) + 2s x u(x) = 0 \quad \text{for $x \in (-1, 1)$}.
\end{equation}
It is not difficult to show that 
then \eqref{eq:equation} holds in fact in the sense of distributions on~$\R$
as $u$ is defined by~\eqref{ex:DPsol}. 
Thus, we can take the Fourier transform
of both sides of~\eqref{eq:equation} and obtain
\begin{equation}
\mathcal{F}((1-x^2) u'(x) + 2s x u(x)) = 0 \,\,\,\,\hbox{ in } \R_\xi.
\end{equation}
A straightforward computation ensures that
$$
\mathcal{F}(x^2 u') = -\frac{\d^2}{\d\xi^2}\hat v(\xi),
$$
provided that $v(x) = u'(x)$. This implies that 
$$
 \mathcal{F}(x^2 u') = -i\Big( \xi \frac{\d^2}{\d\xi^2}\hat u 
  + 2\frac{\d}{\d\xi} \hat u\Big).
$$
By using the above equality we can re-write \eqref{eq:equation} as
\begin{equation}
\label{eq:bessel}
  \xi \frac{\d^2}{\d\xi^2}\hat u + (2 + 2s) \frac{\d}{\d\xi} \hat u 
   + \xi\hat u = 0\,\,\,\,\hbox{ in } \R_\xi.
\end{equation}
Note furthermore that $\hat u$ is a smooth function since $u$ is compactly supported. \\
{\sc Step 3:} we only consider the case $\xi\in (2,+\infty)$,
since the case $\xi\in(-\infty,-2)$ is analogous. We set
$z(\xi):=\xi^{1+s}\hat u(\xi)$. Then, noting for simplicity by
$v'(\xi)$ the derivative of a generic function $v(\xi)$, 
by a direct computation we can 
check that $z$ solves 
\begin{equation}
\label{eq:bessel2}
  z''(\xi) + \Big( 1 - \frac{s(1+s)}{\xi^2} \Big) z(\xi) = 0
   \,\,\,\,\hbox{ for } \xi \in (2,+\infty).
\end{equation}
By multiplying the above expression times $z'(\xi)$ we then infer
\begin{equation}
\label{eq:bessel3}
  \frac12 \frac{\d}{\d \xi} \Big[ (z')^2(\xi)
     + z^2(\xi) \Big( 1 - \frac{s(1+s)}{\xi^2} \Big)\Big]
     = \frac{s(1+s)z^2(\xi)}{\xi^3}
   \,\,\,\,\hbox{ for } \xi \in (2,+\infty).
\end{equation}
Next, we point out that, if $s\in(0,1)$ and $\xi \in (2,+\infty)$, then
\begin{equation}
\label{eq:bessel4}
   \Big( 1 - \frac{s(1+s)}{\xi^2} \Big) \ge \frac12.
\end{equation}
We then set   
\begin{equation}
\label{eq:bessel5}
   m(\xi):= \Big[ z'(\xi)^2 
     + z^2(\xi) \Big( 1 - \frac{s(1+s)}{\xi^2} \Big)\Big],
\end{equation}
and 
we obtain the differential inequality
\begin{equation}
\label{eq:bessel6}
  m'(\xi) \le \frac{C}{\xi^3} m(\xi)
   \,\,\,\,\hbox{ for } \xi \in (2,+\infty).
\end{equation}
By applying Gronwall's lemma, we deduce that $m$, and consequently $z$, 
is bounded in the interval $(2,+\infty)$. By performing a similar argument 
on the interval $(-\infty,-2)$, we then have    
%
% 
%\begin{equation}
%\label{ex:explicit}
%\hat u(\xi) =\frac{ (c_1 \cos\xi + c_2\sin\xi)}{\vert \xi\vert^{1+s}}, \quad c_1, c_2 \in \R
%\end{equation}
%provides a solution of~\eqref{eq:bessel} in the set $|\xi|>1$. To find the above formula one 
%can proceed as in \cite[10.4.107]{special}. 
%Note that~\eqref{ex:explicit} implies that 
%
$$
    \int_{|\xi|>2}  |\xi|^{2r} \hat u^2 (\xi) \,\d \xi 
    \leq C \int_{|\xi|>2}  |\xi|^{2(r-1-s)} \,\d \xi < + \infty
$$
provided that $2(r-1-s) < -1$, namely that $r< s + 1/2$. This establishes~\eqref{ex:fourier4} 
and henceforth~\eqref{ex:accaesse} and concludes the proof of the lemma. 
\end{proof}

%%%%%%%%%%%%%%%%%%%%%%%%%%%%%%%%%%%%%%%%%%%%%%%%%%%%%%%%%%%%%%%%%%%%%%%%%%%%%%%%%%%%%%%%%%%%%%%%%%%%%%%%%%%%%%%%%%%%%%%%%%
%
%     NEW SECTION
% 
%%%%%%%%%%%%%%%%%%%%%%%%%%%%%%%%%%%%%%%%%%%%%%%%%%%%%%%%%%%%%%%%%%%%%%%%%%%%%%%%%%%%%%%%%%%%%%%%%%%%%%%%%%%%%%%%%%%%%%%%%%

\section{Proof of Theorem~\ref{th:spectral_stability}}
\label{ss:proof_s_stability}

%%%%%%%%%%%%%%%%%%%%%%%%%%%%%%%%%%%%%%%%%%%%%%%%%%%%%%%%%%%%%%%%%%%%%%%%%%%%%%%%%%%%%%%%%%%%%%%%%%%%%%%%%%%%%%%%%%%%%%%%%%

\subsection{Preliminary results}
\label{ss:stpreli}

The following results is well-known. A proof is given, e.g., in 
\cite[Prop.~9]{serva-valdi11} under the additional assumption $2s<N$.
Actually, the argument in \cite{serva-valdi11} seems to work for
general $s\in(0,1)$. However, for the reader's convenience, we provide here
a sketch of an alternative proof. 
\begin{lemma}
\label{l:spectral}
Let $\Omega \subset \R^N$ an open and bounded set and let $s \in (0, 1)$. Then the following properties hold:
\begin{itemize}
\item[i)] The operator $(-\Delta)^s$ admits a diverging sequence of positive eigenvalues 
\begin{equation}
\label{e:sequence2}
      0 < \lambda_1 < \lambda_2 \leq \lambda_3 \leq \dots \leq \lambda_n \nearrow + \infty 
\end{equation}
in $\Omega$. As usual, in~\eqref{e:sequence2} we count each eigenvalue according to its multiplicity. 
Note furthermore that the first eigenvalue $\lambda_1$ is simple, namely it has multiplicity $1$. 
\item[ii)] The Rayleigh min-max principle holds, namely for every $n \in \N$ we have 
\begin{equation}
\label{e:minmax}
         \lambda_n = \min_{V \in \mathcal V (n)} \max_{u \in V \setminus \{ 0 \}}
         \frac{ \| u \|^2_s}{\| u \|_{L^2 (\Omega)}^2} 
         = \max_{u \in S_n \setminus \{ 0 \}}
         \frac{ \| u \|^2_s}{ \| u \|_{L^2 (\R^N)}^2} =
         \frac{ \| u_n \|^2_s}{ \| u_n \|_{L^2 (\R^N)}^2}. 
\end{equation}
In the previous expression, $\mathcal V(n)$ is the set of $n$-dimensional subspaces of $\spazio (\Omega)$, 
$u_1, \dots, u_n$ are the eigenfunctions associated to the eigenvalues $\lambda_1, \dots, \lambda_n$ and $S_n$ is the subspace generated by $u_1, \dots, u_n$.
\end{itemize}
\end{lemma}
\begin{proof}
We first establish i).
We consider the linear operator $R: L^2 (\Omega) \to \spazio (\Omega)$ which maps the function $f \in L^2 (\Omega) \subseteq \spazio (\Omega)'$ 
to the weak solution $u = R(f)$ of the  Poisson problem~\eqref{eq:DP}. We start with showing 
that $R$ is continuous. We recall that the bilinear form $[\cdot, \cdot]_s$ 
is defined by~\eqref{e:scalarpoincare} and by plugging $u$ as a test function in~\eqref{eq:weak_solution_bis} we get 
$$
   C(N, s) \| u \|_s^2 = [u, u]_s= 
     \langle f, u \rangle \leq \| f \|_{L^2 (\Omega)} \|u\|_{L^2 (\Omega)} 
     \leq C(N, s, \mathrm{diam} \, \Omega) \| f \|_{L^2 (\Omega)} \|u\|_{s}.
$$
To establish the last inequality, we have used~\eqref{eq:poincare}. The above inequality implies
$$
    \| R(f) \|_s \leq C(N, s, \mathrm{diam} \, \Omega) \| f \|_{L^2 (\Omega)}
$$ 
and hence establishes the continuity of $R$. 

Next, we term $i$ the immersion $i : \spazio (\Omega) \to L^2 (\Omega)$. 
Since~$\Omega$ is bounded, by general results on fractional Sobolev spaces 
(see~\cite[Theorem 8.2]{Dine_Pala_Vald})
$i$ is a compact map. 
%
%Indeed, since $\Omega$ is 
%bounded we can fix a sufficiently large ball $D$ such that $\Omega \subseteq D$. We then have $\spazio (\Omega) \subseteq \spazio (D)$. 
%Since $D$ is an extension domain, the map $j: \spazio (D) \to L^2 (D)$ is compact, see~\cite[Theorem 8.2]{Dine_Pala_Vald}, 
%and this implies that $i$ is also compact. \III WHAT IS $j$? BUT ISN'T THIS OBVIOUS? \EEE

Finally, we consider the operator $R \circ i$, which is continuous and compact because it is the composition of 
a continuous operator with a compact operator. Note furthermore that the operator $R \circ i$ is self-adjoint with respect 
to the bilinear form  $[\cdot, \cdot]_s$, which is a scalar product on $\spazio (\Omega)$. 
Indeed, owing to~\eqref{eq:weak_solution_bis}, for every $u, v \in \spazio (\Omega)$ we have
$$
    [R(u), v]_s= \langle u, v \rangle=  ( u, v ) = (v, u) = \langle v, u \rangle = 
    [ R(v), u]_s= [ u, R(v)]_s   .  
$$
We conclude that $R \circ i$ is a compact, self-adjoint operator on a separable Hilbert space and henceforth 
admits a sequence of eigencouples $\{ (\mu_n, u_n) \}_{n \in \N}$
with $\{ \mu_n \}$ converging to $0$ as $n \to + \infty$. 
Namely, for $n\in \N$, we have 
$$
\mu_n (-\Delta)^s (u_n) = u_n.
$$
By using $u_n$ as a test function we then obtain
$$
    \mu_n C(N, s) \| u_n \|^2_s = \|u_n \|^2_{L^2 (\Omega)} >0,
$$
where we have also used that $u_n \neq 0$ by definition of eigenvector. 
The above equality implies that $\mu_n >0$ for every $n$. 
By setting $\lambda_n : = 1/ \mu_n>0$ we obtain a sequence of eigenvalues for the operator 
$(- \Delta)^s$. Note that, for $n \to + \infty$, $\lambda_n$ diverges to $+ \infty$ 
because $\mu_n$ converges to $0$. 

Finally, arguing as in the case of the standard Laplace operator, one can prove that the first 
eigenvalue is simple and that the Rayleigh min-max principle holds. The details are omitted. 
\end{proof}
To state the next result, we have to introduce some notation. First, we fix two open and 
bounded sets $\Omega_a, \Omega_b \subseteq \R^N$ and an open ball $D$ containing both $\Omega_a$ and $\Omega_b$. As in~\eqref{eq:proj} we denote
by ${P_{D\to \Omega_a}: \Xzs(D)\to \Xzs(\Omega_a)}$
the projection operator with respect to the scalar
product \eqref{e:scalarpoincare}. Also, we fix $s \in (0, 1)$ and we term $(\lambda_n^a,u^a_n)$, $(\lambda_n^b,u^b_n)$ 
the sequence of eigencouples of the operator $(-\Delta)^s$ in $\Omega_a$ and $\Omega_b$, respectively. Finally, we term 
\begin{equation}
\label{eq:span}
  S_{n}^b:=\hbox{span}\left\{u_{i}^b,\,\,i=1,\ldots,n  \right\}.
\end{equation}
the subspace generated by the eigenfunctions $u_1^b, \dots, u^b_n.$ Note that 
 $S_n^b\subseteq \spazio(\Omega_b) \subseteq \spazio (D)$.   

The following lemma reduces the problem of controlling the eigenvalues 
 to the problem of controlling the projections of 
the corresponding eigenfunctions. 
\begin{lemma}\label{lemma:proj}
 Fix $n\in \mathbb{N}$ and suppose that
 there are positive constants $A>0$ and $0<B<1$ such 
 that, for every $u\in S_{n}^b$,
 \begin{eqnarray}
   \| P_{D\to\Omega_a}u - u \|^2_{s}\le A\| u\|^2_{L^2(\RN)},\label{eq:stimas}\\
   \| P_{D\to\Omega_a}u-u\|^2_{L^2(\RN)}\le B\| u\|^2_{L^2(\RN)}\label{eq:stimaL2}. 
 \end{eqnarray}
 Then
 \begin{equation}
 \label{eq:stima_eigen:2}
 \lambda_n^{a}-\lambda_{n}^b\le \frac{A}{(1-\sqrt{B})^2}.
 \end{equation}
\end{lemma}
\begin{proof}
We simply apply~\cite[Lemma 15]{LMS2013} with 
$$
    H : = \spazio (D), \quad  V_a : = \spazio (\Omega_a), \quad
    \mathcal H(u, v) : = [ u, v ]_s, 
   \quad h(u, v) : = (u, v).  
$$
\end{proof}

%%%%%%%%%%%%%%%%%%%%%%%%%%%%%%%%%%%%%%%%%%%%%%%%%%%%%%%%%%%%%%%%%%%%%%%%%%%%%%%%%%%%%%%%%%%%%%%%%%%%%%%%%%%%%%%%%%%%%%%%%%

\subsection{Conclusion of the proof of Theorem~\ref{th:spectral_stability}}

Let $n$, $s$ and $\Omega_a$, $\Omega_b$ be as in the statement of Theorem~\ref{th:spectral_stability}. We also fix $\varepsilon$ such that 
\begin{equation}
\label{e:stainmezzo}
        0 < d_H^c (\Omega_a, \Omega_b) < \varepsilon < \nu
\end{equation}
and we proceed according to the following steps. \\
{\sc Step 1:} owing to~\eqref{e:chaus}, condition~\eqref{e:stainmezzo} implies $e^c (\Omega_b, \Omega_a)< \varepsilon$. By using Lemma~\ref{l:hausdorff}, we infer 
\begin{equation}
\label{e:dentro2}
        \Omega_b^{-\varepsilon} \subseteq \Omega_a.
\end{equation}
We now fix $i=1, \dots, n$ and consider the $i$-th eigencouple $(\lambda_i^b, u_i^b)$ of $(-\Delta)^s$ on $\Omega_b$. 
We apply Lemma~\ref{l:tildeusharp} and we infer that, if $\nu$ and henceforth $\varepsilon$ satisfies~\eqref{e:eps}, then there is $\tilde u
\in \spazio ( \Omega_b^{-\varepsilon})$ such that 
\begin{equation}
\label{e:proiezioneaut}
      \| u_i^b - \tilde u \|_s 
    %  
    %  \leq C(N, s, \mathrm{diam} \, D, \rho, \theta) \varepsilon^{s/2} \| \lambda_i^b
    %  u_i^b \|_{{\color{magenta} H^{-s}(\R^N)}}
  \leq  C(N, s, \mathrm{diam} \, D, \rho, \theta) \varepsilon^{s/2}  \lambda_i^b
      \| u_i^b \|_{L^2 (\R^N)}.  
\end{equation}  
By~\eqref{e:dentro2}, we have $\spazio ( \Omega_b^{-\varepsilon}) \subseteq \spazio (\Omega_a)$. 
Hence, using~\eqref{e:proiezioneaut}, we get 
\begin{equation}
\label{e:proiezioneaut2}
      \| P_{D\to \Omega_a}(u_{i}^{b}) - u^b_i \|_s 
      \leq  C(N, s, \mathrm{diam} \, D, \rho, \theta) \varepsilon^{s/2}  \lambda_i^b
      \| u_i^b \|_{L^2 (\R^N)}. 
\end{equation}  
Finally, we recall that by assumption $\Omega_b$ contains a ball $B_r$ of radius $r$. 
This implies that $\spazio (B_r) \subseteq \spazio (\Omega_b)$ 
and by using the monotonicity of the eigenvalues with respect to set inclusion (which follows 
from the Rayleigh min-max principle~\eqref{e:minmax}) we conclude that $\lambda_i^b \leq C(N, s, r, i)$.
By using~\eqref{e:proiezioneaut2} we finally arrive at 
\begin{equation}
\label{e:proiezioneaut3}
      \| P_{D\to \Omega_a}(u_{i}^{b}) - u^b_i \|_s 
      \leq  C(N, s, \mathrm{diam} \, D, \rho, \theta, r, i) \, \varepsilon^{s/2}  
      \| u_i^b \|_{L^2 (\R^N)}. 
\end{equation}   
{\sc Step 2:} we fix $u \in S^b_{n}$ (see \eqref{eq:span}), namely 
\begin{equation}
\label{eq:span_u}
u = \sum_{i=1}^{n}z_i u^{b}_{i}
\end{equation}
for some $z_1, \dots, z_n \in \R$. We recall that by construction the eigenfunctions $u_j^b$ are orthogonal 
with respect to the scalar product  $[\cdot, \cdot]_s$, which implies 
$$
    (u^b_j , u^b_i ) = \frac{1}{ \lambda_j^b} [u^b_j, u^b_i ]_s =0\quad 
    \text{if $i \neq j$}.
$$
By using~\eqref{e:proiezioneaut3} we get  
\begin{equation}
\begin{split}
\label{eq:laura}
  \| P_{D\to \Omega_a}u - u \|_{s}
 & = \left\| \sum_{i=1}^{n}z_i \big[ P_{D\to \Omega_a} u^{b}_{i} 
  - u^b_i \big]\right\|_s 
  \leq  \sum_{i=1}^{n} | z_i |  \| P_{D\to \Omega_a} u^{b}_{i} 
  - u^b_i  \|_s \\ & 
    \le C(N, s, \mathrm{diam} \, D, \rho, \theta, r, n) 
    \eps^{s/2}    \sum_{j=1}^n \vert z_j\vert \| u^{b}_j\|_{L^2(\RN)} \\ & 
   \le C(N, s, \mathrm{diam} \, D, \rho, \theta, r, n) 
    \eps^{s/2}   \| u\|_{L^2(\RN)}. 
    \end{split}
\end{equation}
Owing to~\eqref{eq:poincare} the above inequality also implies 
\begin{equation}
\label{e:elledue}
  \| P_{D\to \Omega_a}u - u \|_{L^2(\R^N)}
  \leq C(N, s, \mathrm{diam} \, D, \rho, \theta, r, n) 
    \eps^{s/2}    \| u\|_{L^2(\RN)}.
\end{equation}
{\sc Step 3:}  we apply Lemma~\ref{lemma:proj}. We recall~\eqref{eq:laura} and~\eqref{e:elledue} and 
we conclude that the hypotheses are satisfied if we assume that 
$$
    B: =  C(N, s, \mathrm{diam} \, D, \rho, \theta, r,n) 
    \eps^{s} \leq \frac{1}{2} <1. 
$$ 
We can choose the constant $\nu$ in the statement of Theorem~\ref{th:spectral_stability} in such a way 
that the above condition is satisfied for every $\eps < \nu$. By using Lemma~\ref{lemma:proj} we conclude that 
$$
     \lambda^a_n - \lambda_n^b \leq \frac{C(N, s, \mathrm{diam} \, D, \rho, \theta, r, n) 
    \eps^{s} }{(1- 1/\sqrt{2})^2} \leq C(N, s, \mathrm{diam} \, D, \rho, \theta, r, n) 
    \eps^{s}. 
$$
Since the above inequality holds for every $\eps$ satisfying~\eqref{e:stainmezzo}, 
we arrive at  
$$
   \lambda^a_n - \lambda_n^b \leq C(N, s, \mathrm{diam} \, D, \rho, \theta, r, n) 
    d_H^c(\Omega_a, \Omega_b)^{s}
$$
and by exchanging the roles of $\Omega_a$ and $\Omega_b$ we eventually conclude the proof of~\eqref{eq:stima_eigen}.

%%%%%%%%%%%%%%%%%%%%%%%%%%%%%%%%%%%%%%%%%%%%%%%%%%%%%%%%%%%%%%%%%%%%%%%%%%%%%%%%%%%%%%%%%%%%%%%%%%%%%%%%%%%%%%%%%%%%%%%%%%
%
%     NEW SECTION
% 
%%%%%%%%%%%%%%%%%%%%%%%%%%%%%%%%%%%%%%%%%%%%%%%%%%%%%%%%%%%%%%%%%%%%%%%%%%%%%%%%%%%%%%%%%%%%%%%%%%%%%%%%%%%%%%%%%%%%%%%%%%

% Eigenfunctions Stability 

\section{Proofs of Proposition \ref{P:fund} and Theorem \ref{t:eigenfunction}}\label{s:ef}

%%%%%%%%%%%%%%%%%%%%%%%%%%%%%%%%%%%%%%%%%%%%%%%%%%%%%%%%%%%%%%%%%%%%%%%%%%%%%%%%%%%%%%%%%%%%%

In this section, we discuss the stability of the eigenfunctions of $\Ds$ with respect to domain perturbation. We first provide the proof of Proposition \ref{P:fund}.
\begin{proof}[Proof of Proposition \ref{P:fund}]
We use \eqref{e:ones} and \eqref{e:autosiavvicinano} and we infer that 
 $\|\uj\|_s^2 = \lambda^j \to \lambda$  and $\uj = 0$ in $\RN \setminus \Oj$. This implies that there is $u \in \Hs$ such that, up to  subsequences, we have 
\begin{align*}
 \uj \to u \quad &\mbox{ weakly in } \Hs \ \mbox{ and strongly in } L^2(\mathbb R^N).
\end{align*}
By recalling~\eqref{e:ones}, this implies $\|u\|_{L^2(\RN)} = 1$ and hence $u \neq 0$.

We now show that $(u,\lambda)$ is an eigencouple for $\Ds$ on $\Omega$ by proceeding according to the following steps. \\
{\sc Step 1:} we show that ${} u \in \mathcal X_0^s(\Omega)$. Since $u$ belongs to $H^s(\RN)$, we are left to show $u = 0$ a.e.~in $\RN \setminus \Omega$. To this end, we fix
$\varphi \in C^\infty_c(\mathbb R^N \setminus \overline\Omega)$. 
We claim that 
$$
\supp \varphi \subset \mathbb R^N \setminus \Oj \quad \mbox{ for any } \ \text{$j$ sufficiently large}.
$$
Indeed, $\supp \varphi$ is compactly contained in $\mathbb R^N \setminus \overline\Omega$. Hence, there is $\vep_0 > 0$ such that $\supp \varphi \subset \mathbb R^N \setminus \overline{\Omega^{\vep_0}}$ (see \eqref{e:allargato} for the definition of $\Omega^{\vep_0}$). On the other hand, assumption~\eqref{e:siavvicinano} implies, in particular, that $e(\Oj, \Omega) < 1/j$ and hence by using Lemma \ref{l:hausdorff} we conclude that $\Oj \subseteq \Omega^{1/j}$. In other words, $\supp \varphi \subset \mathbb R^N \setminus \overline{\Oj}$ for $j > 1/\vep_0$ and this implies 
$$
\int_{\mathbb R^N} \uj (x) \varphi(x) \, \d x = 0, 
\quad \text{for every $j > 1/\vep_0$}.
$$
We let $j \to \infty$ and we obtain
$$
\int_{\mathbb R^N} {} u (x) \varphi(x) \, \d x = 0.
$$
Owing to the arbitrariness of $\varphi \in C^\infty_c(\mathbb R^N \setminus \overline\Omega)$ and to the fact that $|\partial \Omega| = 0$, we conclude that $u$ vanishes a.e.~in $\mathbb R^N \setminus \Omega$ and hence that ${} u \in \mathcal X^s_0(\Omega)$. \\
{\sc Step 2:} we show that $\Ds u = \lambda u$ a.e.~in $\Omega$. We fix $\varphi \in C^\infty_c(\Omega)$. 
By using assumption~\eqref{e:siavvicinano}, we infer $e^c(\Omega, \Oj) < 1/j$ and hence that
$$
\supp \varphi \subset \Oj \quad \text{for any $j$ sufficiently large}.
$$
By using $\varphi$ as a test function in the equation for $u^j$, we get
$$
\int_{\mathbb R^N} \uj(x) \Ds \varphi (x) \, \d x
= \lj \int_{\mathbb R^N} \uj(x) \varphi(x) \, \d x
$$
and by passing to the limit as $j \to \infty$, we arrive at
$$
\int_{\mathbb R^N} {} u(x) \Ds \varphi (x) \, \d x
= \lambda \int_{\mathbb R^N} {} u(x) \varphi(x) \, \d x.
$$
By taking into account the regularity of ${} u$ and the arbitrariness of $\varphi \in C^\infty_c(\Omega)$, 
we deduce that $\Ds u~=~\lambda u$ a.e.~in $\Omega$, and therefore 
that $(u,\lambda)$ is an eigencouple of $\Ds$ on $\Omega$. Moreover, 
$u$ satisfies the relation $\|u\|_s^2 = \lambda \|u\|_{L^2(\RN)}^2$, which combined with the equality $\|u\|_{L^2(\RN)}^2 = 1$ gives $\|u\|_s^2 = \lambda$.

Finally, we recall  that $\uj \to {} u$ 
weakly in $\Hs$ and that $\|\uj\|_s \to \|u\|_s$, we use the uniform convexity of $\Hs$ and we conclude that $\uj \to {} u$ strongly in $\Hs$. 
\end{proof}
Proposition \ref{P:fund} does not provide any information on the rate of the
convergence $\uj \to u$. Indeed, in contrast to the stability result for eigenvalues, 
the convergence rate for eigenfunctions is not uniquely determined in general.
The following remark shows that this happens in particular when 
the corresponding eigenvalues are not simple.
\begin{remark}\label{R:cex}
{\rm
Let us consider the case when the (geometric) multiplicities of $\lj$ and $\lambda$ are two. We denote 
by $\uj_\ell$ and $u_\ell$ ($\ell = 1,2$) the corresponding eigenfunctions such 
that $\|\uj_\ell\|_{L^2(\RN)} = \|u_\ell\|_{L^2(\RN)} = 1$ and $[\uj_\ell,\uj_m]_s = \lj \delta_{\ell m}$. 
We furthermore assume that $\uj_\ell \to u_\ell$ strongly in $\Hs$ for each $\ell = 1,2$. 
Such a situation occurs, e.g., for ball-shaped domains. Note that
\begin{equation}
\label{e:ugei}
\uj(x) := (1-\sigma^j) \uj_1(x) + \sigma^j \uj_2(x)
\to u_1(x) \quad \mbox{ strongly in } \Hs
\end{equation}
for every sequence $\{ \sigma^j \}$ such that $\sigma^j \to 0$ as $j \to \infty$. This implies that $\uj$ is also an eigenfunction corresponding to $\lj$ over $\Oj$. Moreover, one has
$$
\|\uj - u_1\|_s \geq \|\uj - \uj_1\|_s - \|\uj_1 - u_1\|_s
\geq \sigma^j \|\uj_1 - \uj_2 \|_s - \|\uj_1 - u_1\|_s
= \sqrt{2\lj} \sigma^j - \|\uj_1 - u_1\|_s.
$$
If we choose $\{ \sigma^j \}$ in such a way that 
$$
\liminf_{j \to \infty} \dfrac{\sigma^j}{\|\uj_1 - u_1\|_s} \in (1/\sqrt{2\lambda},+\infty],
$$
then 
$$
\|\uj - u_1\|_s \geq \kappa \sigma^j \quad \text{ for every sufficiently large $j$}
$$
for some constant $\kappa > 0$. This means that we can construct $ \{ \sigma^j \}$ in such a way that the eigenfunction $\uj$  defined as in~\eqref{e:ugei} has an arbitrarily slow rate of convergence to $u_1$. Hence, in general, the convergence rate of eigenfunctions is not uniquely determined. 
}
\end{remark}
On the other hand, the convergence rate for principal eigenfunctions might be estimated, since the principal eigenvalue
is simple and consequently the situation outlined in~Remark~\ref{R:cex} cannot occur. 
In the general case, we can control the convergence rate of eigenspaces. 
More precisely, in the rest of this section, we will control a suitable notion of ``distance between 
eigenspaces'' by the domain perturbation rate. In particular, we 
will give a proof of Theorem \ref{t:eigenfunction}. 

Throughout the rest of this section, $\Oa$, $\Ob$ are bounded, open sets of $\RN$ satisfying assumptions i)--iii) 
in the statement of Theorem~\ref{th:spectral_stability}. Let $(\lO_j, \eO_j)$ denote the $j$-th eigencouple of $\Ds$ on $\Oa$ (see~\eqref{eigen:intro}). We can assume that $(\eO_j)$ are a CONS of $L^2(\Oa)$: in particular, $(\eO_i, \eO_j) = \delta_{ij}$ and $[\eO_i, \eO_j]_s = \lO_j \delta_{ij}$ 
(cf., e.g., \cite[Proposition 9]{serva-valdi11}). Moreover, we assume that
\begin{equation}\label{ansatz}
\lO_{k-1} < \lO_k \leq \lO_{k+1} \leq \cdots \leq \lO_{k+m-1} < \lO_{k+m}
\end{equation}
for some $k, m \in \N$ (if $k = 1$, we replace $\lO_{k-1}$ by $0$)
and we define the $m$-dimensional space
$$
\Nkm := \mathrm{span} \{ \eO_k, \eO_{k+1}, \ldots, \eO_{k+m-1} \}.
$$
If, for instance, $\lO_k$ is an eigenvalue with (geometric) multiplicity $m$, then \eqref{ansatz} holds 
true and $\Nkm$ is the corresponding eigenspace. Note that we equip both $\Nkm$ and $\Nkme$ with the norm $\|\cdot\|_s$. 

Next, we recall some notions of distance between two \emph{subspaces} $M$, $N$ of $\Xzs(D)$. 
Let $D \subset \RN$ be an open ball containing both $\Oa$ and $\Ob$ as in assumption ii) of 
Theorem~\ref{th:spectral_stability}. Assume moreover that $\Xzs(D)$ is endowed with the norm $\|\cdot\|_s$, which is equivalent to $\|\cdot\|_{\Hs}$. 
We define the \emph{excess} of $M$ from $N$ in $\Xzs(D)$ by setting
$$
\eHs(M,N) := \sup_{x \in M, \;\|x\|_s = 1} \diHs(x,N),
$$
where $\diHs(x,N) := \inf_{y \in N} \|x - y\|_s$. Also, we define the Hausdorff distance between $M$ and $N$ by setting 
$$
\dHs(M,N) := \eHs(M,N) + \eHs(N,M).
$$
Finally, we define the operator $\T : \Xzs(\Oa) \to \Xzs(\Oa)'$ by setting 
$$
\T u = f \quad \Leftrightarrow \quad 
%\Ds u = f \mbox{ in } \Oa, \quad u = 0 \mbox{ in } \mathbb R^N \setminus \Oa
[u,\varphi]_s = \langle f, \varphi \rangle_{\Xzs(\Oa)}
\quad \mbox{ for all } \ \varphi \in \Xzs(\Oa).
$$
Note that $\T$ is bounded, linear and bijective (see \eqref{eq:lax_milgram}). By the Open Mapping 
Theorem, the map $\T^{-1}~:~ \Xzs(\Oa)' \to \Xzs(\Oa)$ is a well defined and bounded linear operator. One can analogously define $\Te : \Xzs(\Ob) \to \Xzs(\Ob)'$ and its inverse $\Te^{-1} : \Xzs(\Ob)' \to \Xzs(\Ob)$. Moreover, every $v \in L^2(D)$ can be regarded as an element $f_v$ of $\Xzs(\Oa)'$ by setting 
$$
    \langle f_v, \varphi \rangle_{\Xzs(\Oa)} : = \int_{\Oa} v(x) 
    \varphi(x) \, \d x \quad \text{for every $\varphi \in \Xzs(\Oa)$}.
$$ 
To simplify notation, we will directly write $v$ instead of $f_v$. 
Note that $\T^{-1}$ and $\Te^{-1}$ are also well-defined if $v \in L^2(D)$ (see~\cite{ROS}), namely
$$
\T u = v \quad \Leftrightarrow \quad 
\Ds u = v \ \mbox{ in } \Oa, \quad u = 0 \ \mbox{ in } \mathbb R^N \setminus \Oa.
$$
In the following, $\T^{-1}$ and $\Te^{-1}$ are often regarded as bounded linear operators from $L^2(D)$ into $\Xzs(D)$.

The proof of the following lemma is based on an abstract theory due to Feleqi~\cite[Lemma 2.4]{Feleqi}. For the reader's convenience we provide a proof, which is specific to our setting.  
\begin{lemma}\label{L:mod-Feleqi}
Assume that condition \eqref{ansatz} holds for some $k, m \in \N$. Define $\delta$ by setting 
\begin{equation}\label{delta}
\delta := \begin{cases}
\dfrac 1 2 \min \left\{
\dfrac 1 {\lO_{k-1}} - \dfrac 1 {\lO_k}, \
\dfrac 1 {\lO_{k+m-1}} - \dfrac 1 {\lO_{k+m}}
\right\} \quad &\mbox{ if } \ k \geq 2,\\
\dfrac 1 2 \left( \dfrac 1 {\lO_{1}} - \dfrac 1 {\lO_{2}} \right)
\quad &\mbox{ if } \ k = 1.
\end{cases}
\end{equation}
Then 
  \begin{align*}
 \eHs \left( \Nkm , \Nkme \right)
\leq m \max \left\{ \delta^{-1}, \lO_{k+m-1}  \right\} \left\| \left. \left( \T^{-1} - \Te^{-1} \right) \right|_{\Nkm} \right\|_{\BDL},
  \end{align*}  
provided that
\begin{equation}\label{hypo}
\begin{cases}
\max \left\{
\left| \dfrac 1 {\lOe_{k-1}} - \dfrac 1 {\lO_{k-1}} \right|, \ 
\left| \dfrac 1 {\lOe_{k+m}} - \dfrac 1 {\lO_{k+m}} \right|
\right\} < \delta \quad &\mbox{ if } \ k \geq 2,\\
\left| \dfrac 1 {\lOe_1} - \dfrac 1 {\lO_1} \right| < \delta
\quad &\mbox{ if } \ k = 1.
\end{cases}
\end{equation}
In the previous expression $( \T^{-1} - \Te^{-1} ) |_{\Nkm} : \Nkm \to \Xzs(D)$ denotes the restriction of
$\T^{-1} - \Te^{-1}$ to $\Nkm$, which is bounded and linear. Also, 
$\|\cdot\|_{\BDL}$ denotes the (standard) norm of bounded linear operators 
from $\Nkm$ to $\Xzs(D)$ (see \eqref{BDL} below for its definition).
\end{lemma}
%
%Finally, we give a proof of Lemma \ref{L:mod-Feleqi}.
%
\begin{proof}%[Proof of Lemma \ref{L:mod-Feleqi}]
We first point out that,  if $M, N$ are closed subspaces of a given Hilbert space $X$ with inner product $(\cdot,\cdot)_X$, then
\begin{equation}\label{eMN}
\eX (M,N) = \|(1 - Q) \circ P\|_{\mathcal L(X)}.
\end{equation}
In the previous expression, $P$ and $Q$ are the projection maps from $X$ onto $M$ and $N$, respectively, and $\eX(M,N)$ is
the excess of $M$ from $N$ computed with respect to the norm $\|\cdot\| := \sqrt{(\cdot,\cdot)_X}$. To establish~\eqref{eMN} we point out that 
\begin{align*}
 \|(1-Q) \circ P\|_{\mathcal L(X)}
&:= \sup_{u \in X, \; \|u\| = 1} \|(1-Q) \circ Pu\|\\
%&= \sup_{u \in X, \; \|u\| \leq 1} \|(1-Q) \circ Pu\|,\\
%&= \sup_{v \in M, \; \|v\| \leq 1} \|(1-Q)v\|\\
&= \sup_{v \in M, \; \|v\| = 1} \|(1-Q)v\|
= \sup_{v \in M, \; \|v\| = 1} \diX(v,N)
= \eX(M,N).
\end{align*}
In the rest of the proof, we will always choose $X = \Xzs(D)$, $\|\cdot\| = \|\cdot\|_s$, $\eX(\cdot,\cdot) = \eHs(\cdot,\cdot)$, $M = \Nkm$ and $N = \Nkme$.
Next, we fix $i = 1,2,\ldots,m$ and observe that 
\begin{align}\label{e:inversi}
 \left\| \left. \left( \T^{-1} - \Te^{-1} \right) \right|_{\Nkm} \right\|_{\BDL}
 &= \sup_{f \in \Nkm, \,\|f\|_s = 1}
  \left\| \left( \T^{-1} - \Te^{-1} \right) f \right\|_s\\
 \nonumber 
 &\geq  \left\| \left( \T^{-1} - \Te^{-1} \right) \dfrac{\eO_{k+i-1}}{\sqrt{\lambda_{k+i-1}^a}} \right\|_s\\
%\\
% &= \left\| \T^{-1} \eO_{k+i-1} - \Te^{-1} \eO_{k+i-1} \right\|_s\\
%
 \nonumber
 & = \dfrac 1 {\sqrt{\lambda_{k+i-1}^a}} \left\| \dfrac{\eO_{k+i-1}}{\lO_{k+i-1}} - \Te^{-1} \eO_{k+i-1} \right\|_s.
\end{align}
%
% Here we used the fact that
%$$
%\T^{-1} g = \sum_{j = 1}^\infty \dfrac{(g, \eO_j)}{\lO_j} \eO_j
%\quad \mbox{ for } \ g \in L^2(\Oa).
%$$
%Indeed, set $u := \T^{-1} g \in L^2(\Oa)$ (equivalently, $\T u = g$) and write $u = \sum_{j = 1}^\infty c_j \eO_j$ for some sequence $(c_j)$ (since $(\eO_j)$ forms a CONS of $L^2(\Oa)$). Then one has
%$$
%\T \left( \sum_{j = 1}^\infty c_j \eO_j \right)
%= \T u = g
%= \sum_{j = 1}^\infty (g, \eO_j) \eO_j
%$$
%by $g \in L^2(\Oa)$.
%  Since the left-hand side can be rewritten as $\sum_{j = 1}^\infty c_j \lO_j \eO_j$, by comparison of coefficients, it follows that $c_j \lO_j = (g, \eO_j)$.
%
%Now,
%
Let us define the orthogonal projection $\Pe : \Xzs(D) \to \Xzs(\Ob)$ as in \eqref{eq:proj}. 
By using the fact that $\{\eOe_j\}_{j \in \mathbb N}$ is a CONS of $L^2(\Ob)$, we get
\begin{equation}\label{proj-cons}
\Pe v = \sum_{j = 1}^\infty \dfrac{[v, \eOe_j]_s}{\lOe_j} \eOe_j
\quad \mbox{ for } \ v \in \Xzs(D).
\end{equation}
Note that the series at the right-hand side of the above equality  is convergent, since 
$\Pe v~\in~\Xzs(\Ob)$ and 
$\{\eOe_j\}_{j \in \mathbb N}$ is complete.
\begin{comment}
%
In particular, by Theorem \ref{th:domain_perturb}, for each eigencouple $(\lO_j,\eO_j)$ of $\Ds$ in $\Oa$, we note that
\begin{align}\label{ni}
 \| \Pe \eO_j - \eO_j \|_s
 &= \min_{v \in \Xzs(\Ob)} \|v - \eO_j\|_s\\
 &\leq \|u^b - \eO_j\|_s\nonumber
 \leq C \lO_j \dd(\Ob,\Oa)^{s/2} \|\eO_j\|_{L^2(\RN)}
 \leq C \lO_j \dd(\Ob,\Oa)^{s/2},
\end{align}
 where $u^b$ is the solution of \eqref{eq:DP} with $\Omega = \Ob$ and $f$ replaced by the zero extension of $\lO_j \eO_j$ onto $D$, for some $C = C(N,s,\rho,\theta,\mathrm{diam}\,D) \geq 0$. Hence
\end{comment}
%
Since $\eO_{k+i-1} - \Pe \eO_{k+i-1} \in \Xzs(\Ob)^\bot$ and $\Te^{-1}(\eO_{k+i-1} - \Pe \eO_{k+i-1}) = 0$, then 
\begin{align*}
\lefteqn{
 \left\| \dfrac{\eO_{k+i-1}}{\lO_{k+i-1}} - \Te^{-1} \eO_{k+i-1} \right\|_s^2
}\\
&=
 \left\| \dfrac{\Pe \eO_{k+i-1}}{\lO_{k+i-1}} - \Te^{-1} \circ \Pe \eO_{k+i-1} 
     + \dfrac{ \eO_{k+i-1} - \Pe \eO_{k+i-1} }{\lO_{k+i-1}}
     - \Te^{-1} (\eO_{k+i-1} - \Pe \eO_{k+i-1} )
     \right\|_s^2 \\
     &=
     \left\| \dfrac{\Pe \eO_{k+i-1}}{\lO_{k+i-1}} - \Te^{-1} \circ \Pe \eO_{k+i-1} \right\|_s^2
     + \left\|\dfrac{ \eO_{k+i-1} - \Pe \eO_{k+i-1} }{\lO_{k+i-1}}\right\|_s^2
\\
     &\stackrel{\text{\eqref{proj-cons}}}=
\left\|
\dfrac 1 {\lO_{k+i-1} } \sum_{j = 1}^\infty \dfrac{[\eO_{k+i-1}, \eOe_j]_s}{\lOe_j} \eOe_j
- \sum_{j = 1}^\infty \dfrac 1 {\lOe_j} \dfrac{[\eO_{k+i-1}, \eOe_j]_s}{\lOe_j} \eOe_j
     \right\|_s^2
          + \left\|\dfrac{ \eO_{k+i-1} - \Pe \eO_{k+i-1} }{\lO_{k+i-1}}\right\|_s^2\\
&=
 \sum_{j=1}^\infty \left( \dfrac 1 {\lO_{k+i-1} } - \dfrac 1 {\lOe_j} \right)^2
     \dfrac{[\eO_{k+i-1}, \eOe_j]_s^2}{\lOe_j}
          + \left\|\dfrac{ \eO_{k+i-1} - \Pe \eO_{k+i-1} }{\lO_{k+i-1}}\right\|_s^2\\
&\geq
     \sum_{j\neq k, \cdots,k+m-1} \left( \dfrac 1 {\lO_{k+i-1} } - \dfrac 1 {\lOe_j} \right)^2 \dfrac{[\eO_{k+i-1}, \eOe_j]_s^2}{\lOe_j}
               + \left\|\dfrac{ \eO_{k+i-1} - \Pe \eO_{k+i-1} }{\lO_{k+i-1}}\right\|_s^2.
%\\
%&=
%\sum_{j=1, \, j\neq k, \cdots,k+m-1}^\infty \left( \dfrac 1 {\lO_{k+i-1} } - \d%frac 1 {\lOe_{j+i-1}} \right)^2 \|(1-Q) \Pe \eO_{k+i-1}\|_s^2.
\end{align*}
Note furthermore that by combining  
\eqref{delta} and \eqref{hypo} we get 
\begin{equation}
\label{e:basel}
 \left| \dfrac 1 {\lO_{k+i-1}} - \dfrac 1 {\lOe_j} \right| > \delta
\quad \mbox{ if } \ j \leq k - 1 \ \mbox{ or } \ j \geq k+m.
\end{equation}
Indeed, for every $j \geq k + m$ we have 
\begin{align*}
 \left| \dfrac 1 {\lO_{k+i-1}} -  \dfrac 1 {\lOe_j} \right|
% &\geq \dfrac 1 {\lO_{k+i-1}} -  \dfrac 1 {\lOe_j}\\
% &\geq \dfrac 1 {\lO_{k+i-1}} -  \dfrac 1 {\lOe_{k+m}} \qquad \mbox{(because $\lOe_j \geq \lOe_{k+m}$)}\\
 &\geq \dfrac 1 {\lO_{k+m-1}} -  \dfrac 1 {\lOe_{k+m}} \qquad \mbox{(by $\lO_{k+i-1} \leq \lO_{k+m-1}$ and $\lOe_j \geq \lOe_{k+m}$)}\\
 &\geq \dfrac 1 {\lO_{k+m-1}} - \dfrac 1 {\lO_{k+m}} - \left|
  \dfrac 1 {\lOe_{k+m}} - \dfrac 1 {\lO_{k+m}}
 \right|\\
 & \stackrel{\eqref{delta},~\eqref{hypo}}> 2\delta - \delta = \delta. \end{align*}
By using an analogous argument we can establish~\eqref{e:basel} when $j \leq k -1$. Also,  we have 
$$
 \sum_{j\neq k, \cdots,k+m-1} \dfrac{[\eO_{k+i-1}, \eOe_j]_s^2}{\lOe_j}
  \stackrel{\text{\eqref{proj-cons}}}= \sum_{j\neq k, \cdots,k+m-1} \dfrac{[\Pe \eO_{k+i-1}, \eOe_j]_s^2}{\lOe_j}
  = \|(1-Q) \circ \Pe \eO_{k+i-1}\|_s^2
$$
and $\eO_{k+i-1} - \Pe \eO_{k+i-1} = (1-Q)(\eO_{k+i-1} - \Pe \eO_{k+i-1}) \in \Xzs(\Ob)^\bot$.
By going back to \eqref{e:inversi} and using the above relations, we obtain
\begin{align*}
\lefteqn{
 \left\| \left. \left( \T^{-1} - \Te^{-1} \right) \right|_{\Nkm} \right\|_{\BDL}^2
 }\\
 &\geq \dfrac 1 {\lO_{k+i-1}} \left[ \delta^2 \|\underbrace{(1-Q) \circ \Pe \eO_{k+i-1}}_{\quad \in \, \Xzs(\Ob)} \|_s^2 
 + \left( \dfrac 1 {\lO_{k+i-1}} \right)^2 \| \underbrace{(1-Q) (\eO_{k+i-1} - \Pe \eO_{k+i-1})}_{\quad \in \, \Xzs(\Ob)^\bot} \|_s^2 \right]\\
 & \geq \dfrac{1}{\lO_{k+i-1}} \min \left\{ \delta^2, \left(\frac 1 {\lO_{k+i-1}} \right)^2 \right\} \|(1-Q) \eO_{k+i-1}\|_s^2. 
\end{align*} 
By using \eqref{eMN}, we deduce that
\begin{align*}
 \eHs(\Nkm,\Nkme)
&= \|(1 - Q) \circ P\|_{\mathcal L(\Xzs(D))}\\
&= \sup_{u \in \Xzs(D), \, \|u\|_s = 1} \|(1 - Q) \circ Pu\|_s\\
&= \sup_{u \in \Xzs(D), \, \|u\|_s = 1} \left\|
(1-Q) \sum_{i = 1}^m \dfrac{[u, \eO_{k+i-1}]_s}{\lO_{k+i-1}} \eO_{k+i-1}
\right\|_s\\
&\leq \sup_{u \in \Xzs(D), \, \|u\|_s = 1} 
\sum_{i = 1}^m \dfrac{|[u, \eO_{k+i-1}]_s|}{\lO_{k+i-1}} 
\left\|(1-Q) \eO_{k+i-1} \right\|_s
\\
&\leq \sup_{u \in \Xzs(D), \, \|u\|_s = 1} 
\sum_{i = 1}^m \|u\|_s \dfrac{\|\eO_{k+i-1}\|_s}{\lO_{k+i-1}} 
\left\|(1-Q) \eO_{k+i-1} \right\|_s
\\
&\leq
\sum_{i = 1}^m \dfrac 1 {\sqrt{\lO_{k+i-1}}} \left\|(1 -Q) \eO_{k+i-1} \right\|_s\\
&\leq m \max \left\{ \delta^{-1}, \lO_{k+m-1}  \right\} \left\| \left. \left( \T^{-1} - \Te^{-1} \right) \right|_{\Nkm} \right\|_{\BDL}
\end{align*}
and this completes the proof.
\end{proof}
%
%
%In order to obtain an estimate for the distance between two eigenspaces for the fractional eigenvalue problems, 
%
In the following we use this lemma:
\begin{lemma}\label{L:RHS}
Under the same assumptions as in Theorem \ref{th:domain_perturb}, we have 
$$
\left\|
\left( \T^{-1} - \Te^{-1} \right) \Big|_{\Nkm}
\right\|_{\BDL}
\leq C (N, s, \rho, \theta, \mathrm{diam} \, D) {\mathfrak d}(\Ob, \Oa)^{s/2}.
$$
\end{lemma}

\begin{proof}
% By the definition of the norm for bounded linear operators from $\Nkm (\subset  \Hs)$ to $\Hs$, 
We observe that
\begin{align}
\left\|
\left( \T^{-1} - \Te^{-1} \right) \Big|_{\Nkm}
\right\|_{\BDL}
&:= \sup_{f \in \Nkm, \; \|f\|_s = 1} \left\|
(\T^{-1} - \Te^{-1}) f
\right\|_s \label{BDL}\\
&= \sup_{f \in \Nkm, \; \|f\|_s = 1} \left\| u_a - u_b \right\|_s,
\nonumber
\end{align}
where $u_a := \T^{-1} f$ and $u_b := \Te^{-1} f$. 
By applying Theorem~\ref{th:domain_perturb}, we deduce that
\begin{align*}
\left\|
\left( \T^{-1} - \Te^{-1} \right) \Big|_{\Nkm}
\right\|_{\BDL}
&\leq C (N, s, \rho, \theta, \mathrm{diam} \, D) \sup_{f \in \Nkm, \; \|f\|_s = 1}
\|f\|_{L^2(D)}^{1/2} \|f\|_{H^{-s}(\RN)}^{1/2} \mathfrak{d}(\Ob, \Oa)^{s/2}\\
&\stackrel{\eqref{eq:poincare},~\eqref{e:embeddingchian}}\leq 
 C (N, s, \rho, \theta, \mathrm{diam} \, D)
 \mathfrak{d}(\Ob, \Oa)^{s/2},
\end{align*}
which is the desired result.
\end{proof}
We can now state our main results concerning the eigenspace stability. 
\begin{theorem}\label{T:1}
Let assumptions i)--iii) in Theorem~\ref{th:spectral_stability} and \eqref{eq:hypo_dist} hold. Assume furthermore that~\eqref{ansatz} holds for some $k,m \in \mathbb N$ and let $\delta > 0$ be the same as in \eqref{delta}.
 Let $\lambda_1(D) > 0$ denote the
 principal eigenvalue for $\Ds$ on $D$. Then there is a positive constant $\nu$, only  depending  
 on $N$, $s$, $\rho$, $\theta$, $\mathrm{diam} \,D$, $r$, $k$, $m$, $\lambda_1(D)$ and $\delta$, such that, if $d_H^c(\Oa,\Ob) < \nu$, then 
 \begin{equation}\label{ef1}
 \eHs (\Nkm,\Nkme) \leq C(N,s,\rho,\theta,\mathrm{diam}\,D) m \max \left\{ \delta^{-1}, \lO_{k+m-1} \right\} \dd(\Ob,\Oa)^{s/2}.
 \end{equation}
Also, let $\lambda_j(B_r)$ denote the $j$-th eigenvalue of $\Ds$ on $B_r$.
 If, in addition, $\dd(\Oa,\Ob) < (\rho \sin\theta)/2$, then we also have
 \begin{equation}\label{ef2}
  \dHs (\Nkm,\Nkme) \leq C(N,s,\rho,\theta,\mathrm{diam}\,D) m \max \left\{ \delta^{-1}, \lambda_{k+m-1}(B_r) \right\} \left(\dd(\Ob,\Oa)+\dd(\Oa,\Ob)\right)^{s/2}.
 \end{equation}
\end{theorem}
\begin{proof}
%By Theorem \ref{th:spectral_stability}, for each $j = 1,2,\ldots$, we can assure that
%$$
%|\lO_j - \lOe_j| \to 0 \quad \mbox{ as } \ d_H^c(\Oa,\Ob) \to 0_+.
%$$
 By recalling that $\lO_j \geq \lO_1 \geq \lambda_1(D) > 0$ %and $\lOe_1 \geq \lO_1/2$ for $d_H^c(\Oa,\Ob) \ll 1$,
 we get that
$$
\left| \dfrac 1 {\lOe_j} - \dfrac 1 {\lO_j} \right|
= \left| \dfrac{\lO_j - \lOe_j}{\lOe_j \lO_j} \right|
\leq \dfrac{1}{\lambda_1(D)^2} \left|\lO_j - \lOe_j\right| \to 0
\quad \mbox{ for } \ j = k-1,k+m
\ \mbox{ as } \ d_H^c(\Oa,\Ob) \to 0_+.
$$
This implies that \eqref{hypo} holds true if %$d_H^c(\Oa,\Ob) > 0$
 $\nu > 0$ is small enough (the smallness threshold of $\nu$ may also depend on $\lambda_1(D)$ and $\delta$). 
By applying Lemmas \ref{L:mod-Feleqi} and \ref{L:RHS}, we get \eqref{ef1}. 
 If, in addition,  $\dd(\Oa,\Ob) < (\rho \sin \theta)/2$, then by switching $a$ and $b$ and by repeating the same argument as before, we obtain
 $$
   \eHs (\Nkme,\Nkm) \leq C(N,s,\rho,\theta,\mathrm{diam}\,D)
   m \max \left\{ \delta^{-1}, \lOe_{k+m-1} \right\} \dd(\Oa,\Ob)^{s/2}
 $$
 (with same $\nu$). By using the fact that $\lO_{k+m-1}, \lOe_{k+m-1} \leq \lambda_{k+m-1}(B_r)$ and adding the above
 inequality to \eqref{ef1}, we eventually establish \eqref{ef2}.
%
%\begin{comment}
%Moreover, note that
%$$
%\eHs (\Nkme,\Nkm) \leq \dfrac{\eHs (\Nkm,\Nkme)}{1 - \eHs (\Nkm,\Nkme)}
%$$
%  (see~\cite{Feleqi}) and that $\lO_j \leq \lambda_j(B_r)$ for each $j$. Then for $\hat\nu > 0$ small enough so that $\eHs (\Nkm,\Nkme) < 1/2$, \eqref{ef2} follows.
% \end{comment}
%
\end{proof}
%
%\begin{comment}
%\begin{corol}
%Under the same assumptions as in Theorems \ref{th:domain_perturb}, let $\eO_1$ and $\eOe_1$ denote the (normalized) principal eigenfunctions in $\Oa$ and $\Ob$, respectively, satisfying $\|\eO_1\|_{L^2(\RN)} = \|\eOe_1\|_{L^2(\RN)} = 1$ and $[\eO_1, \eOe_1]_s \geq 0$.
%Set
%$$
%\delta := \dfrac 1 2 \left(
%\dfrac 1 {\lO_1} - \dfrac 1 {\lO_2}
%\right) > 0.
%$$
%Then there exists a sufficiently small constant $\nu = \nu (N,s,\rho,\theta,\mathrm{diam}\,D,r,\lambda_1(D)) > 0$ such that, if $d_H^c(\Oa,\Ob) < \nu$, then it follows that
%$$
% \left\| \dfrac{\eO_1}{\sqrt{\lO_1}} - \dfrac{\eOe_1}{\sqrt{\lOe_1}} \right\|_s \leq \dfrac C \delta \left\{ 1 + (1+\delta)\lO_1 \right\}
%\dd(\Ob,\Oa)^{s/2}
%$$
%for some constant $C = C(N,s,\rho,\theta,\mathrm{diam}\,D) > 0$.
%\end{corol}
%\end{comment}
%
We can now give the proof of Theorem \ref{t:eigenfunction}.
\begin{proof}[Proof of Theorem \ref{t:eigenfunction}]
Since the first eigenvalue is simple, 
%, one can check that $\delta > 0$. Moreover, 
assumption \eqref{hypo} is satisfied 
with $k = m = 1$  provided $\nu > 0$ is small enough (the smallness threshold of $\nu$ may again depend on
$\lambda_1(D)$ and $\delta$). By applying Theorem \ref{T:1} we conclude that
\begin{equation}\label{ie1}
\eHs(N_{1,1}^a, N_{1,1}^b) 
 \leq C \max \left\{ \delta^{-1}, \lO_1 \right\} \dd(\Ob,\Oa)^{s/2}.
\end{equation}
Note furthermore that
\begin{align*}
\eHs(N_{1,1}^a, N_{1,1}^b)
%&= \inf_{x = \pm \eO_1} \mathrm{dist}(x,N_{1,1}^{\Ob})\\
&= \max \left\{ \inf_{c \in \mathbb R} \|\eO_1 - c \eOe_1\|_s, 
\inf_{c \in \mathbb R} \|-\eO_1 - c \eOe_1\|_s
\right\}\\
&= \inf_{c \in \mathbb R} \|\eO_1 - c \eOe_1\|_s
= \left\|\eO_1 - \dfrac{[\eO_1,\eOe_1]_s}{\lOe_1} \eOe_1 \right\|_s
= \sqrt{ \lO_1 - \dfrac{[\eO_1, \eOe_1]_s^2}{\lOe_1} },
\end{align*}
which implies
\begin{align*}
\left\|\dfrac{\eO_1}{\sqrt{\lO_1}} - \dfrac{\eOe_1}{\sqrt{\lOe_1}}\right\|_s^2
= 2 \left(1 - \left[\dfrac{\eO_1}{\sqrt{\lO_1}}, \dfrac{\eOe_1}{\sqrt{\lOe_1}} \right]_s \right) 
\leq 2 \left(1 - \left[\dfrac{\eO_1}{\sqrt{\lO_1}}, \dfrac{\eOe_1}{\sqrt{\lOe_1}} \right]_s^2 \right)
= \dfrac{2}{\lO_1} \eHs(N_{1,1}^a, N_{1,1}^b)^2.
\end{align*}
By combining the above equality with \eqref{ie1} we arrive at
$$
 \left\|\dfrac{\eO_1}{\sqrt{\lO_1}} - \dfrac{\eOe_1}{\sqrt{\lOe_1}}\right\|_s
 \leq \dfrac{C}{\sqrt{\lO_1}} \max \left\{ \delta^{-1}, \lO_1 \right\} \dd(\Ob,\Oa)^{s/2}
 \leq \dfrac{C}{\sqrt{\lambda_1(D)}} \max \left\{ \delta^{-1}, \lambda_1(B_r) \right\} \dd(\Ob,\Oa)^{s/2}.
 $$
On the other hand, by switching $a$ and $b$, we obtain
 $$
 \left\|\dfrac{\eOe_1}{\sqrt{\lOe_1}} - \dfrac{\eO_1}{\sqrt{\lO_1}}\right\|_s
 \leq \dfrac{C}{\sqrt{\lambda_1(D)}} \max \left\{ \delta^{-1}, \lambda_1(B_r) \right\} \dd(\Oa,\Ob)^{s/2}.
 $$
 By taking the minimum of the right-hand side of both the previous inequalities, we eventually establish \eqref{ef-result}.
\end{proof}

%%%%%%%%%%%%%%%%%%%%%%%%%%%%%%%%%%%%%%%%%%%%%%%%%%%%%%%%%%%%%%%%%%%%%%%%%%%%%%%%%%%%%%%%%%%%%

%%%%%%%%%%%%%%%%%%%%%%%%%%%%%%%%%%%%%%%%%%%%%%%%%%%%%%%%%%%%%%%%%%%%%%%%%%%%%%%%%%%%%%%%%%%%%
%
%   GNAMPA
%
%%%%%%%%%%%%%%%%%%%%%%%%%%%%%%%%%%%%%%%%%%%%%%%%%%%%%%%%%%%%%%%%%%%%%%%%%%%%%%%%%%%%%%%%%%%%%

\section*{Acknowledgements}

All authors are supported by the JSPS-CNR bilateral joint research project ``VarEvol: Innovative Variational Methods for Evolution Equations". 
GA is supported by JSPS KAKENHI Grant Number~16H03946, by the Alexander
von Humboldt Foundation and by the Carl Friedrich von Siemens Foundation. 
AS and GS are supported by the MIUR-PRIN Grant 2010A2TFX2 
``Calculus of Variations'' and LVS is supported by the MIUR-PRIN Grant ``Nonlinear Hyperbolic Partial Differential Equations, 
Dispersive and Transport Equations: theoretical and applicative aspects". 
AS, GS and LVS are also member of the GNAMPA (Gruppo Nazionale per l'Analisi Matematica, la Probabilit\`a 
e le loro Applicazioni) group of INdAM (Istituto Nazionale di Alta Matematica). 
%

%%%%%%%%%%%%%%%%%%%%%%%%%%%%%%%%%%%%%%%%%%%%%%%%%%%%%%%%%%%%%%%%%%%%%%%%%%%%%%%%%%%%%%%%%%%%%
%
%   BIBLIOGRAPHY
%
%%%%%%%%%%%%%%%%%%%%%%%%%%%%%%%%%%%%%%%%%%%%%%%%%%%%%%%%%%%%%%%%%%%%%%%%%%%%%%%%%%%%%%%%%%%%%

%\bibliographystyle{plain}
%\bibliography{eigen}

%%%%%%%%%%%%%%%%%%%%%%%%%%%%%%%%%%%%%%%%%%%%%%%%%%%%%%%%%%%%%%%%%%%%%%%%%%%%%%%%%%%%%%%%%%%%%

\end{document}